\documentclass[11pt]{amsart}
\usepackage{amsmath, amsthm, amscd, amssymb, enumerate, verbatim, graphicx}
\usepackage[all]{xy}
\usepackage{xcolor}

\newtheorem{thm}{Theorem}[section]
\newtheorem{prop}{Proposition}[section]
\newtheorem{cor}{Corollary}[section]
\newtheorem{lem}{Lemma}[section]

\theoremstyle{definition}

\newtheorem{remark}{Remark}[section]
\newtheorem{claim}{Claim}[section]
\newtheorem{defn}{Definition}[section]
\newtheorem{notation}{Notation}[section]
\newtheorem{example}{Example}[section]
\newtheorem*{problem}{Problem}
\newtheorem{fact}{Fact}[section]
\newtheorem*{property}{Property}
 
\newtheorem*{observation}{Observation}
\newtheorem*{assumption_ast}{Assumption $(\ast)$}
\newtheorem*{assumption_sharp}{Assumption $(\sharp)$}

\def \bthm {\begin{thm}}
\def \bthmb {\begin{thm} \mbox{}}
\def \ethm {\end{thm}} 

\def \bprop {\begin{prop}}
\def \bpropb {\begin{prop} \mbox{}}
\def \eprop {\end{prop}} 

\def \blem {\begin{lem}}
\def \blemb {\begin{lem} \mbox{}}
\def \elem {\end{lem}} 

\def \bcor {\begin{cor}}
\def \bcorb {\begin{cor} \mbox{}}
\def \ecor {\end{cor}}

\def \bdefn {\begin{defn}} 
\def \bdefnb {\begin{defn} \mbox{}}
\def \edefn {\end{defn}} 

\def \bfact {\begin{fact}} 
\def \bfactb {\begin{fact} \mbox{}}
\def \efact {\end{fact}} 

\def \bnot {\begin{notation}} 
\def \bnotb {\begin{notation} \mbox{}}
\def \enot {\end{notation}} 

\def \brem {\begin{remark}} 
\def \bremb {\begin{remark} \mbox{}} 
\def \erem {\end{remark}} 

\def \bpty {\begin{property}} 
\def \bptyb {\begin{property} \mbox{}} 
\def \epty {\end{property}} 

\def \bexp {\begin{example}} 
\def \bexpb {\begin{example} \mbox{}} 
\def \eexp {\end{example}} 

\def \bexp {\begin{example}} 
\def \bexpb {\begin{example} \mbox{}} 
\def \eexp {\end{example}} 

\def \bcl {\begin{claim}}
\def \bclb {\begin{claim} \mbox{}} 
\def \ecl {\end{claim}} 

\def \bobs {\begin{observation}}
\def \bobsb {\begin{observation} \mbox{}} 
\def \eobs {\end{observation}} 

\def \btab {\begin{tabular}}
\def \btabb {\begin{tabular} \mbox{}} 
\def \etab {\end{tabular}} 

\def \bary {\begin{array}}
\def \baryb {\begin{array} \mbox{}} 
\def \eary {\end{array}} 

\def \barr {\begin{array}}
\def \barrb {\begin{array} \mbox{}} 
\def \earr {\end{array}} 

\def \bpf {\begin{proof}}
\def \bpfb {\begin{proof} \mbox{}} 
\def \epf {\end{proof}} 

\def \benum {\begin{enumerate}}
\def \eenum {\end{enumerate}} 

\def \bit {\begin{itemize}}
\def \eit {\end{itemize}} 

\def \balign {\begin{align*}}
\def \ealign {\end{align*}}

\def \bCD {\begin{CD}}
\def \eCD {\end{CD}}

\def \edc {\end{document}} 

\def \beq {\begin{equation}}
\def \eeq {\end{equation}}

\newcommand{\bmp}[1]{\begin{minipage}[t]{#1} \baselineskip 6mm}
\def \emp {\end{minipage}}

\def \bec {\mbox{($\because$) }}

\def \lmt {\longmapsto}

\def \e   {\varepsilon}

\def \cal {\mathcal}
\def \phi {\varphi}

\def \ds {\displaystyle}

\def \Llra {\Longleftrightarrow}

\def \lra {\longrightarrow}
\def \lla {\longleftarrow}
\def \LLRA {\ $\Llra$ \ }

\def \whcF \widehat{\cF}

\def \itemi {\item[(i)\,\mbox{}]}

\def \itemiii {\item[(iii)]}
\def \itemiv {\item[(iv)]}

\def \itemI {\item[(i)\ \mbox{}]}
\def \itemII {\item[(ii)\,\mbox{}]}

\def \itema {\item[(a)]} 
\def \itemb {\item[(b)]} 
\def \itemc {\item[(c)]}

\def \hsf {\hspace*{2mm}}
\def \hsh {\hspace*{5mm}}
\def \hsp {\hspace*{10mm}}
\def \hspp {\hspace*{20mm}}
\def \hsppp {\hspace*{30mm}}
\def \hspppp {\hspace*{40mm}}

\def \bprob {\begin{problem}}
\def \bprobb {\begin{problem} \mbox{}}
\def \eprob {\end{problem}}

\def \IR {{\Bbb R}}
\def \IC {{\Bbb C}}
\def \IZ {{\Bbb Z}}

\newcommand{\id}{\mathrm{id}}

\newcommand{\pr}{\mathrm{pr}}

\newcommand{\incto}{\ar@{^{(}->}}
\newcommand{\into}{\ar@{>->}}
\newcommand{\onto}{\ar@{->>}}

\def \e {\varepsilon}

\renewcommand{\Theta}{\varTheta}
\renewcommand{\Sigma}{\varSigma}
\renewcommand{\Phi}{\varPhi}
\renewcommand{\Psi}{\varPsi}
\renewcommand{\Gamma}{\varGamma}
\renewcommand{\rho}{\varrho}

\begin{document}
\baselineskip 7 mm


\title[]{Boundedness of bundle diffeomorphism groups over a circle}

\author[Kazuhiko Fukui]{Kazuhiko Fukui} 
\address{228-168  Nakamachi, Iwakura, Sakyo-ku, Kyoto 606-0025, Japan}
\email{fukui@cc.kyoto-su.ac.jp}

\author[Tatsuhiko Yagasaki]{Tatsuhiko Yagasaki}
\address{Faculty of Arts and Sciences, Kyoto Institute of Technology, Kyoto, 606-8585, Japan}
\email{yagasaki@kit.ac.jp}


\maketitle

\begin{abstract}{}
In this paper we study boundedness of bundle diffeomorphism groups over a circle. 
For a fiber bundle $\pi : M \to S^1$ with fiber $N$ and structure group $\Gamma$ and $r \in \IZ_{\geq 0} \cup \{ \infty \}$ 
we distinguish an integer $k = k(\pi, r) \in \IZ_{\geq 0}$ and construct a function $\widehat{\nu} : {{\rm Diff}^r_\pi(M)_0} \to \IR_k$. 
When $k \geq 1$, it is shown that the bundle diffeomorphism group ${{\rm Diff}^r_\pi(M)_0}$ is {bounded} 
and ${clb_\pi d\,{\rm Diff}^r_\pi(M)_0} \leq k+3$, 
if ${{\rm Diff}^r_{\rho, c}(E)_0}$ is perfect for the trivial fiber bundle $\rho : E \to \IR$ with fiber $N$ and structure group $\Gamma$. 
On the other hand, when $k = 0$, it is shown that $\widehat{\nu}$ is a unbounded quasimorphism, so that 
${{\rm Diff}^r_\pi(M)_0}$ is unbounded and not uniformly perfect. 
We also describe the integer $k$ in term of the attaching map $\phi$ for a mapping torus $\pi : M_\phi \to S^1$ and 
give some explicit examples of (un)bounded groups. 
\end{abstract}

\thispagestyle{empty}

\section{Introduction} 

The boundedness of the diffeomorphism group ${\rm Diff}^r(M)$ of a smooth manifold $M$ is studied by 
D.~Burago, S.~Ivanov and L.~Polterovich \cite{BIP} and T.~Tsuboi \cite{Ts2, Ts3, Ts4}, et al. 
The case of equivariant diffeomorphism groups under free Lie group actions is studied by 
K.~Fukui \cite{AF, Fu2}, J.~Lech, I.~Michalik and T.~Rybicki \cite{LMR} et al,  and 
the case of leaf-preserving diffeomorphism groups is studied in \cite{Fu1, Ry, Ts1}. 

In this paper we continue the study of the boundedness of bundle diffeomorphism groups. 
Suppose $\pi : M \to B$ is a $C^\infty$ fiber bundle with fiber $N$ and structure group $\Gamma < {\rm Diff}(N)$.
The group $\Gamma$ determines the corresponding structure on each fiber of $\pi$.  
Let $r \in \IZ_{\geq 0} \cup \{ \infty \}$. 
A $C^r$ bundle diffeomorphism of $\pi$ is a $C^r$ diffeomorphism $f$ of $M$ 
which preserves the family of fibers of $\pi$ together with the $\Gamma$-structure 
(i.e., $\pi f = \underline{f} \pi$ for some $\underline{f} \in {\rm Diff}^r(B)$) and 
the restriction of $f$ to each fiber of $\pi$ lies in $\Gamma$ when it is represented as a diffeomorphism of $N$ using local trivializations of $\pi$). 
(See Subsection 3.1.)
Let ${\rm Diff}^r_\pi(M)$ denote the group of $C^r$ bundle diffeomorphisms of $\pi$ and 
${\rm Diff}^r_\pi(M)_0$ denote the identity component of ${\rm Diff}^r_\pi(M)$
(i.e., ${\rm Diff}^r_\pi(M)_0$ consists of all $f \in {\rm Diff}^r_\pi(M)$ which is $C^r$ isotopic to $\id_M$ in ${\rm Diff}^r_\pi(M)$. 
There exists a natural group homomorphism $P : {\rm Diff}^r_\pi(M)_0 \to {\rm Diff}^r(B)_0$, $P(f) = \underline{f}$. 

Our aim is to study the boundedness of the group ${\rm Diff}^r_\pi(M)_0$. 
Especially we are concerned with $clb_\pi f$, the commutator length of $f \in {\rm Diff}^r_\pi(M)_0$ supported in balls in $B$,  
and $clb_\pi d \, {\rm Diff}^r_\pi(M)_0$, the diameter of ${\rm Diff}^r_\pi(M)_0$ with respect to $clb_\pi$. 
In this paper we study the case where the base space $B$ is a circle $S^1$ and 
show that the boundedness of ${\rm Diff}^r_\pi(M)_0$ is distinguished by an integer $k = k(\pi, r) \geq 0$ as in Theorems 1 and 2 below. 

Take a universal covering $\pi_{S^1} : \IR \to \IR/\IZ \cong S^1$ with $\IZ$ as the covering transformation group. 
Fix a distinguished point $p \in S^1$. 
Let ${\cal P}(S^1)$ denote the space of continuous paths in $S^1$. 
The winding number $\lambda : {\cal P}(S^1) \to \IR$ is defined as follows. 
For each $c \in {\cal P}(S^1)$ take any lift $\widetilde{c} \in {\cal P}(\IR)$ of $c$ (i.e., $\pi_{S^1}\,\widetilde{c} = c$) and 
define $\lambda(c) \in \IR$ by $\lambda(c) := \widetilde{c}(1) - \widetilde{c}(0)$. 
This quantity is independent of the choice of the lift $\widetilde{c}$. 

Suppose $\pi : M \to S^1$ is a fiber bundle with fiber $N$ and structure group $\Gamma$ and let $r \in \IZ_{\geq 0} \cup \{ \infty \}$. 
{Let ${\rm Isot}^r_\pi(M)_0$ denote the space of $C^r$ isotopies $F$ of $M$ with $F_0 = \id_M$ and $F_t \in {\rm Diff}^r_\pi(M)$ $(t \in I)$. 
Here,  $I \equiv [0,1]$. 
Then we have the surjective group homomorphism $R : {\rm Isot}^r_\pi(M)_0 \to {\rm Diff}^r_\pi(M)_0$, $R(F) = F_1$, which has the kernel 
${\rm Isot}^r_\pi(M)_{\id, \id} := \{ F \in {\rm Isot}^r_\pi(M)_0 \mid F_1 = \id_M \}$.} 
Each $F = (F_t)_{t \in I} \in {\rm Isot}^r_\pi(M)_0$ induces 
an isotopy $\underline{F} = (\underline{F_t})_{t \in I} \in {\rm Isot}^r(S^1)_0$ and a path $\underline{F}_p := \underline{F}(p, \ast) \in {\cal P}(S^1)$. 
In Section 4 we show that the function \\ 
\hspp $\nu : {\rm Isot}^r_\pi(M)_0 \to \IR$, $\nu(F) = \lambda(\underline{F}_p)$, \\
is a surjective quasimorphism. 
Furthermore, 
it restricts to a surjective group homomorphism \\
\hspp $\nu : {\rm Isot}^r_\pi(M)_{\id, \id} \to k\IZ$ \\
\ for a unique $k = k(\pi, r) \in \IZ_{\geq 0}$. 
Finally, we obtain a surjective map \\
\hspp $\widehat{\nu} : {\rm Diff}^r_\pi(M)_0 
\to \IR_k \equiv \IR/k\IZ$; 
\ $\widehat{\nu}(F_1) = [\nu(F)]$ \ $(F \in {\rm Isot}^r_\pi(M)_0)$. \\ 
It restricts to a surjective group homomorphism \hsf $\widehat{\nu}|_{{\rm Ker}\,P} : {\rm Ker}\,P \to \IZ_k \equiv \IZ/k\IZ$ 
and the latter induces the group isomorphism $(\widehat{\nu}|_{{\rm Ker}\,P})^\sim : ({\rm Ker}\,P)\big/({\rm Ker}\,P)_0 \cong \IZ_k$. 

\vskip 2mm 

Consider the following assumption on $(N, \Gamma, r)$. 

\begin{assumption_ast} 
Suppose $J$ is an open interval and $\varrho : L \to J$ is a trivial fiber bundle with fiber $N$ and structure group $\Gamma$. 
Then ${\rm Diff}^r_{\varrho,c}(L)_0$ is perfect. 
\end{assumption_ast}

{This assumption is satisfied, for example, in the following cases : (i) $\pi$ is a principal $G$ bundle, $G$ is a compact Lie group and 
$1 \leq r \leq \infty$, $r \neq 2$, or (ii) $\pi$ is a locally trivial bundle, $N$ is a closed manifold and $r = \infty$.  
See Subsection 3.4 for details.}

\bthm\label{prop_k_geq_1} Suppose $k = k(\pi, r) \geq 1$ and $(N, \Gamma, r)$ satisfies Assumption $(\ast)$. 

If $f \in {\rm Diff}^r_\pi(M)_0$ and $\widehat{\nu}(f) = [s] \in \IR_k$ $\big(s \in \big(-\frac{k}{2}, \frac{k}{2}\big]\big)$, 
then $clb_\pi f \leq 2[|s|]+ 3 \leq k+3$. 
\ethm

\bcor\label{cor_S^1_k_geq_1} \ Suppose $k = k(\pi, r) \geq 1$ and $(N, \Gamma, r)$ satisfies Assumption $(\ast)$.
\benum
\item[{\rm (1)}] $clb_\pi d \,{\rm Diff}^r_\pi(M)_0 \leq k + 3$. 
\item[{\rm (2)}] \hspace*{0.7mm}{\rm (i)} \ ${\rm Diff}_\pi^r(M)_0$ is uniformly simple relative to ${\rm Ker}\,P$.  \\
{\rm (ii)} \ ${\rm Diff}_\pi^r(M)_0$ is bounded. 
\eenum 
\ecor 

See {Subsection 2.3} for the definition and basic properties of the notion of uniform simplicity of a group relative to a normal subgroup. 
This property implies the boundedness of the group. 

When $k = 0$, we have $(\IR_k, \IZ_k) = (\IR, \IZ)$ and it is seen that 
the function $\widehat{\nu} : {\rm Diff}^r_\pi(M)_0 \lra \IR$ is a surjective quasimorphism and 
the restriction $\widehat{\nu} : {\rm Ker}\,P  \to \IZ$ is a surjective group homomorphism. 
This implies the following conclusion. 

\bthm\label{prop_k=0} If $k = 0$, then the group ${\rm Diff}^r_\pi(M)_0$ is {unbounded and not uniformly perfect.} 
\ethm

Note that ${\rm Diff}_\pi^r(M)_0/{\rm Ker}\,P \cong {\rm Diff}^r(S^1)_0$ is uniformly simple, but 
${\rm Diff}_\pi^r(M)_0$ is not uniformly simple relative to ${\rm Ker}\,P$, since ${\rm Diff}_\pi^r(M)_0$ is unbounded. 

Any fiber bundle over a circle can be represented as a mapping torus $\pi_\phi : M_\phi \to S^1$ for some $\phi \in \Gamma < {\rm Diff}^\infty(N)$. 
We describe the invariant $k = k(\pi_\phi, r)$ in term of the attaching map $\phi$. 
For example, if $N = T^2 \equiv\IR^2/\IZ^2$ (a torus) and 
$\phi \in {\rm Diff}^\infty(N)$ is induced by a matrix $A \in SL(2, \IZ)$ with $A^n \neq E_2$ $(n \in \IZ - \{ 0 \})$, 
then $k(\pi_{\phi}, r) = 0$. 
In the case of a principal $\Gamma$ bundle 
over a circle for a Lie group $\Gamma$, 
its attaching map is a left translation $\phi_a$ by some $a \in \Gamma$ and 
the integer $k = k(\pi_{\phi_a}, r)$ can be described in term of $a$. 
This observation leads to some examples for the groups of equivariant diffeomorphisms. 

In a succeding paper we study the case of fiber bundles with higher dimensional base manifolds \cite{FY}. 


\section{Basics on conjugation-invariant norms and quasimorphisms}

\subsection{Conjugation-invariant norms} \mbox{} 

First we recall basic facts on conjugation-invariant norms \cite{BIP, Ts1}. 
Suppose $\Gamma$ is a group with the unit element $e$. 
An extended conjugation-invariant norm on $\Gamma$ is a function $q : \Gamma\to [0,\infty]$ which satisfies the following conditions~: \ \ 
for any $g, h\in \Gamma$
\bit 
\itemi $q(g)=0$ iff $g=e$ \hsh 
(ii) $q(g^{-1})=q(g)$ \hsh 
(iii) $q(gh)\leq q(g)+ q(h)$ \hsh 
(iv) $q(hgh^{-1})= q(g)$. 
\eit 
A conjugation-invariant norm on $\Gamma$ is an extended conjugation-invariant norm on $\Gamma$ with values in $[0, \infty)$. 
(Below we abbreviate ``an (extended) conjugation-invariant norm'' to an (ext.) conj.-invariant norm.) 

Note that for any ext.~conj.-invariant norm $q$ on $\Gamma$ the inverse image  
$\Lambda := q^{-1}([0, \infty))$ is a normal subgroup of $\Gamma$ and $q|_\Lambda$ is a conj.-invariant norm on $\Lambda$. Conversely, 
if $\Lambda$ is a normal subgroup of $\Gamma$ and $q : \Lambda \to [0, \infty)$ is conj.-invariant norm on $\Lambda$, then 
its trivial extension $q : \Gamma \to [0, \infty]$ by $q = \infty$ on $\Gamma - \Lambda$ is an ext.~conj.-invariant norm on $G$. 
For any ext.~conj.-invariant norm $q$ on $\Gamma$, 
the $q$-diameter of a subset $A$ of $\Gamma$ is defined by \ $q\hspace{0.2mm}d\,A := \sup\,\{ q(g) \mid g \in A \}$.

A group $\Gamma$ is called \emph{bounded} if any conj.-invariant norm on $\Gamma$ is bounded 
(or equivalently, any bi-invariant metric on $\Gamma$ is bounded).

\begin{example}\label{exp_norm} (Basic construction)

Suppose $S$ is a subset of $\Gamma$. The normal subgroup of $\Gamma$ generated by $S$ is denoted by $N(S)$.  
If $S$ is symmetric ($S = S^{-1}$) and conjugation-invariant ($gSg^{-1} = S$ for any $g \in \Gamma$), then 
$N(S) = S^\infty := \bigcup_{k=0}^\infty S^k$ and 
the ext.~conj.-invariant norm \  
$q_{
\mbox{\tiny $(\Gamma,S)$}} : \Gamma \lra {\Bbb Z}_{\geq 0} \cup \{ \infty \}$ \ is defined by \\[2mm] 
\hspace*{20mm} 
$q_{
\mbox{\tiny $(\Gamma,S)$}
}(g) := 
\left\{ \hspace{-1mm}
\bary[c]{l}
\min \{ k \in {\Bbb Z}_{\geq 0} \mid \text{$g = g_1 \cdots g_k$ for some $g_1, \cdots, g_k \in S$}\}
\hsh (g \in N(S)), \\[2mm] 
\infty \hsh (g \in \Gamma - N(S)).
\eary \right.$ \\[2mm] 
Here, the empty product $(k=0)$ denotes the unit element $e$ in $\Gamma$ and $S^0 = \{ e \}$. 
\end{example}

\subsection{Commutator length and uniform perfectness} \mbox{}

\begin{example}\label{exp_norm} (Commutator length)  

The symbol $\Gamma^c$ denotes the set of commutators in $\Gamma$. 
Since $\Gamma^c$ is symmetric and conjugation-invariant in $\Gamma$, 
we have $[\Gamma, \Gamma] = N(\Gamma^c) = (\Gamma^c)^\infty$ and obtain the ext.~conj.-invariant norm \ 
$q_{(\Gamma,\Gamma^c)} : \Gamma \lra {\Bbb Z}_{\geq 0} \cup \{ \infty \}$. 
We denote $q_{(\Gamma,\Gamma^c)}$ by $cl_\Gamma$ and call it the commutator length of $\Gamma$. 
The diameters $q_{(\Gamma,\Gamma^c)}d\,\Gamma$ and $q_{(\Gamma,\Gamma^c)}d\,A$ $(A \subset \Gamma)$ are denoted by 
$cld\,\Gamma$ and $cld_\Gamma\,A \equiv cld\,(A, \Gamma)$ and called the commutator length diameter of $\Gamma$ and $A$ respectively. 
Sometimes we write $cl(g) \leq k$ in $\Gamma$ instead of $cl_\Gamma(g) \leq k$. 

A group $\Gamma$ is perfect if $\Gamma = [\Gamma, \Gamma]$, that is, any element of $\Gamma$ is written as a product of commutators in $\Gamma$.
We say that a group $\Gamma$ is uniformly perfect if $cld\,\Gamma < \infty$, that is, 
any element of $\Gamma$ is written as a product of a bounded number of commutators in $\Gamma$.

More generally, suppose $S$ is a subset of $\Gamma^c$ and it is symmetric and conjugation-invariant in $\Gamma$.  
Then, $N(S) \subset [\Gamma,\Gamma]$ and we obtain $q_{(\Gamma,S)}$, which is denoted by $cl_{(\Gamma, S)}$ 
and is called the commutator length of $\Gamma$ with respect to $S$. 
The diameters $q_{(\Gamma,S)}d\,\Gamma$ and $q_{(\Gamma,S)}d\,A$ $(A \subset \Gamma)$ are denoted by 
$cld_S\,\Gamma$ and $cld_{(\Gamma,S)}\,A \equiv cld_S\,(A, \Gamma)$ respectively. 
\end{example} 

Every bounded perfect group is uniformly perfect. 

For simplicity we use the notation $b^a := aba^{-1}$ for $a,b \in \Gamma$. 

\subsection{Conjugation-generated norm and relative uniform simplicity} \mbox{} 

Next we recall basic facts on conjugation-generated norms on groups \cite{BIP}. 
Suppose $\Gamma$ is a group. 
For $g \in \Gamma$ let $C(g)$ denote the conjugacy class of $g$ in $\Gamma$ 
and let $C_g := C(g) \cup C(g^{-1})$. 
Since $C_g$ is a symmetric and conjugation invariant subset of $\Gamma$, it follows that $N(g) = N(C_g) = \bigcup_{k\geq 0}(C_g)^k$ 
and we obtain the ext.~conj.-invariant norm $q_{\Gamma, C_g}$ on $\Gamma$. 
This norm is denoted by $\zeta_g$ and called the conjugation-generated norm with respect to $g$. 
Note that $C_h = C_g$ and $\zeta_h = \zeta_g$ for any $h \in C_g$. 

The norms $\zeta_g$ $(g \in \Gamma)$ are used to define the notion of uniform simplicity of a group $\Gamma$. 
Let $\Gamma^\times := \Gamma - \{ e \}$. Consider the quantity \\[1.5mm] 
\hspace*{20mm} 
$\zeta(\Gamma) := \min \{ k \in \IZ_{\geq 0} \cup \{ \infty \} \mid \zeta_g(f) \leq k \text{ for any $g \in \Gamma^\times$ and $f \in \Gamma$}\}.$ \\[0.5mm] 
A group $\Gamma$ is called \emph{uniformly simple} (cf. \cite{Ts3}) if  $\zeta(\Gamma) < \infty$, that is, there is $k \in \IZ_{\geq 0}$ such that for any $f \in \Gamma$ and $g \in \Gamma^\times$, $f$ can be expressed as a product of at most $k$ conjugates of $g$ or $g^{-1}$. 

More generally, for any normal subgroup ${\Lambda}$ of $\Gamma$, we can consider the quantity \\
\hspp $\zeta(\Gamma; \Lambda) := \min \{ k \in \IZ_{\geq 0} \cup \{ \infty \} \mid \zeta_g(f) \leq k \text{ for any $g \in \Gamma-\Lambda$ and $f \in \Gamma$}\}.$ \\[0.5mm] 
We say that a group $\Gamma$ is \emph{uniformly simple} relative to ${\Lambda}$ if  $\zeta(\Gamma; \Lambda) < \infty$, that is, there is $k \in \IZ_{\geq 0}$ such that for any $f \in \Gamma$ and $g \in \Gamma - \Lambda$, $f$ can be expressed as a product of at most $k$ conjugates of $g$ or $g^{-1}$. 

\begin{fact}\label{fact_unif-simple_bounded} \mbox{} 
\benum 
\item If $\zeta_g$ is bounded for some $g \in \Gamma^\times$, then $\Gamma$ is bounded. More precisely, 
if $g \in \Gamma^\times$ and $\zeta_g \leq k$ for some $k \in \IZ_{\geq 0}$, then 
$q \leq k q(g)$ for any ext.~conj.-invariant norm $q$ on $\Gamma$.  
\item $\Gamma$ is uniformly simple iff $\zeta_g$ $(g \in \Gamma^\times)$ are uniformly bounded. 
\item[] $\Gamma$ is uniformly simple relative to $\Lambda$ iff $\zeta_g$ $(g \in \Gamma - \Lambda)$ are uniformly bounded. 
\item $\zeta(\Gamma;\Lambda) \geq \zeta(\Gamma/\Lambda)$. 
\item If $\Gamma$ is uniformly simple, then $\Gamma$ is simple. \\
If $\Gamma$ is uniformly simple relative to $\Lambda$, then 
\bit 
\itemI if $L$ is a normal subgroup of $\Gamma$, then $L \subset \Lambda$ or $L = \Gamma$, 
\itemII $\Gamma/\Lambda$ is uniformly simple, 
so $\Gamma/\Lambda$ is simple ({equivalently}, $\Lambda \subset L \vartriangleleft \Gamma$, then $L = \Lambda$ or $\Gamma$.)
\eit 
\eenum 
\end{fact}

\subsection{Quasimorphisms} \mbox{} 

Suppose $\Gamma$ is a group. A function $\phi : \Gamma \to \IR$ is said to be a quasimorphism if \\
\hspace*{20mm} $\ds D_\phi := \sup_{a,b \in \Gamma} |\phi(ab) - \phi(a) - \phi(b)| < \infty$. \\[1mm] 
The quantity $D_\phi$ is called the defect of $\phi$. A quasimorphism $\phi$ on $\Gamma$ is said to be homogeneous if 
$\phi(a^n) = n\phi(a)$ for any $a \in \Gamma$ and $n \in \IZ$. 

\bfact\label{fact_qm} \mbox{} (cf.\,\cite{Fuj, GG})
\benum
\item Suppose $\phi$ is a quasimorphism on $\Gamma$. 
\bit 
\itemI $|\phi(ab)| \leq |\phi(a) + \phi(b)| + D_\phi \leq |\phi(a)| + |\phi(b)| + D_\phi$ \ for any $a,b \in \Gamma$. 
\vskip 1mm 
\itemII The function \ \ 
$\ds \overline{\phi} : \Gamma \to \IR$ : \ $\ds \overline{\phi}(a) = \lim_{n \to \infty} \frac{\phi(a^n)}{\,n\,}$ \\[1mm] 
is a well-defined homogeneous quasimorphism, which has the following properties. 

\bit 
\itema $|\phi(a) - \overline{\phi}(a)| \leq D_\phi$ \ \ $(a \in \Gamma)$ 
\hsp (b) \ $D_{\overline{\phi}} \leq 4D_\phi$ 
\eit 
\eit 
\item Suppose $\phi$ is a homogeneous quasimorphism on $\Gamma$. Then, it satisfies the following properties. 
\bit 
\item[(0)\,] If $\phi$ is bounded, then it is trivial. 
\itemI $\phi$ is conjugation-invariant \ (i.e., $\phi(aba^{-1}) = \phi(b)$ \ for any $a,b \in \Gamma$). 
\itemII (a) $\ds \sup_{a,b \in \Gamma} |\phi([a,b])| = D_\phi$. \hsh (b) 
If $a = a_1 \cdots a_k$ \ ($a_1, \cdots, a_k \in \Gamma^c$), then $|\phi(a)| \leq (2k-1) D_\phi$ 
\vskip 1mm 
\itemiii If $\ell := cld \,\Gamma < \infty$, then $|\phi(a)| \leq (2\ell-1) D_\phi$ \ for any $a \in \Gamma$. 
\itemiv Take any $D > 0$ with $D \geq D_\phi$ and define the function $q : \Gamma \to [0, \infty)$ by \\[1mm] 
\hspace*{20mm} $q(a) := \left\{ \hspace{-1mm} 
\bary[c]{ll}
|\phi(a)| + D & (a \in \Gamma^\times) \\[2mm]
0 & (a = e \ \text{(the unit element of $\Gamma$)})
\eary \right.$. \\[2mm]
(a) $q$ is a conjugation-invariant norm on $\Gamma$. \\
(b) $q$ is bounded iff $\phi$ is bounded. 
\eit 
\item The following conclusions follow from (1)\,(ii) and (2)\,(iii),(iv). 
\bit 
\itemI 
 If $\Gamma$ is uniformly perfect, then any 
quasimorphism on $\Gamma$ is bounded. \\
(If $\Gamma$ admits a unbounded 
quasimorphism, then $\Gamma$ is not uniformly perfect.) 
\itemII If $\Gamma$ admits a unbounded  
quasimorphism, then $\Gamma$ is not bounded.
\eit 
\eenum
\efact

\section{Generality on bundle diffeomorphism groups} 

\subsection{Groups of diffeomorphisms and isotopies} \mbox{} 

For a topological space $X$ and its subsets $A$, $B$ 
the symbols ${\rm Int}_XA$ and $Cl_XA$ denote the topological interior and closure of $A$ in $X$ and 
the symbol $A \Subset B$ means $A \subset {\rm Int}_XB$. We set $X_A := X - {\rm Int}_X A$.
The symbols ${\cal O}(X)$, ${\cal F}(X)$ and ${\cal K}(X)$ denote 
the collections of open subsets, closed subsets and compact subsets of $X$ respectively. 
We always use the symbol $I$ to denote the interval $[0,1]$. 
The {symbol} ${\rm pr}_X : X \times Y \to X$ denotes the projection onto $X$. 

In this paper an $n$-manifold means a separable metrizable $C^\infty$ manifold possibly with boundary of dimension $n$. 
A closed manifold means a compact manifold without boundary.  

Suppose $M$ and $N$ are $C^\infty$ manifolds and $r \in \IZ_{\geq 0} \cup \{ \infty \}$.  We use the symbols \\
\hspace*{20mm} $C^r(M, N) \supset {\rm Emb}^r(M, N) \supset {\rm Diff}^r(M, N)$ \\ 
to denote 
the spaces of $C^r$ maps, $C^r$ embeddings and $C^r$ diffeomorphisms from $M$ to $N$, while the symbols \\
\hspace*{20mm} $H^r(M, N) \supset {\rm EIsot}^r(M, N) \supset {\rm Isot}^r(M, N)$ \\ 
denote the spaces of $C^r$ homotopies of maps, embeddings and diffeomorphisms from $M$ to $N$ respectively. 
Each element of ${\rm Isot}^r(M, N)$ is called a $C^r$ isotopy, 
while any element of ${\rm EIsot}^r(M, N)$ is called a non-ambient $C^r$ isotopy.  
For ${\cal E} \subset H^r(M, N)$ and $f,g \in C^r(M, N)$ let ${\cal E}_{f} := \{ F \in {\cal E} \mid F_0 = f \}$ and 
${\cal E}_{f,g} := \{ F \in {\cal E} \mid F_0 = f, F_1 = g \}$. 

For a $C^r$ homotopy $F = (F_t)_{t \in I} : M \times I \to N$
its support is defined by ${\rm supp}\,F := Cl_M \big(\bigcup_{t \in I} {\rm supp}\,F_t\big)$. 
Also we have the associated $C^r$ map $F' : M \times I \to N \times I$, $F'(x,t) = (F(x,t), t)$. 
For the simplicity of notation we use the same symbol $F$ to denote the map $F'$. 
This convention implies that \\
${\rm EIsot}^r(M, N) = H^r(M, N) \cap {\rm Emb}^r(M \times I, N \times I)$ 
and ${\rm Isot}^r(M, N) = H^r(M, N) \cap {\rm Diff}^r(M \times I, N \times I)$. 

{ In the next section we apply some arguments which deform some isotopies to those which satisfy some required conditions. 
Here we have some notational remarks on homotopies between isotopies. Essentially we follow the notations for path-homotopies. 
For ${\cal E} \subset H^r(M, N)$ and $G, H \in {\cal E}$, the symbol $G \simeq H$ ($G \simeq_\ast H$) means that 
$G$ and $H$ are $C^r$ homotopic in ${\cal E}$ (relative to ends), that is, there exists \\ 
\hspace*{10mm} a $C^r$ homotopy \ \ $\Phi = (\Phi_t )_{t \in I} : (M \times I) \times I \lra N \times I$ \ \ such that \\
\hspace*{20mm} $\Phi_0 = G$, \ \ $\Phi_1 = H$, \ \ $\Phi_t \in {\cal E}$ \ $(t \in I)$ \ \ 
(and \ $(\Phi_t)_0 = G_0 = H_0$, \ $(\Phi_t)_1 = G_1 = H_1$). \\
In the case where ${\cal E} = {\rm Isot}^r(M, N)$, the homotopy $\Phi$ is called a $C^r$ isotopy from $G$ to $H$ (rel ends).}

As usual, when $M = N$, we omit the symbol $N$ from the notations 
$C^r(M, N)$, ${\rm Diff}^r(M, N)$ and $H^r(M, N)$, ${\rm Isot}^r(M, N)$. 
For ${\cal E} \subset H^r(M)$ let ${\cal E}_0 := \{ F \in {\cal E} \mid F_0 = \id_M \}$. 

Note that ${\rm Diff}^r(M)$ is a group under the composition. 
For $F, G \in H^r(M)$ their composition $FG \in H^r(M)$ is defined by {$FG = (F_tG_t)_{t \in I}$}. 
Then, ${\rm Isot}^r(M)$ is a group under this composition. 

For a subset $C$ of $M$ we have the following subgroups of these groups:  \\[1mm] 
\hspace*{6mm} ${\rm Isot}^r(M; C)_0 = \{ F \in {\rm Isot}^r(M) \mid F_0 = \id_M, \mbox{$F = \id$ on $U \times I$ for some neighborhood $U$ of $C$ in $M$} \}$, \\[0.5mm] 
\hspace*{6mm} ${\rm Diff}^r(M; C)_0 = \{ F_1 \mid F \in {\rm Isot}^r(M; C)_0 \}$. \\
{In} the above notations we usually omit the superscript $r$ when $r = \infty$ and the symbol $C$ when $C = \emptyset$. 

\subsection{Generalities on fiber bundles} \mbox{}

In this {subsection} we review basics 
on fiber bundles.  
Suppose $\pi : M \to B$ is a $C^\infty$ locally trivial bundle with fiber $N$. { (In this paper we assume that $N$ has no boundary.)} 
For notational simplicity, for $q \in B$ the fiber $\pi^{-1}(q)$ is denoted by $M_q$ 
and for any subset $C \subset B$ let $\widetilde{C} := \pi^{-1}(C)$. 
A local trivialization of $\pi$ is a pair $(U, \phi)$ of an open subset $U$ of $B$ and 
a $C^\infty$ diffeomorphism $\phi : \widetilde{U} \to U \times N$ with ${\rm pr}_U \phi = \pi|_{\widetilde{U}}$. 
Note that $\phi$ restricts to a $C^\infty$ diffeomorphism $\phi_q : M_q \to N$ for each $q \in U$.  
For local trivializations $(U, \phi)$ and $(V, \psi)$ of $\pi$, 
we have the transition function $\eta = \eta(\phi, \psi) : U \cap V \to {\rm Diff}(N)$ defined by $\eta_q = \psi_q \phi_q^{\ -1}$ $(q \in U \cap V)$.  
It is natural to say that a map $\phi : L \to {\rm Diff}(N)$ from a $C^\infty$ manifold $L$ is $C^\infty$ 
if the induced map $\widehat{\phi} : L \times N \to N$, $\widehat{\phi}(u,x) = \phi(u)(x)$, is $C^\infty$. 
The transition function $\eta$ is $C^\infty$ in this sense. 
An atlas of $\pi$ is a family ${\cal U} = \{ (U_\lambda, \phi_\lambda) \}_{\lambda \in \Lambda}$ of local trivializations of $\pi$ with 
$B = \cup_{\lambda \in \Lambda} U_\lambda$, 
which induces the family of transition functions, $\eta_{\cal U} = \{ \eta_{\lambda, \mu} \equiv \eta(\phi_\lambda, \phi_\mu) \}_{(\lambda, \mu) \in \Lambda^2}$. 
This family satisfies the transition condition in ${\rm Diff}(N)$: \\
\hspace*{3mm} $(\ast)$ \ $\eta_{\mu, \nu}(q) \eta_{\lambda, \mu}(q) = \eta_{\lambda, \nu}(q)$ \hsf $(\lambda, \mu, \nu \in \Lambda, \ q \in U_\lambda \cap U _\mu \cap U_\nu)$ \hsf and \hsf  
$\eta_{\lambda, \lambda}(q) = \id_N$ \hsf $(\lambda \in \Lambda, \ q \in U_\lambda)$. \\
This condition implies that $\eta_{\mu, \lambda}(q) = \big(\eta_{\lambda, \mu}(q)\big)^{-1}$ \hsf $(\lambda, \mu \in \Lambda, \ q \in U_\lambda \cap U_\mu)$. 

We incorpolate the notion of structure group to locally trivial bundles. 
We restrict ourselves to the simple case where $\Gamma$ is a subgroup of ${\rm Diff}(N)$.
In this case, we say that a map $\phi : L \to \Gamma$ from a $C^\infty$ manifold $L$ is $C^\infty$ if 
the associated map $\phi : L \to {\rm Diff}(N)$ is $C^\infty$.
A $\Gamma$-atlas of $\pi$ is an atlas ${\cal U} = \{ (U_\lambda, \phi_\lambda) \}_{\lambda \in \Lambda}$ of $\pi$ such that 
each transition function in $\eta_{\cal U}$ has its image in $\Gamma$, that is, 
$\eta_{\lambda, \mu}(q) = (\phi_\mu)_q(\phi_\lambda)_q^{\ -1} \in \Gamma$ for each $\lambda, \mu \in \Lambda$ and $q \in U_\lambda \cap U_\mu$. 
Note that each $\eta_{\lambda, \mu} : U_\lambda \cap U_\mu \to \Gamma$ is $C^\infty$ in the above sense and 
that the family $\eta_{\cal U} = \{ \eta_{\lambda, \mu} \}$ satisfies the transition condition $(\ast)$. 
 
A fiber bundle with fiber $N$ and structure group $\Gamma$ (or {an} $(N, \Gamma)$-bundle) means  
a locally trivial bundle $\pi : M \to B$ with fiber $N$ provided with a $\Gamma$-atlas ${\cal U}$ which is maximal with respect to the inclusion relation. 
Each element of ${\cal U}$ is called a local trivialization of $\pi$. 
For each $q \in B$, the family $\Gamma_{\pi, q} := \{ \phi_q \in {\rm Diff}^r(M_q, N) \mid (U, \phi) \in {\cal U}, q \in U \}$ 
determines a $\Gamma$-structure on $M_q$. 
When $C$ is a $C^\infty$ submanifold of $B$, the restriction $\pi_C := \pi|_{\widetilde{C}} : \widetilde{C} \to C$ 
becomes canonically {an} $(N, \Gamma)$-bundle (together with a $\Gamma$-atlas ${\cal U}|_C$). 
{ In this formulation, for a Lie group $\Gamma$, 
the standard definition of principal $\Gamma$ bundles 
can be interpreted as $(\Gamma, \Gamma_L)$-bundles, where \\
\hsp $\Gamma_L := \{ \phi_a \mid a \in \Gamma \} < {\rm Diff}^\infty(\Gamma)$ and 
$\phi_a$ is the left translation on $\Gamma$ by $a$. }

Next we review basic definitions on $C^r$ maps between fiber bundles for $r \in \IZ_{\geq 0} \cup \{ \infty \}$. 
Suppose $\pi : M \to B$ and $\pi' : M' \to B'$
are $(N, \Gamma)$-bundles (together with a maximal $\Gamma$-atlas ${\cal U}$ and ${\cal U}'$ respectively). 
A $C^r$ map $f : \pi \to \pi'$ means 
a $C^r$ map $f \in C^r(M, M')$ such that 
\bit 
\itemI $\pi' f = \underline{f}\,\pi$ for some map $\underline{f} : B \to B'$ and 
\item[(ii)$'$] $\psi_{\underline{f}(q)} f_q (\phi_q)^{-1} \in \Gamma$ 
for any $(U, \phi) \in {\cal U}$, $(V, \psi) \in {\cal U}'$ and $q \in U \cap \underline{f}^{-1}(V)$. 
\eit 
Here, the symbol $f_q : M_q \to M'_{\underline{f}(q)}$ denotes the restriction of $f$ obtained by the condition (i).
Note that $\underline{f}$ is uniquely determined by $f$ and $\underline{f} \in C^r(B, B')$ 
and that $f_q \in {\rm Diff}^r(M_q, M'_{\underline{f}(q)})$ in the condition (ii).

Let $C^r(\pi, \pi')$ denote the subset of $C^r(M, M')$ consisting of all $C^r$ maps from $\pi$ to $\pi'$. 
It includes the following subsets: 
\hsh $C^r(\pi, \pi') \supset {\rm Emb}^r(\pi, \pi') \supset {\rm Diff}^r(\pi, \pi')$, \\
where \hsf 
${\rm Emb}^r(\pi, \pi') := C^r(\pi, \pi') \cap {\rm Emb}^r(M, M')$ \hsf and \hsf 
${\rm Diff}^r(\pi, \pi') := C^r(\pi, \pi') \cap {\rm Diff}^r(M, M')$. \\
We call each element of ${\rm Diff}^r(\pi, \pi')$ a $C^r$ diffeomorphism from $\pi$ onto $\pi'$. 
{Since each $f \in C^r(\pi, \pi')$ is a fiberwise $C^r$ diffeomorphism, it follows that 
$f \in {\rm Emb}^r(\pi, \pi')$ iff $\underline{f} \in {\rm Emb}^r(B, B')$ and 
$f \in {\rm Diff}^r(\pi, \pi')$ iff $\underline{f} \in {\rm Diff}^r(B, B')$.}

A $C^r$ homotopy from $\pi$ to $\pi'$ means a $C^r$ homotopy $F \in H^r(M, M')$ with $F_t \in C^r(\pi, \pi')$ $(t \in I)$. 
Let $H^r(\pi, \pi')$ denote the set of $C^r$ homotopies from $\pi$ to $\pi'$. It includes the following subsets: \\ 
\hsp $H^r(\pi, \pi') \supset {\rm EIsot}^r(\pi, \pi') \supset {\rm Isot}^r(\pi, \pi')$, \\
where \hsf 
${\rm EIsot}^r(\pi, \pi') := H^r(\pi, \pi') \cap {\rm EIsot}^r(M, M')$ \hsf and \hsf 
${\rm Isot}^r(\pi, \pi') := H^r(\pi, \pi') \cap {\rm Isot}^r(M, M')$. \\
We may call each element of ${\rm EIsot}^r(\pi, \pi')$ a non-ambient $C^r$ isotopy from $\pi$ to $\pi'$. 
Note that each $F \in H^r(\pi, \pi')$ yields $\underline{F} := \{ \underline{F_t} \}_{t \in I} \in H^r(B, B')$ 
and that $\underline{F} \in {\rm EIsot}^r(B, B')$ if $F \in {\rm EIsot}^r(\pi, \pi')$ and 
$\underline{F} \in {\rm Isot}^r(B, B')$ if $F \in {\rm Isot}^r(\pi, \pi')$. 

{ For homotopies between $C^r$ homotopies from $\pi$ to $\pi'$ we can apply the general notations introduced in \S3.1.  
For instance, for ${\cal E} \subset H^r(\pi, \pi')$  and $G, H \in {\cal E}$, the symbol $\Phi : G \simeq_\ast H$ means that 
$\Phi = (\Phi_t )_{t \in I} : (M \times I) \times I \lra M' \times I$ is a $C^r$ homotopy from $G$ and $H$ in ${\cal E}$ relative to ends. 
In the case where ${\cal E} = {\rm Isot}^r(\pi, \pi')$, the homotopy $\Phi$ is called a $C^r$ isotopy from $G$ to $H$ rel ends.}

{ Since we are mainly concerned with an ambient fiber bundle and its restrictions, 
below we use more familiar notations which emphasize the total space of the ambient fiber bundle. 
Suppose $\pi : M \to B$ is a fiber bundle with fiber $N$ and structure group $\Gamma$ and let $r \in \IZ_{\geq 0} \cup \{ \infty \}$. 
We obtain the subgroups \\
\hspace*{20mm} ${\rm Diff}^r_\pi(M) := {\rm Diff}^r(\pi, \pi) < {\rm Diff}^r(M)$ \hsh and \hsh 
${\rm Isot}^r_\pi(M) := {\rm Isot}^r(\pi, \pi) < {\rm Isot}^r(M)$. \\
Note that $F_t \in {\rm Diff}^r_\pi(M)$ $(t \in I)$ for $F \in {\rm Isot}^r_\pi(M)$. 
So we can consider the subgoups \\
\hspace*{20mm} ${\rm Isot}^r_\pi(M)_0 := {\rm Isot}^r_\pi(M) \cap {\rm Isot}^r(M)_0$
and ${\rm Diff}^r_\pi(M)_0 = \{ F_1 \mid F \in {\rm Isot}^r_\pi(M)_0 \}$. \\
The latter plays the role of the identity component of the group ${\rm Diff}^r_\pi(M)$. 
We say that $F \in {\rm Isot}^r_{\pi}(M)$ has compact support with respect to $\pi$ if 
${\rm supp}\,F \subset \pi^{-1}(K)$ for some compact subset $K$ of $B$. 
The symbol ${\rm Isot}^r_{\pi,c}(M)$ denotes the subgroup of ${\rm Isot}^r_{\pi}(M)$ consisting of $F \in {\rm Isot}^r_{\pi}(M)$ with compact support with respect to $\pi$. As before, let ${\rm Diff}^r_{\pi,c}(M)_0 := \{ F_1 \mid F \in {\rm Isot}^r_{\pi,c}(M)_0\}$.}

\subsection{Relative isotopy lifting lemma for fiber bundles} \mbox{} 

{Suppose $\pi : M \to B$ is a fiber bundle with fiber $N$ and structure group $\Gamma$ and let $r \in \IZ_{\geq 0} \cup \{ \infty \}$. 
We further introduce some subgroups of 
${\rm Diff}^r_\pi(M)_0$ relative to subsets of the base space $B$. 
For any subset $C \subset B$ and any $U \in {\cal O}(B)$ we define subgroups of ${\rm Isot}^r_{\pi}(M)_0$ by \\
\hspace*{10mm} ${\rm Isot}^r_{\pi}(M; C)_0 = \{ F \in {\rm Isot}^r_{\pi}(M)_0 \mid 
\text{$F_t|_{\widetilde{V}} = {\rm inc}_{\widetilde{V}}$ $(t \in I)$ for some neighborhood $V$ of $C$ in $B$} \}$, \\ 
\hspace*{10mm} ${\rm Isot}^r_{\pi, c}(M; C)_0 = {\rm Isot}^r_{\pi}(M; C)_0 \cap {\rm Isot}^r_{\pi, c}(M)$ \hsh and \\ 
\hspace*{10mm} ${\rm Isot}^r_{\pi,c}(\widetilde{U})_0 
= \{ F \in {\rm Isot}^r_{\pi,c}(M)_0 \mid \text{${\rm supp}\,F \subset \pi^{-1}(K)$ for some $K \in {\cal K}(U)$}\}$. \\ 
They determine the following subgroups of ${\rm Diff}^r_{\pi}(M)_0$, \\
\hspace*{10mm} ${\rm Diff}^r_{\pi}(M; C)_0 := \{ F_1 \mid F \in {\rm Isot}^r_{\pi}(M; C)_0 \}$, \ \  
${\rm Diff}^r_{\pi, c}(M; C)_0 := \{ F_1 \mid F \in {\rm Isot}^r_{\pi, c}(M; C)_0 \}$ \ \ and \\ 
\hspace*{10mm} ${\rm Diff}^r_{\pi,c}(\widetilde{U})_0 := \{ F_1 \mid F \in {\rm Isot}^r_{\pi,c}(\widetilde{U})_0\}$. \\ 
Each $F \in {\rm Isot}^r_\pi(M)$ induces the $C^r$ isotopy $\underline{F} = \{ \underline{F_t} \}_{t \in I} \in {\rm Isot}^r(B)$ and 
if $F \in {\rm Isot}^r_\pi(M; C)_0$, then $\underline{F} \in {\rm Isot}^r(B; C)_0$. 
It follows that $\underline{f} \in {\rm Diff}^r(B; C)_0$ for each $f \in {\rm Diff}^r_\pi(M; C)_0$. 
From the definition, for $B_C = B - {\rm Int}_BC$, it follows that 
$F \in {\rm Isot}_\pi^r(M; B_C)_0$ iff $F \in {\rm Isot}_\pi^r(M)_0$ and 
${\rm supp}\,F \subset \widetilde{D}$ for some $D \in {\cal F}(B)$ with $D \Subset C$ 
and that 
$f \in {\rm Diff}_\pi^r(M; B_C)_0$ means that $f = F_1$ for some $F \in {\rm Isot}_\pi^r(M; B_C)_0$.  
Note that $g({\rm Diff}_\pi^r(M; B_C)_0)g^{-1} = {\rm Diff}_\pi^r(M; B_{\underline{g}(C)})_0$ \ for any $g \in {\rm Diff}_\pi^r(M)_0$ 
and that for any $U \in {\cal O}(B)$ we have the restriction $\pi_U : \widetilde{U} \to U$ of $\pi$ and 
there exists a canonical isomorphism \\
\hsp \hsh ${\rm Diff}^r_{\pi_U, c}(\widetilde{U})_0 \cong {\rm Diff}^r_{\pi, c}(\widetilde{U})_0 = {\rm Diff}^r_{\pi, c}(M; B_U)_0$.}

When $C$ is a submanifold of $B$, 
the restriction $\pi_C \equiv \pi|_{\widetilde{C}} : \widetilde{C} \to C$ is also a fiber bundle with fiber $N$ and structure group $\Gamma$.
In this situation, we use the notations: \\
\hspace*{20mm} ${\rm Emb}^r_\pi(\widetilde{C}, M) := {\rm Emb}^r(\pi_C, \pi)$ \hsf and \hsf 
${\rm EIsot}^r_\pi(\widetilde{C}, M) := {\rm EIsot}^r(\pi_C, \pi)$. \\
As above, let 
\btab[t]{l}
${\rm EIsot}^r_\pi(\widetilde{C}, M)_0 = \{ F \in {\rm EIsot}^r_\pi(\widetilde{C}, M) \mid F_0 = {\rm inc}_{\widetilde{C}} \}$ \ and \\[1.5mm]
${\rm Emb}^r_\pi(\widetilde{C}, M)_0 = \{ F_1 \mid F \in {\rm EIsot}^r_\pi(\widetilde{C}, M)_0 \}$. 
\etab 

Each $F \in {\rm Isot}^r_\pi(M)_0$ induces the restriction $F_C := F|_{\widetilde{C} \times I} \in {\rm EIsot}^r_\pi(\widetilde{C}, M)_0$ and 
each $f \in {\rm Diff}^r_\pi(M)_0$ induces the restriction $f_C := f|_{\widetilde{C}} \in {\rm Emb}^r_\pi(\widetilde{C}, M)_0$. 

Recall that for any set ${\cal E}$ of homotopies the symbol ${\cal E}_{f}$ denotes the subset ${\cal E}_{f} \equiv \{ F \in {\cal E} \mid F_0 = f \}$.

\begin{prop}\label{rel_iso_lift} $($Relative isotopy lifting property$)$ \\ 
Suppose $\pi : M \to B$ is an $(N, \Gamma)$-bundle,  
$r \in \IZ_{\geq 0} \cup \{ \infty \}$, 
$f \in {\rm Diff}^r_\pi(M)$,  
$G \in {\rm Isot}^r(B)_{\underline{f}}$ and $F' \in {\rm EIsot}^r_\pi(\widetilde{U}, M)_{f|_{\widetilde{U}}}$ with $\underline{F'} = G|_{U \times I}$, 
where $K \in {\cal F}(B)$ and $U$ is an open neighborhood of $K$ in $B$. 
Then, there exists $F \in {\rm Isot}^r_\pi(M)_f$ such that $\underline{F} = G$ and $F = F'$ on $\widetilde{K} \times I$. 
When $K$ is a submanifold of $B$, we do not need to take a neighborhood $U$ of $K$. 

\end{prop} 


Note that the statement of Proposition~\ref{rel_iso_lift} is readily reduced to the standard case where $f = \id_M$ 
by replacing $G$ and $F'$ with $(\underline{f}^{-1} \times \id_I)G$ and $(f^{-1} \times \id_I)F'$ respectively. 

\begin{cor}\label{cor 1 rel_iso_lift} Suppose $K, L \in {\cal F}(B)$, 
$F \in {\rm Isot}^r_\pi(M; K \cap L)_0$ and $\underline{F}= G'H'$, where 
$G' \in {{\rm Isot}^r(B; L)_0}$, $H' \in {{\rm Isot}^r(B; K)_0}$. 
Then $F = GH$ for some $G \in {\rm Isot}^r_\pi(M; L)_0$ and $H \in {\rm Isot}^r_\pi(M; K)_0$ with 
$\underline{G} = G'$ and $\underline{H} = H'$. 
\end{cor}

\begin{proof}
Take $K_1, L_1, \in {\cal F}(B)$ and $U, V \in {\cal O}(B)$ such that \\[1mm] 
\hsp $K \Subset K_1 \subset U$, $L \Subset L_1 \subset V$ and 
$F \in {\rm Isot}^r_\pi(M; U \cap V)_0$, $G' \in {\rm Isot}^r(B; V)_0$, $H' \in {\rm Isot}^r(B; U)_0$. \\[1mm] 
We can define a non-ambient isotopy \\[1mm] 
\hspp $G^\ast \in {\rm EIsot}^r_\pi(\widetilde{U} \cup \widetilde{V}, M)_0$ \ \ by \ \ 
$G^\ast = F$ on $\widetilde{U} \times I$ and $G^\ast = \id$ on $\widetilde{V} \times I$. \\[1mm] 
{ 
In fact, (i) $G^\ast$ is well-defined as a $C^r$ map between $(N, \Gamma)$ fiber bundles, since $F = \id$ on $(\widetilde{U} \cap \widetilde{V}) \times I$, \\
(ii) $\underline{G^\ast} = G'$ on $(U \cup V) \times I$, 
since $\underline{G^\ast}|_{U \times I} = \underline{F}|_{U \times I} = G'H'|_{U \times I} = G'|_{U \times I}$ and 
$\underline{G^\ast}  = \id = G'$ on $V \times I$, 
(iii) $G^\ast$ is an embedding since so is $\underline{G^\ast}$. 
}

By Proposition~\ref{rel_iso_lift} there exists $G \in {\rm Isot}^r_\pi(M)_0$ such that $\underline{G} = G'$ and 
$G = G^\ast$ on $(\widetilde{K}_1 \cup \widetilde{L}_1) \times I$. 
Then $G \in {{\rm Isot}^r_\pi(M; L)_0}$ and 
$G = F$ on $\widetilde{K}_1 \times I$. 
Let $H := G^{-1}F$. 
It follows that $H \in {\rm Isot}^r_\pi(M;K)_0$ and 
$\underline{H} = \underline{G}^{-1}\underline{F} = (G')^{-1} \underline{F} = H'$. 
\end{proof} 

\begin{cor}\label{cor 2 rel_iso_lift} $($Lifting of an isotopy of isotopies$)$ \\
Consider the $(N, \Gamma)$-bundle $\pi \times \id_I : M \times I \lra B \times I$ with $\partial B = \emptyset$. 
Suppose $\Psi = (\Psi_t)_{t \in I}$ is an isotopy in ${\rm Isot}^r(B)$ $($i.e., $\Psi \in {\rm Isot}^r(B \times I)$ with $\Psi_t \in {\rm Isot}^r(B)$ $(t \in I)$$)$. 
Then, for any $F \in {\rm Isot}^r_\pi(M)$ with $\underline{F} = \Psi_0$ there exists a bundle isotopy $\Phi = (\Phi_t)_{t \in I}$ in ${\rm Isot}^r_{\pi}(M)$ 
$($i.e., $\Phi \in {\rm Isot}^r_{\pi \times \id_I}(M \times I)$ with $\Phi_t \in {\rm Isot}^r_{\pi}(M)$ $(t \in I)$$)$
such that $\Phi_0 = F$ and $\underline{\Phi} = \Psi$. If $\Psi$ is an isotopy rel ends 
$($i.e., $(\Psi_t)_0 = (\Psi_0)_0$ and $(\Psi_t)_1 = (\Psi_0)_1$$)$, 
then we can take $\Phi$ a bundle isotopy rel ends $($i.e., $(\Phi_t)_0 = (\Phi_0)_0$ and $(\Phi_t)_1 = (\Phi_0)_1$$)$. 
\end{cor} 


\begin{proof} We apply Proposition~\ref{rel_iso_lift}  to the fiber bundle 
$\pi \times \id_I$, the isotopy $\Psi$ on the base $B \times I$ and 
the 0-level lift $F$ of $\Psi_0$, which yields $\Phi \in {\rm Isot}^r_{\pi \times \id_I}(M \times I)$ with $\Phi_0 = F$ and $\underline{\Phi} = \Psi$. 
Since $\Psi_t \in {\rm Isot}^r(B)$, we have $\Phi_t \in {\rm Isot}^r_{\pi}(M)$. 
When $\Psi$ is an isotopy rel ends, 
we can take the submanifold $K \equiv B \times \{ 0, 1 \}$ of $B \times I$ and the partial lift of $\Psi$ to $\widetilde{K} \equiv M \times \{ 0, 1\}$, \\
\hspace*{10mm} $\Phi' \in {\rm EIsot}^r_\pi(M \times \{ 0, 1\}, M \times I)$ : $\Phi'((x,i), t) = (F(x,i),t)$ $(((x,i), t) \in (M \times \{ 0, 1\}) \times I)$. \\
This leads to a bundle isotopy $\Phi$ rel ends. 
\end{proof} 

As usual, if we regard an isotopy as a path in a diffeomorphism group, 
then this result might be interpreted as the lifting of a path-homotopy (rel ends) for the map $P : {\rm Diff}^r_\pi(M) \to {\rm Diff}^r(B)$, 
though we have some technical difficulty on the topology of ${\rm Diff}^r_\pi(M)$ when the fiber $N$ is non-compact.

\subsection{{ \bf Perfectness of bundle diffeomorphism groups}} \mbox{} 

{ In this subsection we discuss some examples of fiber bundles which satisfies Assumption $(\ast)$ in Introduction.
In summary, we have the folloiwng list of examples at the moment. 

Suppose $\pi : M \to B$ is a $C^\infty$ bundle with fiber $N$ and structure group $\Gamma$ and $r \in \IZ_{\geq 1} \cup \{ \infty \}$

\begin{example}~\label{exp_perfectness} The group ${\rm Diff}^r_{\pi, c}(M)_0$ is perfect in the following cases : 
\benum 
\item[{\rm [\hspace{0.2mm}I\hspace{0.2mm}]}] (Principal bundle case) \hsh $\pi$ is a principal $G$ bundle for a compact Lie group $G$ (equivalently, 
$(N, \Gamma) = (G, G_L)$, where $G_L$ is the group of all left translations on $G$), 
$\partial B = \emptyset$, $\dim B \geq 1$ and $r \neq \dim B + 1$. 
\item[{\rm [I\hspace{-0.2mm}I]}] (Locally trivial bundle case) \hsh $\Gamma = {\rm Diff}^\infty(N)$, $N$ is a closed manifold, $\partial B = \emptyset$ and $r = \infty$. 
\eenum 
\end{example}

Below we give complementary explanations to these cases. \\ 
{[\hspace{0.2mm}I\hspace{0.2mm}] Principal bundle case} 

In \cite{AF} K.~Abe and K.~Fukui exhibited the perfectness of equivariant diffeomorphism groups under free action of compact Lie groups. 
The statement [\hspace{0.2mm}I\hspace{0.2mm}] in Example~\ref{exp_perfectness} is a precise translation of their result in term of bundle diffeomorphism groups on principal bundles. \\
{[\hspace{0.2mm}I\hspace{-0.2mm}I\hspace{0.2mm}] Locally trivial bundle case} 

First we recall a basic result on the perfectness of leaf preserving diffeomorphism groups on foliated manifolds, since 
the subgroup of fiber preserving diffeomorphisms in a bundle diffeomorphism group is 
a basic example of leaf preserving diffeomorphism groups. 

The following result is obtained in \cite{Ry} and \cite{Ts1}. 

\begin{thm}\label{thm_perfectness}
Suppose $M$ is a $C^\infty$ manifold with a $C^\infty$ foliation ${\cal F}$. 
Let ${\rm Diff}^\infty_c({\cal F})$ denote the group of $C^\infty$ diffeomorphisms of $M$ with compact support which send each leaf $L$ of ${\cal F}$ to $L$ itself. 
Let ${\rm Diff}^\infty_c({\cal F})_0$ denote the subgroup of ${\rm Diff}^\infty_c({\cal F})$ 
consisting of $f \in {\rm Diff}^\infty_c({\cal F})$ which is isotopic to $\id_M$ by a compactly supported isotopy $F$ with $F_t \in {\rm Diff}^\infty_c({\cal F})$ $(t \in I)$.
Then, ${\rm Diff}^\infty_c({\cal F})_0$ is perfect. 
\end{thm} 

Below we show that the statement [\hspace{0.2mm}I\hspace{-0.2mm}I\hspace{0.2mm}] in Example~\ref{exp_perfectness} is reduced to the above theorem. 

Suppose $\pi : M \to B$ is a $C^\infty$ locally trivial bundle with fiber $N$ and structure group ${\rm Diff}(N)$. 
Assume that $N$ is a closed manifold, $\partial B = \emptyset$ and $r = \infty$. Below we omit the symbol $\infty$ from the notations.

\benum[(1)] 
\item Consider the surjective group homomorphism $P_c : {\rm Diff}_{\pi, c}(M)_0 \lra {\rm Diff}_{c}(B)_0$, $P_c(f) = \underline{f}$. \\
We have the normal subgroups of ${\rm Diff}_{\pi, c}(M)_0$,  \\
\hsp \hsh ${\rm Ker}\,P_c$ \ \ and \ \  
$({\rm Ker}\,P_c)_0 := \{ F_1 \mid F \in {\rm Isot}_{\pi,c}(M)_0, \underline{F} = \id_{B \times I}\}$. \\   
Consider the $C^\infty$ foliation ${\cal F}_\pi = \{ M_b \mid b \in B \}$ on $M$. 
Since $N$ is compact, it follows that \\ 
\hsp \hsh $({\rm Ker}\,P_c)_0 = {\rm Diff}_c({\cal F}_\pi)_0$, \ which is perfect by Theorem~\ref{thm_perfectness}. 

\item First we consider the case where the bundle $\pi : M \to B$ is trivial. \\
(i) \ ${\rm Ker}\,P_c = ({\rm Ker}\,P_c)_0$. 
\bit 
\item[\bec] 
We may assume that $M = B \times N$ and $\pi = \pr_B : B \times N \lra B$. 
Given any $f \in {\rm Ker}\,P_c$. This means that $\underline{f} = \id_B$. 
There exists $F \in {\rm Isot}_{\pi, c}(M)_0$ with $F_1 = f$, which induces 
a $C^\infty$ map $\pr_N F : (B \times N) \times I \lra B \times N \lra N$. 
Note that $\pr_N F_t(b, \ast) \in {\rm Diff}^r(N)$ for any $(b,t) \in B \times I$ and that 
there exists $K \in {\cal K}(B)$
such that $\pr_N F_t(b, \ast) = \id_N$ for $(b,t) \in (B - K) \times I$. 
Therefore, we obtain $G \in {\rm Isot}^r_{\pi, c}(M)_0$ defined by \\
\hsp $G((b,x), t) = (b, \pr_N F_t(b,x))$ \ $((b,x), t) \in (B \times N) \times I)$. \\
Since $G_1 = f$ and $\underline{G} = \id_B$, it follows that $f \in ({\rm Ker}\,P_c)_0$. 
\eit 
(ii) The group ${\rm Diff}_{\pi,c}(M)_0$ is perfect. 
\bit 
\item[\bec] Give any $f \in {\rm Diff}_{\pi,c}(M)_0$. 
Since ${\rm Diff}_{c}^\infty(B)_0$ is perfect, we have a decomposition \\
\hsp $\underline{f} = [g_1', h_1'] \cdots [g_\ell', h_\ell']$ \ for some 
$g_1', h_1', \cdots, g_\ell', h_\ell' \in {\rm Diff}_{c}(B)_0$. \\
Since $P_c : {\rm Diff}_{\pi, c}(M)_0 \lra {\rm Diff}_{c}(B)_0$ is surjective, 
there exist $g_1, h_1, \cdots, g_\ell, h_\ell \in {\rm Diff}_{\pi, c}(M)_0$ such that 
$\underline{g_i} = g_i'$, $\underline{h_i} = h_i'$ $(i=1, \cdots, \ell)$. 
Then, $\widetilde{f} := [g_1, h_1] \cdots [g_\ell, h_\ell]$ 
satisfies $P_c(\widetilde{f}) = \underline{f}$ and 
$f\widetilde{f}^{-1} \in {\rm Ker}\,P_c = [{\rm Ker}\,P_c,{\rm Ker}\,P_c]$.  
Hence, we have \ $f = (f\widetilde{f}^{-1})\widetilde{f} \in [{\rm Diff}_{\pi, c}(M)_0, {\rm Diff}_{\pi, c}(M)_0]$. 
\eit 

\item The general case reduces to the trivial bundle case according to the standard strategy. 
\bit
\item[3-1)] Any $f \in {\rm Diff}_{\pi, c}(M)_0$ has a factorization $f = g_1 \cdots g_m$ so that 
each $g_i$ is sufficiently $C^1$-close to $\id_M$. 
\item[3-2)] If $f$ is sufficiently $C^1$-close to $\id_M$, then we can decompose $f$ as $f = h_1 \cdots h_n$ such that 
each $h_i$ has compact support in a local trivialization of the bundle $\pi$. 
\eit 

\item[] For any $C \in {\cal K}(B)$ we have the following subgroup endowed with the $C^1$-topology \\ 
\hsp \hsp ${\cal D}_C := \{ F_1 \mid F \in {\rm Isot}_\pi(M)_0, \ {\rm supp}\,F \subset \pi^{-1}(C) \} < {\rm Diff}_{\pi,c}(M)_0$. \\
Since $N$ is compact, $\pi^{-1}(C) \in {\cal K}(M)$, so that the $C^1$-topology on ${\cal D}_C$ is controlled on $\pi^{-1}(C)$ even if $M$ is noncompact.

\item[] 3-1) Given $f \in {\rm Diff}_{\pi, c}(M)_0$. There exists $F \in {\rm Isot}_{\pi, c}(M)_{\id, f}$ and $K \in {\cal K}(B)$ with ${\rm supp}\,F \subset \pi^{-1}(K)$. 
Take any $C^1$ nbd ${\cal U}$ of $\id_M$ in ${\cal D}_K$. 
We can find a subdivision $0 = t_0 < t_1 < \cdots < t_m = 1$ of the interval $I$ such that $g_i := F_{t_i}F^{\ -1}_{t_{i-1}} \in {\cal U}$ \ $(i = 1, \cdots, m)$. 
Since $f = g_m \cdots g_2g_1$, we may assume that $f$ itself is arbitrarily $C^1$ close to $\id_M$ in ${\cal D}_K$. 

\item[] 3-2) We can find $K_i \in {\cal K}(B)$, $U_i \in {\cal O}(B)$ $(i=1,2, \cdots, n)$ such that 
$K = \bigcup_{i = 1}^n K_i$, $K_i \subset U_i$ and the bundle $\pi$ admits 
a local trivialization $\phi_i : \pi^{-1}(U_i) \cong U_i \times N$ over $U_i$.  
For each $i = 1, \cdots, n$ take $L_i \in {\cal K}(B)$ with $K_i \Subset L_i \subset U_i$. 
Based on these data, for any $C^1$ nbd ${\cal V}$ of $\id_M$ in ${\cal D}_K$ 
we can find \\
\hsp a sequence \ ${\cal V} \equiv {\cal V}_{n+1} \supset {\cal V}_n \supset \cdots \supset {\cal V}_1 \equiv {\cal U}$ \ 
of $C^1$ nbds of $\id_M$ in ${\cal D}_K$ \\
in which ${\cal V}_i$ is sufficiently small for ${\cal V}_{i+1}$ so that it satisfies the following conditions : \\
\hsh Any $g = g_1 \in {\cal U}$ admits factorizations \ $g = h_1 \cdots h_i \,g_{i+1}$ \ $(i = 1, \cdots, n)$ \ such that 
\bit 
\itema $h_i \in {\cal D}_{K \cap L_i}$, \ $h_i, g_{i+1} \in {\cal V}_{i+1}$, \ $g_{i+1} = \id$ on $\pi^{-1}(K_i)$ \ and \ 
\itemb if $g_i(x) = x$, then $h_i(x) = g_{i+1}(x) = x$. 
\eit 
These conditions imply $g_{n+1} = \id_M$. In fact, 
since $g_1 \in {\cal U} \subset {\cal D}_K$, it follows that ${\rm supp}\,g_1 \subset \pi^{-1}(K)$ and $g_1 = \id$ on $M - \pi^{-1}(K)$.
Inductively, for each $i = 1, \cdots, n$ we have $g_{i+1} = \id$ on $\big(M - \pi^{-1}(K)\big) \cup \bigcup_{j = 1}^i \pi^{-1}(K_j)$. 
Therefore, $g_{n+1} = \id_M$ and we have the required factorization $g = h_1 \cdots h_n$ with $h_i \in {\cal V} \cap {\cal D}_{L_i}$ $(i = 1, \cdots, n)$.
\eenum
}

\subsection{Basic lemma for commutator length} \mbox{} 

The estimates of $cl$ and $clb_\pi$ is based on the next lemma, 
which is a fiber bundle version of the arguments in \cite{BIP}, \cite[Theorem 2.1]{Ts2}, \cite{Ts3}, \cite[Proof of Theorem 1.1\,(1)]{Ts4}, {etc}. 

Suppose $\pi : M \to B$ is a fiber bundle with fiber $N$ and structure group $\Gamma$ and let $r \in \IZ_{\geq 0} \cup \{ \infty \}$. 

\begin{lem}\label{basic_lemma_cl} \mbox{} 
Suppose $U$ is an open subset of ${\rm Int}\,B$. 
\benum 
\item[{\rm (1)}] If $U$ is an open ball and ${\rm Diff}^r_{\pi,c}(\widetilde{U})_0$ is perfect, then $cld\,{\rm Diff}^r_{\pi,c}(\widetilde{U})_0 \leq 2$. $($cf. \cite{Fu2}$)$
\item[{\rm (2)}] More generally, suppose $F \in {\rm Isot}^r_{\pi, c}(\widetilde{U})_0$ and let $f := F_1 \in {\rm Diff}^r_{\pi, c}(\widetilde{U})_0$. 
Assume that there exist compact subsets $W \Subset V$ in $U$ and $\phi \in {\rm Diff}^r_{c}(U)_0$ such that \\
\hspp $\phi(V) \subset W$ \ \ and \ \ ${\rm Cl}_B \,\pi({\rm supp}\,F) \subset O := {\rm Int}_B V - W$. \\
If ${\rm Diff}^r_{\pi,c}(\widetilde{O})_0$ is perfect
, then 
$cl\,f \leq 2$ \ in ${\rm Diff}^r_{\pi,c}(\widetilde{U})_0$. 
Moreover, there exists a factorization \\ 
$f = gh$ 
\ \ for some $g \in {\rm Diff}^r_{\pi,c}(\widetilde{U})_0^c$ and $h \in {\rm Diff}^r_{\pi,c}(\widetilde{{\rm Int}\,W})_0^c$. 
\eenum 
\end{lem}  

\subsection{Relation between $\zeta$ and $clb_\pi$ in bundle diffeomorphism groups} \mbox{} 

Next we clarify relations between the conjugation-generated norm 
and the commutator length supported in balls in bundle diffeomorphism groups (cf. \cite{BIP, Ts3}). 

Suppose $\pi : M \to B$ is a fiber bundle with fiber $N$ and structure group $\Gamma$ and let $r \in \IZ_{\geq 0} \cup \{ \infty \}$. 
Below we assume that $B$ is a connected $n$-manifold.
Let ${\cal B}^r(B)$ denote the collection of $C^r$ $n$-balls in $B$.  
For $D \in {\cal B}^r(B)$ any commutator $[a,b]$ ($a, b \in {\rm Diff}^r_\pi(M; B_D)_0$) is called a commutator supported in $D$.   
Consider the subset ${\cal S}^r_\pi$ of ${\rm Diff}^r_\pi(M)_0$ defined by \ 
${\cal S}^r_\pi := \bigcup \,\big\{ {\rm Diff}^r_\pi(M; B_D)_0^c \mid D \in {\cal B}^r(B) \big\}$. \\
Since ${\cal S}^r_\pi$ is symmetric and conjugation-invariant in ${\rm Diff}^r_\pi(M)_0$, 
we obtain the corresponding ext.~conj.-invariant norm, which we call the commutator length supported in balls in $B$
and denote by the symbol \\ 
\hsp $clb_\pi : {\rm Diff}^r_\pi(M)_0 \to \IZ_{\geq 0} \cup \{ \infty \}$. 

For notational simplicity, 
let ${\cal G} = {\rm Diff}_\pi^r(M)_0$ and ${\cal G}_A := {\rm Diff}_\pi^r(M, B_A)_0$ for $A \subset B$. 
Note that $g({\cal G}_A)g^{-1} = {\cal G}_{\underline{g}(A)}$ \ for any $g \in {\cal G}$. 

\begin{fact}\label{lem_C_g^4}  $($cf.~\cite{BIP}, \cite[Lemma 3.1]{Ts2}, {etc}.$)$ \ \ 
Suppose $g \in {\cal G}$, $A \subset B$ and $\underline{g}(A) \cap A = \emptyset$. 
\benum 
\item[{\rm (1)}] 
$({\cal G}_A)^c \subset C_g^4$ \ in ${\cal G}$. 
More precisely, for any $a, b \in {\cal G}_A$ the following identity holds :  \\[0.5mm] 
\hspace*{10mm} $[a,b] =g\big(g^{-1}\big)^c g^{bc}\big(g^{-1}\big)^b$ \ in \ ${\cal G}$, \hsh 
where $c := a^{g^{-1}} \in {\cal G}_{\underline{g}^{-1}(A)}$. Note that $cb = bc$.  
\vskip 0.5mm 
\item[{\rm (2)}] $({\cal G}_{A'})^c \subset C_g^4$ \ in ${\cal G}$, if $A' \subset B$ and 
$\underline{h}(A') \subset A$ for some $h \in {\cal G}$.  
\eenum 
\end{fact} 

\begin{proof} (2) 
Let $k := h^{-1}gh \in {\cal G}$. Then, $\underline{k}(A') \cap A' = \emptyset$ and by (1) we have $({\cal G}_{A'})^c \subset C_{k}^4 = C_g^4$.  
\end{proof} 

\begin{lem}\label{lem_zeta_leq_4clb_pi} 
{\rm (1)} $\zeta_g \leq 4 \,clb_\pi$ in ${\rm Diff}_\pi^r(M)_0$ for any $g \in {\rm Diff}_\pi^r(M)_0 - {\rm Ker}\,P$. 
\benum
\item[{\rm (2)}] 
The group ${\rm Diff}_\pi^r(M)_0$ is uniformly simple relative to ${\rm Ker}\,P$ if $clb_\pi d\,{\rm Diff}_\pi^r(M)_0 < \infty$.  
\eenum
\end{lem}

\begin{proof}
(1) Take any $g \in {\rm Diff}_\pi^r(M)_0 - {\rm Ker}\,P$. 

First we see that $({\cal G}_D)^c = {{\rm Diff}_\pi^r(M; B_D)_0^c} \subset C_g^4$ \ for any $D \in {\cal B}^r(B)$. 
In fact, since $\underline{g} \neq \id_B$, 
there exists $E \in {\cal B}^r(B)$ with $\underline{g}(E) \cap E = \emptyset$.
Since $B$ is connected, we can find $\eta \in {\rm Diff}^r(B)_0$ with $\eta(D) \subset E$. 
There exists $h \in {\cal G}$ with $\underline{h} = \eta$. 
Then, the assertion follows from Fact~\ref{lem_C_g^4}\,(2). 

To show that $\zeta_g \leq 4 \,clb_\pi$, given any $f \in {\cal G}$. If $clb_\pi(f) = \infty$, then the assertion is trivial. 
Suppose $k := clb_\pi(f) \in \IZ_{\geq 0}$. 
Then $f = f_1 \cdots f_k$ for some $D_i \in {\cal B}^r(B)$ and $f_i \in ({\cal G}_{D_i})^c$ $(i=1, \cdots, k)$. 
Since $f_i \in C_g^4$ and $\zeta_g(f_i) \leq 4$ $(i=1, \cdots, k)$, it follows that $\zeta_g(f) \leq 4k$.  
\end{proof}

\section{Boundedness of bundle diffeomorphism groups over a circle} 

\subsection{Generality on winding number} \mbox{} 

\noindent {{[\hspace{0.2mm}I\hspace{0.2mm}]}} For a bundle over a circle $S^1$ we can use the winding number on $S^1$ to extract some invariant of bundle diffeomorphisms over $S^1$. 

Let ${\cal P}(X) = C^0(I, X)$ denote the set of $C^0$ paths in a topological space $X$. 
As usual, for $a,b, c \in {\cal P}(X)$, 
\bit 
\itemI the inverse path $c_- \in {\cal P}(X)$ is defined by $c_-(t) = c(1-t)$ $(t \in I)$, 
\itemII if $a(1) = b(0)$, then the concatenation $a \ast b \in {\cal P}(X)$ is defined by 
$(a \ast b)(t) = 
\mbox{\small $\left\{ \hspace{-1mm} 
\bary[c]{l@{ \ }l}
a(2t) & (t \in \big[0,\frac{1}{2}\big]) \\[2mm]
b(2t-1) & (t \in \big[\frac{1}{2}, 1\big])
\eary \right.$}$, 
\vskip 1mm 
\itemiii $a \simeq_\ast b$ means that $a$ and $b$ are $C^0$ homotopic relative to ends, that is, 
there exists a $C^0$ homotopy $\eta : a \simeq b$ with 
$\eta_t(0) = a(0) = b(0)$ and $\eta_t(1) = a(1) = b(1)$. 
\eit 
\vskip 2mm 

Take a universal covering $\pi_{S^1} : \IR \to \IR/\IZ \cong S^1$ with $\IZ$ as the covering transformation group. 
The winding number $\lambda : {\cal P}(S^1) \to \IR$ is define as follows. 
For each $c \in {\cal P}(S^1)$ take any lift $\widetilde{c} \in {\cal P}(\IR)$ of $c$ (i.e., $\pi_{S^1}\,\widetilde{c} = c$) and 
define $\lambda(c) \in \IR$ by $\lambda(c) := \widetilde{c}(1) - \widetilde{c}(0)$. 
The value $\lambda(c)$ is independent of the choice of $\widetilde{c}$. 
For $f \in C^0(S^1) \equiv C^0(S^1, S^1)$, its degree is denoted by $\deg f$. 

\bfact\label{fact_lambda} The map $\lambda$ has the following properties. Suppose $a,b,c \in {\cal P}(S^1)$. 
\benum 
\item 
\bit 
\itemI $\lambda(c) = \deg c \in \IZ$ {if} $c$ is a closed loop in $S^1$ ($c(0) = c(1)$). 
\itemII $\lambda(fc) = (\deg f)\, \lambda(c)$ if $c$ is a closed loop and $f \in C^0(S^1)$. 
\eit 
\item 
\bit 
\itemI $\lambda(a \ast b) = \lambda(a) + \lambda(b)$ if $a(1) = b(0)$.   
\hsp (ii) \ $\lambda(c_-) = - \lambda(c)$.  
\eit 
\item 
\bit 
\itemI $\lambda(a) = \lambda(b)$ if $a \simeq_\ast b$.  

\itemII $\lambda(c \gamma) = \lambda(c)$ for any $\gamma \in {C^0(I) \equiv C^0(I, I)}$ with $\gamma(0) = 0$ and $\gamma(1) = 1$. 
\eit 
\item $|\lambda(fc) - \lambda(c)| < 1$ for any 
$f \in {\rm Diff}^0(S^1)_0$. \\
($|\lambda(fc) - \lambda(c)|$ can be arbitrarily large if $f$ moves within $C^0(S^1)$.) 
\eenum
\efact

\noindent {{[\hspace{0.2mm}I\hspace{-0.2mm}I\hspace{0.2mm}]}} 
Next we consider the winding number of a distinguished point under homotopies on $S^1$. 
Recall that the symbol $H^0(S^1)$ denotes the set of $C^0$ homotopies on $S^1$ (i.e.,  $C^0$ maps $G : S^1 \times I \to S^1$).  
For $g, h \in C^0(S^1)$ let $H^0(S^1)_{g,h} := \{ G \in H^0(S^1) \mid G_0 = g, G_1 = h\}$. 
As usual,  for $F, G, H \in H^0(S^1)$ we use the following notations. 
\bit 
\itemI 
\hspace*{-8mm} 
\begin{minipage}[t]{500pt} 
\baselineskip 6.5mm
\bit 
\itema The homotopy $F^- \in H^0(S^1)$ is defined by $(F^-)_t= F_{1-t}$ $(t \in I)$. 
\itemb The composition $HG \in H^0(S^1)$ is defined by $(HG)_t= H_tG_t$ $(t \in I)$. 
\itemc 
If $G_1 = H_0$, the concatenation $G \ast H \in H^0(S^1)$ is defined by 
\mbox{ \small $(G \ast H)_t = 
\left\{ \hspace{-1mm} 
\bary[c]{l@{ \ }l}
G_{2t} & (t \in \big[0,\frac{1}{2}\big]) \\[2mm]
H_{2t-1} & (t \in \big[\frac{1}{2}, 1\big])
\eary \right.$}. 
\eit 
\end{minipage}
\vskip 2mm 
\itemII $G \simeq_\ast H$ means that $G$ and $H$ are $C^0$ homotopic relative to ends (cf. \S3.1).
\itemiii If $\Phi : G \simeq G'$ and $\Psi : H \simeq H'$ in $H^0(S^1)$, then 
we have the composition $\Psi \Phi : HG \simeq H'G'$ defined by $(\Psi \Phi)_t = \Psi_t \Phi_t$ (or $(\Psi \Phi)_{s,t} = \Psi_{s,t} \Phi_{s,t}$). 
Note that $\Psi\Phi: HG \simeq_\ast H'G'$ if $\Phi : G \simeq_\ast G'$ and $\Psi : H \simeq_\ast H'$. 

\itemiv Any $\gamma \in C^0(I)$ induces a homotopy $G^\gamma \in H^0(S^1)$ defined by $(G^\gamma)_t = G_{\gamma(t)}$ $(t \in I)$. \\
\hspace*{-8mm} 
\begin{minipage}[t]{500pt} 
\baselineskip 6.5mm
\bit 
\itema If $\alpha, \beta \in C^0(I)$ and $\eta : \alpha \simeq \beta$, 
then we obtain a homotopy 
$G^\eta : G^\alpha \simeq G^\beta$ defined by \\
$(G^\eta)_t = G^{\eta_t}$ (or $(G^\eta)_{s,t} = G_{\eta(s,t)}$). 
Note that $G^\eta : G^\alpha \simeq_\ast G^\beta$ if $\eta : \alpha \simeq_\ast \beta$. 

\itemb If $\gamma \in C^0(I)$ and $\gamma(0) = 0$, $\gamma(1) = 1$, then we have a homotopy \\
$\xi : \id_I \simeq_\ast \gamma$ : $\xi(s,t) = {\xi_t(s)} = (1-t)s + t\gamma(s)$.  
\eit 
\end{minipage}
\eit 
\vskip 2mm 

Fix a point $p \in S^1$ and define \ \ $\mu \equiv \mu_p : H^0(S^1) \to \IR$ \ \ by $\mu(G) := \lambda(G_p)$, \\
where $G_p := G(p, \ast) \in {\cal P}(S^1)$. 
Note that $G_p(t) = G(p,t) = G_t(p)$ and $(G^\gamma)_p = G_p \gamma$ for any $\gamma \in C^0(I)$. 

\bfact\label{fact_mu} The map $\mu$ has the following properties. Suppose $F, G, H \in H^0(S^1)$.
\benum 
\item (i) \ $\mu(G \ast H) = \mu (G) + \mu(H)$ if $G_1 = H_0$. 
\hsp (ii) \ $\mu(F^-) = - \mu(F)$ 
\item $\mu(G) = \mu(H)$ if $G \simeq_\ast H$.  
\item $\mu(G^\gamma) = \mu(G)$ for any $\gamma \in {C^0(I)}$ with $\gamma(0) = 0$ and $\gamma(1) = 1$.
{
\item $\mu(GH) = \mu(G_0H) + \mu(GH_1) = \mu(GH_0) + \mu(G_1H)$. 
\item \bit 
\itemI $|\mu(fG) - \mu(G)| <1$ for any $f \in {\rm Diff}^0(S^1)_0$. 
\itemII $|\mu(GH) - \mu(G) - \mu(H)| < 1$ for any $G, H \in {\rm Isot}^0(S^1)_0$. 
\eit 
}
\eenum 
\efact 

{The statement (5)(ii) follows from the statements (4) and (5)(i). In fact, since $H_0 = \id_{S^1}$, we have \\
\hsp \hsh $|\mu(GH) - \mu(G) - \mu(H)| = |(\mu(GH_0) + \mu(G_1H)) - \mu(G) - \mu(H)| = |\mu(G_1H) - \mu(H)| < 1$.} 

The set $H^0(S^1)$ is a monoid with respect to the composition of homotopies and 
$H^0(S^1)_{\id, \id}$ is a submonoid of $H^0(S^1)$.
As restrictions of $\mu$ we obtain the following maps.

\bfact\label{fact_mu_homo} \mbox{}
\benum
\item The map $\mu : H^0(S^1)_{\id, \id} \to \IZ$ is a surjective monoid homomorphism. 
\item The map $\mu : {\rm Isot}^r(S^1)_{\id, \id} \to \IZ$ is a surjective group homomorphism for any $r \in \IZ_{\geq 0} \cup \{ \infty \}$. 
\eenum 
\efact 

\bfact\label{fact_mu=0} \mbox{} 
\benum
\item If $F \in {\rm Isot}^r(S^1)_{\id,\id}$ and $\mu(F) = 0$, then $F \simeq_\ast \id_{S^1 \times I}$ {(a $C^r$ isotopy rel ends)}. 
\item The map $\mu$ induces a group isomorphism $\widetilde{\mu} : {\rm Isot}^r(S^1)_{\id,\id}\big/\simeq_\ast \ \cong \ \IZ$. 
\eenum 
\efact 

The assertion (1) follows from Lemma~\ref{lem_sdr} below, which reviews the basic ``topological property'' of the group ${\rm Diff}^r(S^1)_0$. 
We identify $S^1$ with the unit circle in the complex plane $\IC$ so that the distinguished point $p = 1$. 
Then, {${\rm Diff}^r(S^1)_0$} includes the rotation group $SO(2)$ canonically.
For $q_1, q_2 \in S^1$ let $\theta_{q_1, q_2} \in SO(2)$ denote the rotation of $S^1$ which maps $q_1$ to $q_2$. 
We set ${\cal D}^r_{q_1, q_2} := \{ f \in {\rm Diff}^r(S^1)_0 \mid f(q_1) = q_2 \}$. Then ${\cal D}^r_{q_1, q_2} \cap SO(2) = \{ \theta_{q_1, q_2} \}$. \\
Recall the following conventions for $C^\infty$ manifolds $M$, $N$ 
and a subset ${\cal S} \subset C^0(M, N)$. 
\bit 
\itemI For a $C^\infty$ manifold $L$, a map $\alpha : L \to C^0(M, N)$ is $C^r$ 
if the associated map $\widetilde{\alpha} : L \times M \to N$, $\widetilde{\alpha}(u,x) = \alpha(u)(x)$ is $C^r$.  
\itemII A map $\phi : {\cal S} \times I \to {\cal S}$ is $C^r$ if 
for any $C^\infty$ manifold $L$ and any $C^r$ map $\alpha : L \to {\cal S}$ the composition $\phi(\alpha \times \id_I) : L \times I \to {\cal S}$ is $C^r$.  

\itemiii Suppose ${\cal S}$ admits a left action of a group $\Theta$ and ${\cal T}$ is a ($\Theta$-invariant) subset of ${\cal S}$. 
A $C^r$ $\Theta$-equivariant strong deformation retraction (SDR) of ${\cal S}$ onto ${\cal T}$ means 
a $C^r$ map $\phi : {\cal S} \times I \to {\cal S}$ such that $\phi_0 = \id_{\cal S}$, $\phi_1({\cal S}) = {\cal T}$, $\phi_t = \id$ on ${\cal T}$ $(t \in I)$ and 
$\phi_t(\theta f) = \theta \phi_t(f)$ $(f \in {\cal S}, \theta \in \Theta, t \in I)$. 
\eit 

\blem\label{lem_sdr} \mbox{} 
There exists a canonical $SO(2)$-equivariant SDR $\phi$ 
of ${\rm Diff}^0(S^1)_0$ onto $SO(2)$ such that
$\phi_t({\cal D}^0_{p,q}) \subset {\cal D}^0_{p,q}$ $(q \in S^1, t \in I)$. Furthermore, for any $r \in \IZ_{\geq 0} \cup \{ \infty \}$
the map $\phi$ restricts to a $C^r$ $SO(2)$-equivariant SDR 
of ${\rm Diff}^r(S^1)_0$ onto $SO(2)$ such that
$\phi_t({\cal D}^r_{p,q}) \subset {\cal D}^r_{p,q}$ $(q \in S^1, t \in I)$.
The additional condition on $\phi$ means that 
$\phi$ induces a $C^r$ SDR 
of ${\cal D}^r_{p,q}$ onto the singleton $\{ \theta_{p,q} \}$. 
\elem 

\bpf \mbox{} 
\benum 
\item First we construct 
a canonical SDR $\chi$ of ${\cal D}^0_{p,p}$ onto $\{ \id_{S^1} \}$ such that 
$\chi$ restricts to a $C^r$ SDR of ${\cal D}^r_{p,p}$ onto $\{ \id_{S^1} \}$ for any $r \in \IZ_{\geq 0} \cup \{ \infty \}$. 
We identify the universal cover $\pi_{S^1} : \IR \to S^1$ with the exponential map $\pi_{S^1}(s) = e^{2\pi i s}$. 
Then the map $\chi$ is defined by the linear isotopy \\
\hspp $\chi_t(f)(\pi_{S^1}(s)) = \pi_{S^1}((1-t)\widetilde{f}(s) + ts)$ \ \ $(f \in {\cal D}^0_{p,p}, s \in \IR, t \in I)$, \\
where $\widetilde{f} \in {\rm Diff}^0(\IR)$ is the lift of $f$ with $\widetilde{f}(n) = n$ $(n \in \IZ)$. 

\item The map $\phi$ is defined by $\phi_t(f) = \theta_{p,f(p)} \chi_t(\theta_{f(p), p}f)$ $(f \in {\rm Diff}^0(S^1)_0, t \in I)$. 
 \eenum 
\vskip -7mm 
\epf 

\vskip 4mm 

\noindent {{[\hspace{0.2mm}I\hspace{-0.2mm}I\hspace{-0.2mm}I\hspace{0.2mm}]}} 
Finally we consider the winding number of a distinguished fiber under bundle homotopies over $S^1$. 
Suppose $\pi : M \to S^1$ is {a fiber bundle with fiber $N$ and structure group $\Gamma$.} 
Recall that the symbol $H^0_\pi(M)$ denotes the set of $C^0$ homotopies on $\pi$ (i.e., 
$C^0$ homotopies $F : M \times I \to M$ such that $F_t \in C^0_\pi(M) \equiv C^0(\pi, \pi)$ $(t \in I)$). 
Recall that, ${\cal E}_{g,h} := \{ G \in {\cal E} \mid G_0 = g, G_1 = h\}$ 
for ${\cal E} \subset H^0_\pi(M)$ and $g, h \in C^0_\pi(M)$.  
As before,  for $F, G, H \in H^0_\pi(M)$ we use the following notations. 

\bit 
\itemI 
\hspace*{-8mm} 
\begin{minipage}[t]{500pt} 
\baselineskip 6.5mm
\bit 
\itema The homotopy $F^- \in H^0_\pi(M)$ is defined by $(F^-)_t= F_{1-t}$ $(t \in I)$. 
\itemb The composition $HG \in H^0_\pi(M)$ is defined by $(HG)_t= H_tG_t$ $(t \in I)$. 
\itemc 
If $G_1 = H_0$, the concatenation $G \ast H \in {H}^0_\pi(M)$  is defined by 
\mbox{ \small $(G \ast H)_t = 
\left\{ \hspace{-1mm} 
\bary[c]{l@{ \ }l}
G_{2t} & (t \in \big[0,\frac{1}{2}\big]) \\[2mm]
H_{2t-1} & (t \in \big[\frac{1}{2}, 1\big])
\eary \right.$}. 
\eit 
\end{minipage}
\vskip 2mm 
\itemII $G \simeq_\ast H$ means that $G$ and $H$ are $C^0$ homotopic relative to ends (cf. \S3.2).
\itemiii If $\Phi : G \simeq G'$ and $\Psi : H \simeq H'$ in ${H}^0_\pi(M)$, then 
we have the composition $\Psi \Phi : HG \simeq H'G'$ defined by $(\Psi \Phi)_t = \Psi_t \Phi_t$ (or $(\Psi \Phi)_{s,t} = \Psi_{s,t} \Phi_{s,t}$). 
Note that $\Psi\Phi: HG \simeq_\ast H'G'$ if $\Phi : G \simeq_\ast G'$ and $\Psi : H \simeq_\ast H'$. 

\itemiv Any $\gamma \in {C^0(I)}$ induces a homotopy $G^\gamma \in {H}^0_\pi(M)$ defined by $(G^\gamma)_t = G_{\gamma(t)}$ $(t \in I)$. \\
\hspace*{-8mm} 
\begin{minipage}[t]{500pt} 
\baselineskip 6.5mm
\bit 
\itema If $\alpha, \beta \in {C^0(I)}$ and $\Lambda : \alpha \simeq \beta$, 
then we obtain a homotopy 
$G^\Lambda : G^\alpha \simeq G^\beta$ defined by \\
$(G^\Lambda)_t = G^{\Lambda_t}$ (or $(G^\Lambda)_{s,t} = G_{\Lambda(s,t)}$). 
Note that $G^\Lambda : G^\alpha \simeq_\ast G^\beta$ if $\Lambda : \alpha \simeq_\ast \beta$. 

\itemb $G^\gamma \simeq_\ast G$ if $\gamma \in {C^0(I)}$ and $\gamma(0) = 0$, $\gamma(1) = 1$. 
\eit 
\end{minipage}
\vskip 2mm 
\eit 
\vskip 2mm 
\bremb\label{rem_underline}
\benum 
\item (i) \ $\underline{G^-} = \underline{G}^-$, $\underline{HG} = \underline{H}\,\underline{G}$, \hsp
(ii) \ $\underline{G \ast H} = \underline{G} \ast \underline{H}$ if $G, H \in H^0_\pi(M)$ and $G_1 = H_0$. 
\item If $\Phi : G \simeq_\ast H$, then $\underline{\Phi} : \underline{G} \simeq_\ast \underline{H}$
\eenum
\erem

Consider the map \ \ $\nu : H^0_\pi(M) \to \IR$, $\nu(G) := \mu(\underline{G})$. \\
The following conclusion follows directly from Fact \ref{fact_mu} and {Remark}~\ref{rem_underline}.

\bfact\label{fact_nu} The map $\nu$ has the following properties. Suppose $F, G, H \in H^0_\pi(M)$.
\benum
\item (i) \ $\nu(G \ast H) = \nu (G) + \nu(H)$ if $G_1 = H_0$. \hsp (ii) \ $\nu(F^-) = - \nu(F)$. 

\item $\nu(G) = \nu(H)$ if $G \simeq_\ast H$ 
\item $\nu(G^\gamma) = \nu(G)$ if $\gamma \in {C^0(I)}$ and $\gamma(0) = 0$, $\gamma(1) = 1$. 
\item $\nu(GH) = \nu(G_0H) + \nu(GH_1) = \nu(GH_0) + \nu(G_1H)$. 
\item \bit 
\itemI $|\nu(fG) - \nu(G)| < 1$ for any $f \in {\rm Diff}^0_{\pi}(M)_0$. 
\itemII $|\nu(GH) - \nu(G) - \nu(H)| < 1$ for any $G, H \in {\rm Isot}^0_\pi(M)_0$. 
\eit 
\eenum
\efact 


The set $H^0_\pi(M)$ is a monoid with respect to the composition of homotopies and 
the map \\
$P_H : H^0_\pi(M) \to H^0(S^1)$, $P_H(G) = \underline{G}$, is a monoid homomorphism. 
Since $H^0(S^1)_{\id, \id}$ is a submonoid of $H^0(S^1)$, the inverse image 
$P_H^{\ -1}(H^0(S^1)_{\id, \id})$ is a submonoid of $H^0_\pi(M)$.
Since $\nu = \mu P_H$, Fact~\ref{fact_mu_homo} and Fact~\ref{fact_nu}\,(5)(ii) {imply} the following conclusion. 

\bfact\label{fact_nu_homo_pre} \mbox{} 
\benum
\item The map $\nu$ restricts to a surjective monoid homomorphism $\nu : P_H^{\ -1}(H^0(S^1)_{\id, \id}) \to \IZ$.  
\item The map $\nu : {\rm Isot}^0_\pi(M)_0 \to \IR$ is a quasimorphism, which satisfies the estimate : \\
\hspp $|\nu(GH) - \nu(G) - \nu(H)| < 1$ for any $G, H \in {\rm Isot}^0_\pi(M)_0$. 
\eenum 
\efact 

\subsection{The map $\widehat{\nu} : {\rm Diff}^r_\pi(M)_0 \to \IR_k$} \mbox{} 

Suppose $\pi : M \to {S^1}$ is a fiber bundle with fiber $N$ and structure group $\Gamma < {\rm Diff}(N)$. 
Let $r \in \IZ_{\geq 0} \cup \{ \infty \}$. 
We are concerned with the following surjective group homomorphisms.
\bit 
\itemI $P_I : {\rm Isot}^r_\pi(M)_0 \to {\rm Isot}^r(S^1)_0$, $P_I(F) = \underline{F}$ \hsp 
(ii) \ $P : {\rm Diff}^r_\pi(M)_0 \to {\rm Diff}^r(S^1)_0$, $P(f) = \underline{f}$
\eit 
Their surjectivity follows from Isotopy lifting property of $\pi$. 
For notational simplicity we set \\ 
${\cal I} \equiv P_I^{\,-1}({\rm Isot}^r(S^1)_{\id, \id})$. 


\bfact\label{fact_nu_homo} \mbox{} 
\benum 
\item 
The map $\nu$ restricts to the following maps. 
\bit 
\itemI a surjective map $\nu : {\rm Isot}^r_\pi(M)_0 \to \IR$
\itemII a surjective group homomorphism $\nu_{\cal I} : {\cal I} \to \IZ$. 
\eit 

\item For any $f \in {\rm Diff}^r_\pi(M)_0$ and $F \in {\rm Isot}^r_\pi(M)_{\id, f}$ \\ 
\hsh ${\rm Isot}^r_\pi(M)_{\id, f} = F\,{\rm Isot}^r_\pi(M)_{\id, \id}$ \ \ and \ \ $\nu({\rm Isot}^r_\pi(M)_{\id, f}) = \nu(F) + \nu({\rm Isot}^r_\pi(M)_{\id, \id})$ . 
\eenum 
\efact 

The relations among $\nu$, $\mu$ and $P_I$ are illustrated in the following diagrams : \\
\hspace*{10mm} 
{\small $\bary[t]{c@{ \ \ }c@{ \ \ }cl}
{\rm Isot}^r_\pi(M)_0 & \vartriangleright & {\cal I} & \\[0.5mm] 
\rotatebox{90}{$\lla\hspace{-5mm} \lla$} & \hspace{-31mm} \smash{\raisebox{2mm}{$P_I$}} & \rotatebox{90}{$\lla\hspace{-5mm} \lla$}
& \hspace{-12mm} \smash{\raisebox{2mm}{$P_I$}} \\
{\rm Isot}^r(S^1)_0 & \vartriangleright & {\rm Isot}^r(S^1)_{\id, \id} & 
\eary$ 
\hspace{10mm} 
$\bary[t]{cccl}
{\rm Isot}^r_\pi(M)_0 & & \hspace*{-15mm} \smash{\raisebox{-3mm}{$\nu$}} \\[0.5mm] 
\rotatebox{90}{$\lla\hspace{-5mm} \lla$} & \hspace{-40mm} \smash{\raisebox{2mm}{$P_I$}} & 
\hspace{-19.8mm} \smash{\raisebox{8mm}{\rotatebox{-35}{\makebox(40,0){\rightarrowfill}}}} & 
\hspace{-20.5mm} \smash{\raisebox{8mm}{\rotatebox{-35}{\makebox(38,0){\rightarrowfill}}}} \\[-2.5mm] 
& \hspace*{-4.5mm} \mu & & \\[-2mm] 
{\rm Isot}^r(S^1)_0 & \hspace{-2.5mm} \makebox(30,0){\rightarrowfill} & \IR & \hspace{-20mm} \makebox(27,0){\rightarrowfill}
\eary$ 
\hspace*{5mm} 
$\bary[t]{ccll}
{\cal I} & & & \hspace*{-17mm} \smash{\raisebox{-3mm}{$\nu_{\cal I}$}} \\[0.5mm] 
\rotatebox{90}{$\lla\hspace{-5mm} \lla$} & \hspace{-44mm} \smash{\raisebox{2mm}{$P_I$}} & 
\hspace{-22mm} \smash{\raisebox{7mm}{\rotatebox{-22}{\makebox(60,0){\rightarrowfill}}}} & 
 \hspace{-27.7mm} \smash{\raisebox{7mm}{\rotatebox{-22}{\makebox(56.5,0){\rightarrowfill}}}} \\[-2.5mm] 
& \hspace*{-6mm} \mu & & \\[-2mm] 
{\rm Isot}^r(S^1)_{\id, \id} & \hspace{-2mm} \makebox(30,0){\rightarrowfill} & \IZ & 
\hspace{-20mm} \makebox(27.5,0){\rightarrowfill}
\eary$}
\vskip 4mm 

Consider the surjective group homomorphism \hsh $R : {\rm Isot}^r_\pi(M)_0 \to {\rm Diff}^r_\pi(M)_0$, $R(F) = F_1$. \\
We have the normal subgroups \\
\hspp (i) ${\rm Ker}\,P_I \, \vartriangleleft \, {\rm Isot}^r_\pi(M)_0$, \hsp 
(ii) ${\rm Ker}\,P, \ ({\rm Ker}\,P)_0 := R({\rm Ker}\,P_I) \, \vartriangleleft \, {\rm Diff}^r_\pi(M)_0$. 

\bfact\label{fact_Ker} \mbox{} 
\benum 
\item ${\rm Ker}\,R = {\rm Isot}^r_\pi(M)_{\id, \id}$, \ \ ${\cal I} = R^{-1}({\rm Ker}\,P)$. \hsp 
{ 
(2) ${\rm Ker}\,\nu_{\cal I} = ({\rm Ker}\,P_I)\cdot ({\rm Ker}\,\nu_{\cal I})_{\id, \id}$. 
\item[(3)] $R^{-1}(({\rm Ker}\,P)_0) = ({\rm Ker}\,P_I)\cdot {\rm Isot}^r_\pi(M)_{\id, \id} = ({\rm Ker}\,\nu_{\cal I}) \cdot {\rm Isot}^r_\pi(M)_{\id, \id}$.  
}
\eenum
\efact 
\vspace*{-2mm} 
{\small 
$$\xymatrix@
R=15pt@L=5pt@C=15pt@M+2pt
{\relax
0 \ar[r] &  {\rm Isot}^r_\pi(M)_{\id, \id} \ar@{}[r]|*{\rotatebox{90}{$\bigcap$}} \ar@{}[d]|*{\rotatebox{0}{$||$}}  & 
{\rm Isot}^r_\pi(M)_0 \ar[r]^R \ar@{}[d]|*{\rotatebox{0}{$\bigtriangledown$}} & 
{\rm Diff}^r_\pi(M)_0 \ar@{}[d]|*{\rotatebox{0}{$\bigtriangledown$}} \ar[r] & 0 \\ 
0 \ar[r] & {\rm Isot}^r_\pi(M)_{\id, \id} \ar@{}[r]|*{\hspace*{6mm} \rotatebox{90}{$\bigcap$}} \ar@{}[d]|*{\rotatebox{0}{$\bigtriangledown$}} & 
{\cal I} \ar[r]^-R \ar@{}[d]|*{\rotatebox{0}{$\bigtriangledown$}} & 
{\rm Ker}\,P \ar[r] \ar@{}[d]|*{\rotatebox{0}{$\bigtriangledown$}} & 0 \\
0 \ar[r] & 
({\rm Ker}\,\nu_{\cal I})_{\id, \id} \ar@{}[r]|*{\hspace*{6mm} \rotatebox{90}{$\bigcap$}} & 
{\rm Ker}\,\nu_{\cal I} \ar[r]^-R & ({\rm Ker}\,P)_0 \ar[r] & 0 
}$$ 
}
\vspace*{0mm} 
{\begin{proof} 
(2) If $F \in {\rm Ker}\,\nu_{\cal I}$ , then $\mu(\underline{F}) = \nu_{\cal I}(F) = 0$ and 
Fact~\ref{fact_mu=0} (1) yields a $C^r$ isotopy of isotopies $\Psi 
: \underline{F} \simeq_\ast \id_{S^1 \times I}$ rel ends.
Then we can apply Corollary~\ref{cor 2 rel_iso_lift} to obtain 
a $C^r$ isotopy of bundle isotopies $\Phi 
: F \simeq_\ast G$ rel ends with $\Phi_t \in {\rm Isot}^r_{\pi}(M)$ $(t \in I)$ and 
$\underline{\Phi} = \Psi$. Since $\underline{G} = \id_{S^1 \times I}$ and  $G_0 = F_0 = \id_M$, $G_1 = F_1$, \break it follows that 
$G \in {\rm Ker}\,P_I$ and $H:= G^{-1}F \in ({\rm Ker}\,\nu_{\cal I})_{\id, \id}$ so that 
$F = GH \in ({\rm Ker}\,P_I)\cdot ({\rm Ker}\,\nu_{\cal I})_{\id, \id}$. 

Since ${\rm Ker}\,P_I \subset {\rm Ker}\,\nu_{\cal I}$, the converse implication is obvious. 


(3) $\bary[t]{l}
R^{-1}(({\rm Ker}\,P)_0) = R^{-1}(R({\rm Ker}\,P_I)) = ({\rm Ker}\,P_I) \cdot {\rm Ker}\,R = ({\rm Ker}\,P_I)\cdot {\rm Isot}^r_\pi(M)_{\id, \id}. \\[2mm]
({\rm Ker}\,\nu_{\cal I}) \cdot {\rm Isot}^r_\pi(M)_{\id, \id}
= ({\rm Ker}\,P_I)\cdot ({\rm Ker}\,\nu_{\cal I})_{\id, \id} \cdot {\rm Isot}^r_\pi(M)_{\id, \id}
= ({\rm Ker}\,P_I)\cdot {\rm Isot}^r_\pi(M)_{\id, \id}. 
\eary$ 
\vskip -3.5mm 
\end{proof}
}
Hence we have the group isomorphisms : \\
\hspp $\widetilde{R} : {\rm Isot}^r_\pi(M)_0/{\rm Isot}^r_\pi(M)_{\id,\id} \cong {\rm Diff}^r_\pi(M)_0$ \ \ and \ \ 
$\widetilde{R} : {\cal I}/{\rm Isot}^r_\pi(M)_{\id,\id} \cong {\rm Ker}\,P$. 

Recall the surjective group homomorphism $\nu_{\cal I} : {\cal I} \to \IZ$, which maps 
${\rm Isot}^r_\pi(M)_{\id, \id} \vartriangleleft {\cal I}$ onto the image $\nu_{\cal I}\big({\rm Isot}^r_\pi(M)_{\id, \id}\big) < \IZ$. 
It follows that $\nu_{\cal I}\big({\rm Isot}^r_\pi(M)_{\id, \id}\big) = k \IZ$ \ for a unique $k = k(\pi, r) \in \IZ_{\geq 0}$. 

For any $H \in {\rm Isot}^r_\pi(M)_0$ and $G \in {\rm Isot}^r_\pi(M)_{\id,\id}$ 
from Fact~\ref{fact_nu}\,(3) it follows that \\
\hspp $\nu(HG) = \nu(H_0G) + \nu(HG_1) = \nu(H) + \nu(G)$.  \\
Therefore, if we consider the right actions of ${\rm Isot}^r_\pi(M)_{\id,\id}$ on ${\rm Isot}^r_\pi(M)_0$ and $k\IZ$ on $\IR$ by right translation, 
then $\nu : {\rm Isot}^r_\pi(M)_0 \to \IR$ is a surjective equivariant map with respect to the homomorphism $\nu : {\rm Isot}^r_\pi(M)_{\id, \id} \to k\IZ$,  
and so, it induces the associated map \\
\hspp $\widetilde{\nu} : {\rm Isot}^r_\pi(M)_0/{\rm Isot}^r_\pi(M)_{\id,\id} \to \IR_k \equiv \IR/k\IZ$, $\widetilde{\nu}([H]) = [\nu(H)]$ \\
(cf.~Fact~\ref{fact_nu_homo}\,(2)). 
By Fact~\ref{fact_nu_homo}\,(1)(ii) 
it restricts to a surjective group homomorphism \\
\hspp $\widetilde{\nu} : {\cal I}/{\rm Isot}^r_\pi(M)_{\id,\id} \to \IZ_k \equiv \IZ/k\IZ$.  

\bdefnb From these arguments we obtain 
\bit 
\itemI a surjective map \hsf $\widehat{\nu} := \widetilde{\nu} \widetilde{R}^{-1} : {\rm Diff}^r_\pi(M)_0 \lra \IR_k$, \ \ $\widehat{\nu}(H_1) = [\nu(H)]$ \ 
$(H \in {\rm Isot}^r_\pi(M)_0)$, 
\itemII  
a surjective group homomorphism \hsf $\widehat{\nu}|_{{\rm Ker}\,P} : {\rm Ker}\,P \to \IZ_k$. 
\eit 
\edefn 

{Fact~\ref{fact_Ker} implies the following.} 

\bfact\label{lem_Ker_P} \mbox{} 
\benum
\item 
$\widehat{\nu}|_{{\rm Ker}\,P} : {\rm Ker}\,P \to \IZ_k$ has the following properties. 
\bit 
\itemI ${\rm Ker}\,\widehat{\nu}|_{{\rm Ker}\,P} = ({\rm Ker}\,P)_0$. 
\itemII It induces the group isomorphism $(\widehat{\nu}|_{{\rm Ker}\,P})^\sim : ({\rm Ker}\,P)\big/({\rm Ker}\,P)_0 \cong \IZ_k$. 
\eit 
\item The quotient group $({\rm Ker}\,P)\big/({\rm Ker}\,P)_0$ is a cyclic group. 
Its order is $k$ when $k \geq 1$ and $\infty$ when $k = 0$. 

\eenum
\efact 

In Subsection 4.4 we study the cases $k=0$ and $k \geq 1$ separately, since these two cases lead to opposite consequences. 
Subsection 4.3 includes some basic results on factorization of bundle diffeomorphisms over $S^1$. 

\subsection{Factorization of bundle diffeomorphisms over $S^1$} \mbox{} 

In this {subsection} we obtain basic results on factorization of bundle diffeomorphisms and isotopies over $S^1$. 

Suppose $\pi : M \to S^1$ is a fiber bundle with fiber $N$ and structure group $\Gamma < {\rm Diff}(N)$. 
Let $r \in \IZ_{\geq 0} \cup \{ \infty \}$. 
We use the following notations. 
For a closed arc 
$J \subset S^1$ let $J' := (S^1)_J = S^1 - {\rm Int}\,J$ 
and $\widetilde{J} := \pi^{-1}(J)$. 
Also recall the definition of the notation ${{\rm Diff}_\pi^r(M; (S^1)_J)_0}$ in {Subsection 3.3}; 
$G \in {{\rm Isot}_\pi^r(M; (S^1)_J)_0}$ means that $G \in {\rm Isot}_\pi^r(M)_0$ and 
$G$ has its support in $\widetilde{K}$ for some closed arc $K \Subset J$ 
and $g \in {{\rm Diff}_\pi^r(M; (S^1)_J)_0}$ means that $g = G_1$ for some $G \in {{\rm Isot}_\pi^r(M; (S^1)_J)_0}$. 
Note that  $g\big({{\rm Diff}_\pi^r(M; (S^1)_J})_0\big)g^{-1} = {{\rm Diff}_\pi^r(M; (S^1)_{\underline{g}(J)})_0}$ 
for any $g \in {\rm Diff}_\pi^r(M)$. 
\vskip 2mm 

\blem\label{lem_fac_iso} Suppose $K \Subset J \subset S^1$ are closed arcs and 
$F \in {\rm Isot}_\pi^r(M)_0$ satisfies $\underline{F}(K \times I) \Subset J$.  
Then there exists a factorization $F = GH$ for some 
$G \in {{\rm Isot}_\pi^r(M; (S^1)_J)_0}$ and 
$H \in {{\rm Isot}_\pi^r(M; (S^1)_{K'})_0}$. 
\elem

\begin{proof}
Take a closed arc $L \Subset J$ such that $K \Subset L$ and $\underline{F}(L \times I) \Subset J$. 
There exists $G' \in {{\rm Isot}^r(S^1; J')_0}$ such that $G' = \underline{F}$ on $L \times I$. 
Take a lift $G \in {{\rm Isot}_\pi^r(M; (S^1)_J)_0}$ such that 
$G = F$ on $\widetilde{L} \times I$ and $\underline{G} = G'$. 
Then, it follows that $F = GH$ for $H := G^{-1}F \in {{\rm Isot}_\pi^r(M; (S^1)_{K'})_0}$. 
\end{proof} 

\begin{lem}\label{lem_fac_c} 
Suppose $K, L \subset S^1$ are closed arcs with ${\rm Int}\,K \cap {\rm Int}\,L \neq \emptyset$. 
\benum 
\item[{\rm (1)}] Any $g \in {{\rm Diff}^r_\pi(M; (S^1)_{{K}})_0}$ admits a factorization 
$g = \widehat{g}\,h$ for some $\widehat{g} \in {{\rm Diff}^r_\pi(M; (S^1)_{{K}})_0^c}$ and 
$h \in {{\rm Diff}^r_\pi(M; (S^1)_{{L}})_0}$. 

\item[{\rm (2)}] For any $g \in {{\rm Diff}^r_\pi(M; (S^1)_{{K}})_0}$ and $h \in {{\rm Diff}^r_\pi(M; (S^1)_{{L}})_0}$ 
there exists a factorizaiton $gh = \widehat{g}\,\widehat{h}$ for some $\widehat{g} \in {{\rm Diff}^r_\pi(M; (S^1)_{{K}})_0^c}$ and 
$\widehat{h} \in {{\rm Diff}^r_\pi(M; (S^1)_{{L}})_0}$. 
\eenum 
\end{lem} 

\begin{proof} \mbox{} 
\benum
\item Take $G \in {{\rm Isot}^r_\pi(M; (S^1)_{{K}})_0}$ with $G_1 = g$. 
There exists a closed arc $D \Subset K$ with $G \in {{\rm Isot}^r_\pi(M; (S^1)_{{D}})_0}$. 
Take a closed arc $E \subset {\rm Int}\,K \cap {\rm Int}\,L$ and an isotopy 
$\Phi \in {{\rm Isot}^r_\pi(M; (S^1)_{{K}})_0}$ with $\Phi_1(\widetilde{D}) = \widetilde{E}$. 
Let $\phi := \Phi_1 \in {{\rm Diff}^r_\pi(M; (S^1)_{{K}})_0}$ 
and $h:=\phi g \phi^{-1} \in {{\rm Diff}^r_\pi(M; (S^1)_{{E}})_0} \subset {{\rm Diff}^r_\pi(M; (S^1)_{{L}})_0}$. 
Then, it follows that $\widehat{g} := [g, \phi] \in {{\rm Diff}^r_\pi(M; (S^1)_{{K}})_0^c}$ and 
$g = g(\phi g^{-1} \phi^{-1}) (\phi g \phi^{-1}) = \widehat{g}\,h$. 

\item The assertion directly follows from (1). 
\eenum 
\vspace*{-7mm} 
\end{proof}

\begin{lem}\label{lem_underline{F}_p} For any $F \in {\rm Isot}^r_\pi(M)$ with $\sigma := \nu(F)$ and $\underline{F}_p(0) = \pi_{S^1}(\tau)$  
there exists $G \in {\rm Isot}_\pi^r(M)$ such that $F \simeq_\ast G$ in ${\rm Isot}^r_\pi(M)$ and 
$\underline{G}_p(s) = \pi_{S^1}(\sigma s + \tau)$ \ $(s \in I)$.
\end{lem}

\bpf 
For simplicity, we use the symbol $[x] := \pi_{S^1}(x)$ $(x \in \IR)$. 
Then, $S^1$ becomes a Lie group under the addition $[x] + [y] = [x+y]$ $(x,y \in \IR)$ (i.e., $S^1 \cong \IR/\IZ$). 
Each $w \in S^1$ induces the translation $L_w$ on $S^1$, $L_w(z) = z + w$ $(z \in S^1)$. 
This defines the $C^\infty$ map $L : S^1 \to {\rm Diff}(S^1)$ and any $c \in {{\cal P}^r(S^1) \equiv C^r(I, S^1)}$ induces an isotopy 
$L_c = (L_{c(t)})_{t \in I} \in {\rm Isot}^r(S^1)$. 
As usual, for maps $f,g : X \to S^1$ the map $f + g : X \to S^1$ is defined by $(f + g)(q) = f(q) + g(q)$ $(q \in X)$. 

Consider the $C^r$ paths $a \equiv \underline{F}_p$, $b(s) = [\sigma s + \tau]$ and $c(s) = b(s) - a(s)$ in $S^1$. 
The path $c$ admits a $C^r$ homotopy $\eta : \e_0 \simeq_\ast c$ in ${{\cal P}^r(S^1)_{0,0}}$, where $\e_0 \in {{\cal P}^r(S^1)}$ is the constant path at the zero $0 = [0]$ in $S^1$. 
In fact, 
the path $a$ has a unique lift $\widetilde{a} \in {\cal P}^r(\IR)$ with $\widetilde{a}(0) = \tau$. 
It satisfies $\widetilde{a}(1) = \lambda(a) + \widetilde{a}(0) = \nu(F) +\tau = \sigma + \tau$. 
The path $b \in {\cal P}^r(S^1)$ has the canonical lift $\widetilde{b} \in {\cal P}^r(\IR)$, $\widetilde{b}(s) = \sigma s + \tau$.
Then the path $\widetilde{c} := \widetilde{b} - \widetilde{a} \in {\cal P}^r(\IR)_{0,0}$ is a lift of $c$ and it admits 
the homotopy $\xi : \widetilde{\e}_0 \simeq_\ast \widetilde{c}$, $\xi_t := t \widetilde{c}$ $(t \in I)$, where $\widetilde{\e}_0 \in {\cal P}^r(\IR)$ is the constant path at $0$. 
It induces the $C^r$ homotopy $\eta := \pi_{S^1}\,\xi : \e_0 \simeq_\ast c$. 
Note that $\eta + a := (\eta_t + a)_{t \in I} : a \simeq_\ast b$ in ${\cal P}^r(S^1)$. 

The homotopy $\eta$ induces a $C^r$ homotopy $L_\eta : I \times I \to 
{\rm Diff}^r(S^1)$, $L_\eta(s,t) = L_{\eta(s,t)}$. 
The latter is regarded as a $C^r$ homotopy $L_\eta = (L_{\eta_t})_{t \in I} : L_{\e_0} = \id \simeq_\ast L_c$ in ${\rm Isot}^r(S^1)_{\id,\id}$. 
Composing with $\underline{F}$, we have the $C^r$ homotopy 
$L_\eta \underline{F} = (L_{\eta_t}\underline{F})_{t \in I} : \underline{F} \simeq_\ast L_c\underline{F}$ in ${\rm Isot}^r(S^1)$.  
It follows that $(L_c\underline{F})_p(s) = (L_c\underline{F})_s(p) = (L_{c(s)}\underline{F}_s)(p) = L_{c(s)} (\underline{F}_p(s)) = a(s) + c(s) = b(s)$.

The homotopy $L_\eta \underline{F}$ lifts to 
a $C^r$ homotopy $\Phi : F \simeq_\ast G$ in ${\rm Isot}^r_\pi(M)$.
Then, $\underline{G}_p = (L_c\underline{F})_p = b$ as required. 
\epf 

Recall the following assumption in Introduction.  

\begin{assumption_ast}\label{assump_perfect} 
Suppose $J$ is an open interval and $\varrho : L \to J$ is a trivial fiber bundle with fiber $N$ and structure group $\Gamma$. 
Then ${\rm Diff}^r_{\varrho,c}(L)_0$ is perfect. 
\end{assumption_ast}

\begin{prop}\label{prop_clf} 
Under Assumption $(\ast)$,  
if $F \in {\rm Isot}^r_\pi(M)_0$ and $|\nu(F)| < \ell \in \IZ_{\geq 1}$, then $clb_\pi(F_1) \leq 2\ell + 1$. 
\end{prop}

\begin{proof}
Define $\sigma, \tau \in \IR$ by $\sigma := \nu(F)$ and $p = \underline{F}_p(0) = \pi_{S^1}(\tau)$.  
By Lemma~\ref{lem_underline{F}_p} 
there exists $G \in {\rm Isot}_\pi^r(M)$ such that $F \simeq_\ast G$ in ${\rm Isot}^r_\pi(M)$ and 
$\underline{G}_p(t) = \pi_{S^1}(\sigma t + \tau)$ \ $(t \in I)$.
Since $F_1 = G_1$, replacing $F$ by $G$, we may assume that $F$ itself satisfies the condition $\underline{F}_p(t) = \pi_{S^1}(\sigma t + \tau)$ \ $(t \in I)$.

Let \ $f := F_1$, \ $f_i := F_{i/\ell} \in {\rm Diff}^r_\pi(M)_0$ \ $(i=0,1, \cdots, \ell)$ \ and \ 
$f^{(i)} := f_i f_{i - 1}^{\ -1} \in {\rm Diff}^r_\pi(M)_0$ \ $(i=1, \cdots, \ell)$. \\ 
Then, $f_0 = \id_M$ \ and \hsh $f = f_\ell = f^{(\ell)}f^{(\ell-1)} \cdots f^{(i)} \cdots f^{(1)}$. \\
Identifying $I_i := [\frac{i-1}{\ell},\frac{i}{\ell}]$ with $I$, we have an isotopy 
$F^{(i)} := (F|_{M \times I_i})\big(f_{i-1}^{\ -1}\times \id_{I_i}) \in {\rm Isot}^r_\pi(M)_0$ 
with ${F^{(i)}}_1 = f^{(i)}$. 
Let $s_i := i(\sigma/\ell) + \tau$
and $p_i := \pi_{S^1}\,(s_i) \in S^1$ $(i = 0,1, \cdots, \ell)$. 
Then, $\underline{f_i}(p) = \underline{F}_{i/\ell}(p) = \pi_{S^1}(s_i) = p_i$. \\[1mm] 
Note that \ 
$\left\{ \hspace{-1mm} 
\bary[c]{ll}
\tau = s_0 < s_1 < \cdots < s_\ell & \mbox{if $\sigma > 0$,} \\[1mm]
\tau = s_0 = s_1 = \cdots = s_\ell & \mbox{if $\sigma = 0$,} \\[1mm]  
s_\ell < \cdots < s_1 < s_0 = \tau & \mbox{if $\sigma < 0$.}
\eary \right.$ \\[1.5mm]
Since $|\sigma|/\ell < 1$, there exists $\e > 0$ such that $|\sigma|/\ell + 2\e < 1$. 
We define $J_i^- \Subset J_i \subset S^1$ $(i=1,2, \cdots, \ell)$ by \\[2mm] 
\hsp $(J_i, J_i^-) := 
\left\{ \hspace{-1mm} 
\bary[c]{ll}
\big(\pi_{S^1}([s_{i-1} - \e, s_i + \e]), \pi_{S^1}([s_{i-1}, s_i])\big) & \mbox{if $\sigma \geq 0$}, \\[2mm] 
\big(\pi_{S^1}([s_i - \e, s_{i-1} + \e]), \pi_{S^1}([s_i, s_{i-1}])\big) & \mbox{if $\sigma < 0$}. 
\eary \right.$ \\[2mm] 
Then, $J_i$ is a closed arc in $S^1$, since it is the image of a closed interval in $\IR$ of width $= |\sigma|/\ell + 2\e < 1$. \\
For each $i = 1, 2, \cdots, \ell$ it follows that $\underline{F^{(i)}} = \underline{F}|_{S^1 \times I_i}\big(\underline{f_{i-1}}^{-1}\times \id_{I_i}) \in {\rm Isot}^r(S^1)_0$ and \\
\hspp $\underline{F^{(i)}}(p_{i-1} \times I_i) 
= \underline{F}(p \times I_i) 
= J_i^- \Subset J_i$. \\
Hence, there exists a small closed arc $K_i$ in $S^1$ such that \\
\hspp $p_{i-1} \in {\rm Int}\,K_i \subset K_i \Subset J_i$, \ \ 
$\underline{F^{(i)}}(K_i \times I_i) \Subset J_i$,  \\
\hspp ${\rm Int}\,J_i \cap {\rm Int}\,K_i' \neq \emptyset$ \ \ and \ \ ${\rm Int}\,J_{i-1} \cap {\rm Int}\,K_i' \neq \emptyset$ \ if $i \geq 2$.  \\
Then, we can apply Lemma~\ref{lem_fac_iso} to $F^{(i)} \in {\rm Isot}^r_\pi(M)_0$ to obtain a factorization $F^{(i)} = G^{(i)} H^{(i)}$ such that \\
\hspp $G^{(i)} \in {\rm Isot}^r_\pi(M; (S^1)_{{J_i}})_0$ \ \ and \ \ 
$H^{(i)} \in 
{\rm Isot}^r_\pi(M; (S^1)_{{K_i'}})_0$. \\[1mm] 
It follows that 
\bit 
\itemI $f^{(i)} = F^{(i)}_1 = G^{(i)}_1 H^{(i)}_1 = g^{(i)} h^{(i)}$, where \\
\hsp $g^{(i)} := G^{(i)}_{\,1} \in {\rm Diff}^r_\pi(M; (S^1)_{{J_i}})_0$ \ \ and \ \ 
$h^{(i)} := H^{(i)}_{\,1} \in 
{\rm Diff}^r_\pi(M; (S^1)_{{K_i'}})_0$,  
\itemII 
${\rm Int}\,J_i \cap {\rm Int}\,K_i' \neq \emptyset$ \ \ $(i =1, \cdots, \ell)$, \ \ 
${\rm Int}\,J_{i-1} \cap {\rm Int}\,K_i' \neq \emptyset$ \ \ \ $(i =2, \cdots, \ell)$,  
\vskip 1mm 
\itemiii 
$f = f^{(\ell)}f^{(\ell-1)} \cdots f^{(i)} \cdots f^{(1)}
= (g^{(\ell)} h^{(\ell)}) (g^{(\ell-1)} h^{(\ell-1)})\cdots (g^{(i)} h^{(i)}) \cdots (g^{(1)} h^{(1)})$.  
\eit 

Let $m := 2\ell$ and define \ $\phi_j$, $L_j$ \ $(j=1, \cdots, m)$ \ by \\[2mm] 
\hspace*{15mm} 
$\phi_j := 
\left\{ \hspace{-1mm} 
\bary[c]{ll}
g^{(i)} & (j=2i) \\[2mm]
h^{(i)} & (j=2i-1)
\eary \right.$ 
\hsp 
$L_j := 
\left\{ \hspace{-1mm} 
\bary[c]{ll}
J_i & (j=2i) \\[2mm]
K_i' & (j=2i-1)
\eary \right.$ \hsp $(i=1, \cdots, \ell)$. 
\vskip 3mm 

Then, the conditions (i) $\sim$ (iii) are rephrased as follows : 

\bit 
\itemiv $\phi_j \in {\rm Diff}^r_\pi(M; (S^1)_{{L_j}})_0$
\hsp (v) \ $f = \phi_m \phi_{m-1}\cdots \phi_1$  
\item[(vi)] ${\rm Int}\,L_j \cap {\rm Int}\,L_{j-1}\neq \emptyset$ \ $(j=2, 3, \cdots, m)$. 
\eit 

We apply Lemma~\ref{lem_fac_c} inductively (backward from $m$ to 2) to the factorization $f = \phi_m \cdots \phi_1$, \ 
which induces a new factorization \ $f = \widehat{\phi}_m \cdots \widehat{\phi}_{2}\phi_{1}'$, \\ 
where \ 
$\widehat{\phi}_j \in {\rm Diff}^r_\pi(M; (S^1)_{{L_j}})_0^c$ \ \ $(j=m, m-1, \cdots, 2)$ \ and \ 
$\phi_1' \in {\rm Diff}^r_\pi(M; (S^1)_{{L_1}})_0$. 

Under Assumption $(\ast)$, Lemma~\ref{basic_lemma_cl} induces a factorization \ $\phi_1' = \widehat{\phi}_{1}\widehat{\phi}_{0}$ \ 
for some $\widehat{\phi}_1, \widehat{\phi}_0 \in {\rm Diff}^r_\pi(M; (S^1)_{{L_{1}}})_0^c$. 

Finally, from the factorization \ $f = \widehat{\phi}_m \cdots \widehat{\phi}_{2}\widehat{\phi}_1\widehat{\phi}_0$, \ $(\widehat{\phi}_j \in {\rm Diff}^r_\pi(M; (S^1)_{L_j})_0^c, L_0 := L_1)$, 
we conclude that $clb_\pi f \leq m+1 = 2\ell+1$.  
\end{proof} 

\subsection{Boundedness of bundle diffeomorphism groups over $S^1$} \mbox{} 

Suppose $\pi : M \to S^1$ is a fiber bundle with fiber $N$ and structure group $\Gamma < {\rm Diff}(N)$ and 
$r \in \IZ_{\geq 0} \cup \{ \infty \}$. 
Let $k := k(\pi, r) \in \IZ_{\geq 0}$. 
Below we study the cases $k \geq 1$ and $k=0$ separately, since these two cases lead to opposite consequences. 

\vskip 2mm 

\noindent {\bf {[\hspace{0.2mm}I\hspace{0.2mm}]} The case $k \geq 1$} \hsp  
We assume that $(N, \Gamma, r)$ satisfies Assumption $(\ast)$. 

\bthm\label{thm_1} 
If $f \in {\rm Diff}^r_\pi(M)_0$ and $\widehat{\nu}(f) = [s] \in \IR_k$ $\big(s \in \big(-\frac{k}{2}, \frac{k}{2}\big]\big)$, 
then $clb_\pi f \leq 2[|s|]+ 3 \leq k+3$. 
\ethm

\bpf 

Since $\nu({\rm Isot}^r_\pi(M)_{\id, f}) = \widehat{\nu}(f)$ (Fact~\ref{fact_nu_homo}\,(2)), 
there exists $F \in {\rm Isot}^r_\pi(M)_{\id, f}$ with $\nu(F) = s$. 
Then, \\
$|\nu(F)| = |s| < [|s|] + 1 \equiv \ell \in \IZ_{\geq 1}$ \ and Proposition~\ref{prop_clf} implies that 
${clb_\pi f} \leq 2\ell + 1 = 2[|s|] + 3 \leq k+3$. 
\epf

\bcor\label{cor_S^1_k_geq_1} \ \ {\rm (1)} \ $clb_\pi d \,{\rm Diff}^r_\pi(M)_0 \leq k + 3$. 
\benum
\item[{\rm (2)}] 
${\rm Diff}_\pi^r(M)_0$ is {\rm (i)} uniformly simple relative to ${\rm Ker}\,P$ and {\rm (ii)} bounded. 
\eenum 
\ecor 

\noindent {\bf {[\hspace{0.2mm}I\hspace{-0.2mm}I\hspace{0.2mm}]} The case $k = 0$} 

In this case, $(\IR_k, \IZ_k) = (\IR, \IZ)$ and the arguments in {Subsection 4.2 and Fact~\ref{fact_nu_homo_pre}\,(2)}
imply the following conclusion on the map \ $\widehat{\nu} = \nu R^{-1}$.  

\bthm\label{prop_nu} \mbox{} The map $\widehat{\nu} : \big({\rm Diff}^r_\pi(M)_0, {\rm Ker}\,P) \to (\IR, \IZ)$ has the following properties. 
\benum 
\item[{\rm (1)}] The map $\widehat{\nu} : {\rm Diff}^r_\pi(M)_0 \lra \IR$ 
is a surjective quasimorphism, which 
satisfies the following estimate: \\
\hspace*{20mm} 
$(\ast)$ \ \ $|\widehat{\nu}(fg) - \widehat{\nu}(f) - \widehat{\nu}(g)| < 1$ \ for any $f, g \in {\rm Diff}^r_\pi(M)_0$. 

\item[{\rm (2)}] The map $\widehat{\nu} : {\rm Ker}\,P  \to \IZ$ is a surjective group homomorphism. 
\eenum
\ethm 

{By Fact~\ref{fact_qm}\,(3) we have the following conclusions.

\bcor\label{cor_unbdd} The group ${\rm Diff}^r_\pi(M)_0$ is unbounded and not uniformly perfect. 
\ecor
}

\subsection{Attaching maps} \mbox{} 

In this {subsection} we describe the relation between the integer $k(\pi, r)$ and the attaching map of $\pi$ for a fiber bundle $\pi$ over $S^1$. 
We regard $S^1 = \IR/\IZ$. Suppose $N$ is a $C^\infty$ manifold and $\Gamma < {\rm Diff}^\infty(N)$. 
For $\phi \in \Gamma$ the mapping torus $\pi_\phi : M_\phi \to S^1$  
is obtained from the product $N \times I$ by identifying $N \times \{ 0 \}$ and $N \times \{ 1 \}$ by $\phi$. 
Here we represent it as the quotient space of $N \times \IR$ under the equivalence relation : 
\hsp $(x,s) \sim_\phi (y,t)$ \LLRA $(y,t) = (\phi^{-n}(x), s+n)$ \ for some $n \in \IZ$ \hsh $((x,s), (y,t) \in N \times \IR)$; \\[1mm] 
\hspace*{30mm} $M_\phi := (N \times \IR)/\sim_\phi$, \hsh $\pi_\phi([x,s]) := [s]$ \ $([x,s] \in M_\phi)$. \\
The $C^\infty$ structure on $M_\phi$ is defined so that the quotient map $\rho_\phi : N \times \IR \to M_\phi$ is a $C^\infty$ covering projection. 
Then, $\pi_\phi : M_\phi \to S^1$ is a $C^\infty$ $(N, \Gamma)$-bundle 
with the canonical maximal $(N, \Gamma)$-atlas induced by the map $\rho_\phi : N \times \IR \to M_\phi$. 
We have the following pullback diagrams : \\[1mm] 
\hspace*{26mm}
$(\ast)_\phi$ \hsh $\bary[c]{c@{ \ }c@{ \ }c@{ \ }c@{ \ }cll}
& & & \rho_\phi & & & \\[-1.5mm] 
N \times I & \subset & N \times \IR & \raisebox{0.5mm}{\makebox(25,0){\rightarrowfill}} & M_\phi & & \\[0mm] 
\rotatebox{90}{$\lla$} & \hspace{-25mm} \smash{\raisebox{2mm}{$\pi_N'$}} 
& \rotatebox{90}{$\lla$} & \hspace{-30mm} \smash{\raisebox{2mm}{$\pi_N$}} 
& \rotatebox{90}{$\lla$} & \hspace{-5mm} \smash{\raisebox{2mm}{$\pi_\phi$}} & \\[-0.8mm] 
I & \subset & \IR & \raisebox{0.5mm}{\makebox(25,0){\rightarrowfill}} & S^1 & & \\[-2mm] 
& & & \pi_{S^1} & & & 
\eary$ \hsh 
$\bary[c]{l}
\pi_N(x,s) = s, \hsh \pi_N'(x,s) = s, \\[2mm] 
\rho_\phi(x,s) = [x,s], \hsh 
\pi_{S^1}(s) = [s]. 
\eary$ \\[1.5mm] 
Here we regard $\pi_N$ and $\pi_N'$ as the product $(N, \Gamma)$-bundles. 

\bfact\label{fact_mapping-torus-1} \mbox{} 
By taking a pull-back by the map $\pi_{S^1}: \IR \to S^1$, it is seen that 
any $(N, \Gamma)$-bundle $\pi : M \to S^1$ 
is isomorphic to a mapping torus $\pi_\phi$ for some $\phi \in \Gamma$. 
\efact 

Consider the mapping tori $\pi_\phi : M_\phi \to S^1$ and $\pi_\psi : M_\psi \to S^1$ associated to $\phi, \psi \in \Gamma < {\rm Diff}^\infty(N)$. 
Let $r \in \IZ_{\geq 0} \cup \{ \infty \}$. 
Below we are concerned with the subgroup \ ${\rm Isot}^r_\Gamma(N) := \{ F \in {\rm Isot}^r(N) \mid F_t \in \Gamma \ (t \in I) \}$ \\ 
and the subsets 
\hspace*{0mm} 
\btab[t]{l}
${\rm Diff}^r(\pi_{\phi}, \pi_{\psi}; \id_{S^1}) := \{ f \in {\rm Diff}^r(\pi_{\phi}, \pi_{\psi}) \mid \underline{f} = \id_{S^1}\}$ \hsh and \\[2mm]
${\rm Isot}^r(\pi_{\phi}, \pi_{\psi}; \id_{S^1}) := \{ F \in {\rm Isot}^r(\pi_{\phi}, \pi_{\psi}) \mid \underline{F}_t = \id_{S^1} \ (t \in I)\}$. 
\etab \\[2mm] 
The meaning of similar notations are obvious from these examples.

\bfact\label{fact_mapping-torus-2} \mbox{} 
\benum
\item Any isotopy $F \in {\rm Isot}^r_\Gamma(N)_{\psi, \phi}$ 
induces $G \in {\rm Diff}^r(\pi_\phi, \pi_\psi; \id_{S^1})$ defined by \\
\hspp $G : M_\phi \to M_\psi$ : $G([x,s]) = [F_{\alpha(s)}(x), s]$, \\
where $\alpha \in C^r(I, I)$ such that $\alpha(t) = 0$ $(t \in [0, \e))$ and $\alpha(t) = 1$ $(t \in (1- \e, 1])$ for some $\e \in (0,1/2)$. \\
In particular, if $\phi$ admits a $C^r$ isotopy $\id_N \simeq \phi$ in $\Gamma$, then
$\pi_\phi$ is $C^r$ isomorphic to the product bundle $\pi_{\id_N} : M_{\id_N} \to S^1$ as $(N, \Gamma)$-bundles. 

\item If $\chi \in \Gamma$ and $\psi \chi = \chi \phi$, then it induces the map 
$\widehat{\chi} \in {\rm Diff}^r(\pi_\phi, \pi_\psi; \id_{S^1})$ defined by \\
\hspp $\widehat{\chi} : M_\phi \to M_\psi$ : $\widehat{\chi}([x,s]) = [\chi(x), s]$. 

\item The attaching map $\phi$ induces the map 
$\widehat{\phi} \in {\rm Diff}^\infty(\pi_\phi, \pi_\phi; \id_{S^1})$ defined by \\
\hspp $\widehat{\phi} : M_\phi \to M_\phi$ : $\widehat{\phi}([x,s]) = [\phi(x), s]$. \\
It admits the isotopy 
$\Phi \in {\rm Isot}^\infty_{\pi_\phi}(M_\phi)_{\id, \widehat{\phi}} = {\rm Isot}^\infty(\pi_\phi, \pi_\phi)_{\id, \widehat{\phi}}$ defined by \\
\hspp $\Phi_t : M_\phi \to M_\phi$ : $\Phi_t([x,s]) = [x, s+t]$ \hsh $(t \in I)$ \hsp (the $2\pi$ rotation over $S^1$) \\
Since {$\zeta \equiv \underline{\Phi} \in {\rm Isot}^\infty(S^1)_{\id, \id}$ is the $2\pi$ rotation of $S^1$}, we have $\nu(\Phi) = 1$. 

\item Let $k = k(\pi_\phi,r)$ and recall the surjective group homomorphism $\nu : {\rm Isot}^r_{\pi_\phi}(M_\phi)_{\id, \id} \lra k\IZ$. 
For $\ell \in \IZ$ the following conditions are equivalent. 
\bit 
\itemI There exists $\Psi \in {\rm Isot}^r_{\pi_\phi}(M_\phi)_{\id, \id}$ with $\nu(\Psi) = \ell$ \ \ (i.e., $\ell \in k\IZ)$. 
\itemII There exists $\chi \in {\rm Isot}^r_{\pi_\phi}(M_\phi)_{\id, \widehat{\phi}^\ell}$ with $\underline{\chi} = \id_{S^1 \times I}$ 
\ \ (i.e., ${\rm Isot}^r(\pi_\phi, \pi_\phi; \id_{S^1})_{\id, \widehat{\phi}^\ell} \neq \emptyset$).  
\eit 
\eenum 
\efact 

\bpf (4) First note that $\Phi^\ell \in {\rm Isot}^r_{\pi_\phi}(M_\phi)_{\id, \widehat{\phi}^\ell}$ and $\underline{\Phi}^\ell = \zeta^\ell$. 

(i) $\Rightarrow$ (ii) : 
By { Lemmas~\ref{lem_underline{F}_p} and ~\ref{lem_sdr}} we can isotope $\Psi$ relative to ends to obtain $\Psi' \in {\rm Isot}^r_{\pi_\phi}(M_\phi)_{\id, \id}$ with $\underline{\Psi'} = \zeta^\ell$. 
Then, it follows that 
$\chi := (\Psi')^{-1}\Phi^\ell \in {\rm Isot}^r_{\pi_\phi}(M_\phi)_{\id, \widehat{\phi}^\ell}$ and 
$\underline{\chi} = (\zeta^\ell)^{-1} \zeta^\ell = \id_{S^1 \times I}$. 

(ii) $\Rightarrow$ (i) : Since $\nu(\chi) = 0$, we can take $\Psi := \Phi^\ell\chi^{-1}$.
\epf

Fact~\ref{fact_mapping-torus-2}\,(4)\mbox{} leads us to study the set ${\rm Isot}^r(\pi_\phi, \pi_\phi; \id_{S^1})_{\id, \widehat{\phi}^\ell}$ more closely. 
In Fact~\ref{fact_mapping-torus-3} - \ref{fact_mapping-torus-4} 
we clarify basic relations among ${\rm Isot}^r(\pi_\phi, \pi_\psi; \id_{S^1})$, ${\rm Isot}^r(\pi_N, \pi_N; \id_{\IR})$ and ${\rm Isot}^r(\pi_N', \pi_N'; \id_I)$, 
which also ensure the well-definedness and smoothness of the maps defined in Fact~\ref{fact_mapping-torus-2}\,(1) $\sim$ (3). 



Note that each $g \in {\rm Diff}^r(\pi_N, \pi_N; \id_{\IR})$ is interpreted as 
a $C^r$ family $g_s \in \Gamma$ $(s \in \IR)$ defined by $(g_s(x), s) = g(x,s)$ $(x \in N)$.
A similar remark is applied to ${\rm Diff}^r(\pi_N', \pi_N'; \id_{I}) (= {\rm Isot}^r_\Gamma(N))$. 

\bfact\label{fact_mapping-torus-3} \mbox{} 
There exists a bijection 
\ \ ${\rm Diff}^r(\pi_N', \pi_N'; \id_I; \phi, \psi) \ni h \lmt \widehat{h} \in {\rm Diff}^r(\pi_{\phi}, \pi_{\psi}; \id_{S^1})$, \\
\hsp where 
\btab[t]{@{\ }c@{ \ }l}
(i) & ${\rm Diff}^r(\pi_N', \pi_N'; \id_I; \phi, \psi) := \{ h \in {\rm Diff}^r(\pi_N', \pi_N'; \id_I) \mid h : (\sharp) \}$ \\[2mm] 
& \hsh $(\sharp)$ \ $\psi h_1 = h_0 \phi$ \ and \ the concatenation $(h_s)_{s \in I} \ast (\psi^{-1}h_s \phi)_{s \in I}$ is a $C^r$ family in $\Gamma$. \\[2mm] 
(ii) & $\widehat{h}$ is defined by \ $\widehat{h}([x,s]) = [h_s(x), s]$ $((x, s) \in N \times I)$. 
\etab 
\vskip 2mm 
This bijection is obtained as the composition of two bijections defined in (1) and (2) below. 
\benum
\item 
For any $g \in {\rm Diff}^r(\pi_N, \pi_N; \id_{\IR})$ the following conditions are equivalent. \\
\hsp $(\flat \,1)$ \  $\rho_\psi g = \widehat{g} \rho_\phi$ for some (unique) $\widehat{g} \in {\rm Diff}^r(\pi_{\phi}, \pi_{\psi}; \id_{S^1})$. \\
\hsp $(\flat\,2)$ \ $(x,s) \sim_\phi (y,t)$ \LLRA $g(x,s) \sim_\psi g(y,t)$ \hsh $(\forall (x,s), (y,t) \in N \times \IR)$ \\
\hsp $(\flat\,3)$ \ $\psi g_{s+1} = g_s \phi$ \hsh $(\forall s \in \IR)$ 

\bit 
\itemI From the pullback diagrams $(\ast)_\phi$ and $(\ast)_\psi$ it is seen that  
for any $f \in {\rm Diff}^r(\pi_{\phi}, \pi_{\psi}; \id_{S^1})$ 
there exists a unique $\widetilde{f} \in {\rm Diff}^r(\pi_N, \pi_N; \id_{\IR})$ with $\rho_\psi \widetilde{f} = f \rho_\phi$. 
\itemII 
Hence we have the bijection \ \ 
${\rm Diff}^r(\pi_N, \pi_N; \id_{\IR}; \phi, \psi) \ni g \lmt \widehat{g} \in {\rm Diff}^r(\pi_{\phi}, \pi_{\psi}; \id_{S^1})$, \\
where \ ${\rm Diff}^r(\pi_N, \pi_N; \id_{\IR}; \phi, \psi) := \{ g \in {\rm Diff}^r(\pi_N, \pi_N; \id_{\IR}) \mid \psi g_{s+1} = g_s \phi \ (s \in \IR) \}$. \\
Its inverse is given by $f \lmt \widetilde{f}$. 
\eit 

\item Consider the restriction map : \ \  
${\rm Diff}^r(\pi_N, \pi_N; \id_{\IR}) \lra {\rm Diff}^r(\pi_N', \pi_N'; \id_I)$ : $g \lmt g' := g|_{N \times I}$. 

\bit 
\itemI The recursive condition $(\flat\,3)$ leads to the following conclusion : 
for any $h \in {\rm Diff}^r(\pi_N', \pi_N'; \id_I)$ \\
\hsp $h = g'$ for some $g \in {\rm Diff}^r(\pi_N, \pi_N; \id_{\IR}; \phi, \psi)$ iff $h$ satisfies the condition $(\sharp)$. 

\itemII We have the bijection \hsh 
${\rm Diff}^r(\pi_N, \pi_N; \id_{\IR}; \phi, \psi) \ni g \lmt g' \in {\rm Diff}^r(\pi_N', \pi_N'; \id_I; \phi, \psi)$. 

\eit 

\item Consider the set ${\rm Diff}^r(\pi_N', \pi_N'; \id_I; \phi, \psi)' := \{ h \in {\rm Diff}^r(\pi_N', \pi_N'; \id_I) \mid \psi h_1 = h_0 \phi \}$. 

\bit 
\itemI For example, $h \in {\rm Diff}^r(\pi_N', \pi_N'; \id_I; \phi, \psi)$ if 
$h \in {\rm Diff}^r(\pi_N', \pi_N'; \id_I; \phi, \psi)'$ and 
$h$ is end stationary, that is, $h_s = h_0$ $(s \in [0,\e])$, $h_s = h_1$ \ $(s \in [1-\e,1])$ for some $\e \in (0, 1/2)$. 

\itemII If $h \in {\rm Diff}^r(\pi_N', \pi_N'; \id_I; \phi, \psi)'$, then 
$h_\alpha := (h_{\alpha(s)})_{s \in I} \in {\rm Diff}^r(\pi_N', \pi_N'; \id_I; \phi, \psi)'$ is end stationary 
for any $\alpha \in C^r(I, I)$ such that $\alpha(t) = 0$ $(t \in [0, \e))$ and $\alpha(t) = 1$ for some $\e \in (0, 1/2)$. 
\eit 
\eenum 
\efact 

Fact~\ref{fact_mapping-torus-3} derives the corresponding conclusions for isotopies. 
Note that each $G = (G_t)_{t \in I} \in {\rm Isot}^r(\pi_N, \pi_N; \id_{\IR})$ is interpreted as 
a $C^r$ family of isotopies $G_{\ast, s} := ((G_t)_s)_{t \in I} \in {\rm Isot}^r_\Gamma(N)$ $(s \in \IR)$. 
A similar remark is applied to ${\rm Isot}^r(\pi_N', \pi_N'; \id_{I})$. 

\bfact\label{fact_mapping-torus-4} \mbox{} 
There exists a bijection 
\ \ ${\rm Isot}^r(\pi_N', \pi_N'; \id_I; \phi, \psi) \ni H \lmt \widehat{H} \in {\rm Isot}^r(\pi_{\phi}, \pi_{\psi}; \id_{S^1})$, \\
\hsp where 
\btab[t]{@{\ }c@{\ }l}
(i) & ${\rm Isot}^r(\pi_N', \pi_N'; \id_I; \phi, \psi) := \{ H \in {\rm Isot}^r(\pi_N', \pi_N'; \id_I) \mid H : (\sharp)_{\rm isot} \}$ \\[2mm] 
& $(\sharp)_{\rm isot}$ \ 
\btab[t]{l}
$\psi H_{\ast, 1} = H_{\ast, 0} \phi$ \ \ (i.e., $\psi (H_t)_1 = (H_t)_0 \phi$ \ $(t \in I)$) \ \ and \\[2mm] 
the concatenation $(H_{\ast,s})_{s \in I} \ast (\psi^{-1}H_{\ast,s} \phi)_{s \in I}$ is a $C^r$ family of isotopies in ${\rm Isot}^r_\Gamma(N)$.
\etab \\[8mm] 
(ii) & $\widehat{H} = \big( \widehat{H}_t\big)_{t \in I}$ is defined by \ $\widehat{H}_t := \widehat{H_t}$ \ \ 
(i.e., $\widehat{H}_t([x,s]) = [(H_t)_s(x), s]$ $((x, s) \in N \times I)$. 
\etab 
\vskip 2mm 

\benum 
\item This bijection is obtained as the composition of two bijections : \\[0.5mm] 
\hspace*{10mm} $\bary[t]{ccccc@{}}
{\rm Isot}^r(\pi_N', \pi_N'; \id_I; \phi, \psi) & \lla & {\rm Isot}^r(\pi_N, \pi_N; \id_{\IR}; \phi, \psi) & \lra 
& {\rm Isot}^r(\pi_{\phi}, \pi_{\psi}; \id_{S^1}) \\[2mm] 
G' := ((G_t)')_{t \in I} & \smash{\raisebox{2mm}{\rotatebox{180}{$\lmt$}}} & G = (G_t)_{t \in I} & \lmt & 
\widehat{G} := \big(\hspace{0mm}\widehat{\hspace{0.5mm}G_t\hspace{0.5mm}}\hspace{0mm}\big)_{t \in I}
\eary$. \\[2.5mm] 
Here  
${\rm Isot}^r(\pi_N, \pi_N; \id_{\IR}; \phi, \psi) 
:= \{ G \in {\rm Isot}^r(\pi_N, \pi_N; \id_{\IR}) \mid \psi\, G_{\ast, s+1} = G_{\ast, s} \,\phi \ (s \in \IR) \}$.  

\vskip 0mm 
\item[] When $\phi = \psi$, these bijections are group isomorphisms. 

\item Consider the set ${\rm Isot}^r(\pi_N', \pi_N'; \id_I; \phi, \psi)' 
:= \{ H \in {\rm Isot}^r(\pi_N', \pi_N'; \id_I) \mid \psi H_{\ast, 1} = H_{\ast, 0} \phi \}$. 

\bit 
\itemI $H \in {\rm Isot}^r(\pi_N', \pi_N'; \id_I; \phi, \psi)$ if 
$H \in {\rm Isot}^r(\pi_N', \pi_N'; \id_I; \phi, \psi)'$ and 
$H$ is end stationary, that is, $H_{\ast, s} = H_{\ast, 0}$ $(s \in [0,\e])$, $H_{\ast, s} = H_{\ast, 1}$ \ $(s \in [1-\e,1])$ for some $\e \in (0, 1/2)$. 

\itemII If $H \in {\rm Isot}^r(\pi_N', \pi_N'; \id_I; \phi, \psi)'$, then 
$H_{\ast, \alpha} := (H_{\ast, \alpha(s)})_{s \in I} \in {\rm Isot}^r(\pi_N', \pi_N'; \id_I; \phi, \psi)'$ is end stationary 
for any $\alpha \in C^r(I, I)$ such that $\alpha(t) = 0$ $(t \in [0, \e))$ and $\alpha(t) = 1$ for some $\e \in (0, 1/2)$. 
\eit 
\eenum 
\efact 

When $\phi = \psi$, we omit the symbol $\psi$ from any notations defined in Fact~\ref{fact_mapping-torus-4} or used below.   

In order to interpret conditions on bundle isotopies to those on families of isotopies on fibers, 
we introduce the following notations. 

\begin{notation}
Let $C^r(I, {\rm Isot}^r_\Gamma(N))$ denote the set of $C^r$ families $\Phi$ of isotopies $\Phi_s \in {\rm Isot}^r_\Gamma(N)$ $(s \in I)$.
For any $\xi \in \Gamma$ we have its subset \ $C^r(I, {\rm Isot}^r_\Gamma(N)_{\id, \xi}; \phi, \psi) := 
\{ \Phi \in C^r(I, {\rm Isot}^r_\Gamma(N)_{\id, \xi}) \mid \psi \,\Phi_1 = \Phi_0\,\phi \}$. \\
There exist canonical identifications \ $C^r(I, {\rm Isot}^r_\Gamma(N)) = {\rm Isot}^r(\pi_N', \pi_N'; \id_I)$ \ and \\
\hspp $C^r(I, {\rm Isot}^r_\Gamma(N)_{\id, \xi}; \phi, \psi) = {\rm Isot}^r(\pi_N', \pi_N'; \id_I; \phi, \psi)_{\id_{N \times I},\xi \times \id_I}'$. 
\end{notation}

Now we concentrate on the mapping torus $\pi_\phi : M_\phi \to S^1$. Let $k = k(\pi_\phi, r) \in \IZ_{\geq 0}$. 

\bprop\label{prop_attaching_map} \mbox{} 
\benum
\item[{\rm (1)}] There exists an isotopy 
$\chi \in {\rm Isot}^r_{\pi_\phi}(M_\phi)_{\id, \widehat{\phi}^k}$ with $\underline{\chi} = \id_{S^1 \times I}$. 
\bit 
\itemI This isotopy $\chi$ is restricted on any fiber of $M_\phi$ to yield a $C^r$ isotopy $\id_N \simeq \phi^k$ in $\Gamma$. 
\itemII For any connected component $L$ of $N$, we have $\phi^k(L) = L$. 
Hence, $\ell|k$ \ if $\ell \in \IZ_{\geq 1}$ satisfies $\phi^\ell(L) = L$ and $\phi^i(L) \neq L$ $(i = 1, 2,\cdots, \ell-1)$. 
\eit 
\item[{\rm (2)}] For $\ell \in \IZ$ we have \ \ $\ell \in k\IZ$ \LLRA $C^r(I, {\rm Isot}^r_\Gamma(N)_{\id, \phi^\ell}; \phi) \neq \emptyset$, \\
\hsh $($that is, there is a $C^r$ family of isotopies 
$H_{\ast, s} \in {\rm Isot}^r_\Gamma(N)_{\id, \phi^\ell}$ $(s \in I)$ with $\phi H_{\ast, 1} = H_{\ast, 0} \phi$.$)$ 
\bit 
\itemI For example, if $\phi^\ell = \id_N$, then $\ell \in k\IZ$, since we can take $H_{\ast, s} = \id$. 
\itemII $k \geq 1$ \LLRA 
\btab[t]{@{}l}
There is $\ell \in \IZ_{\geq 1}$ which admits a $C^r$ family of isotopies $H_{\ast, s} \in {\rm Isot}^r_\Gamma(N)_{\id, \phi^\ell}$ $(s \in I)$ \\[2mm]
with $\phi H_{\ast, 1} = H_{\ast, 0} \phi$. 
\etab 
\eit 
\eenum
\eprop

\bpfb 
(1) The existence of $\chi$ follows from Fact~\ref{fact_mapping-torus-2}\,(4). 

\benum
\item[(2)] Fact~\ref{fact_mapping-torus-4} implies the following conclusion. 
\bit 
\itema There exists a bijection : \\
\hsp ${\rm Isot}^r(\pi_N', \pi_N'; \id_I; \phi, \phi)_{\id_{N \times I},\phi^\ell \times \id_I} \ni H \lmt 
\widehat{H} \in {\rm Isot}^r(\pi_{\phi}, \pi_{\phi}; \id_{S^1})_{\id, \widehat{\phi}^\ell}$. 
\itemb 
${\rm Isot}^r(\pi_N', \pi_N'; \id_I; \phi, \phi)_{\id_{N \times I},\phi^\ell \times \id_I} \neq \emptyset$ 
\LLRA ${\rm Isot}^r(\pi_N', \pi_N'; \id_I; \phi, \phi)_{\id_{N \times I},\phi^\ell \times \id_I}' \neq \emptyset$
\eit 
Hence, by Fact~\ref{fact_mapping-torus-2}\,(4) it is seen that \ 
$\ell \in k\IZ$ \LLRA $C^r(I, {\rm Isot}^r_\Gamma(N)_{\id, \phi^\ell}; \phi) \neq \emptyset$. 
\eenum 
\vskip -8mm 
\epf

We can use the mapping torus construction to provide some examples of fiber bundles $\pi : M \to S^1$ with $k(\pi, r) = 0$. 

\bexp\label{exp_k=0} 
Consider the torus $T^2 = \IR^2/\IZ^2$. 
It is known that $\pi_0({\rm Diff}^r(T^2)) \cong SL(2, \IZ)$. 
In fact, each matrix $A \in SL(2, \IZ)$ induces 
a diffeomorphism $\phi = \phi_A \in {\rm Diff}^\infty(T^2)$ defined by $\phi_A[x,y] = [(x,y)A]$ ($[x,y] \in T^2$). 
It determines the mapping torus $\pi_A : M_A \to S^1$, which is a locally trivial bundle with fiber $T^2$. 
Let $k := k(\pi_A, r) \in \IZ_{\geq 0}$ and let $\ell \in \IZ_{\geq 1} \cup \{ \infty \}$ denote the order of $A$ in $SL(2, \IZ)$. \\
\hsh {\bf Claim.} $k = 0$ if $\ell = \infty$ and $k = \ell \geq 1$ if $\ell < \infty$. \\
This follows from the following observations. 
\bit 
\itemI Proposition~\ref{prop_attaching_map}\,(1) implies $\phi_{A^k} = (\phi_A)^k \simeq \id_{T^2}$ and $A^k = E_2$. 
Hence, (a) if $\ell = \infty$, then $k = 0$ and (b) if $\ell < \infty$, then $k \in \ell \IZ$. 
\itemII When $\ell < \infty$, we have $(\phi_A)^\ell = \phi_{A^\ell} = \phi_{E_2} = \id_{T^2}$ and $\ell \in k\IZ$  
by Proposition~\ref{prop_attaching_map}\,(2). 
\eit 
\vskip 1mm 
For example, the matrix $A = 
\left(\hspace{-1mm} \bary[c]{c@{ \ \ }c}
1 & 1 \\[1mm]
0 & 1 
\eary \hspace{-1mm}\right) \in SL(2, \IZ)$ has the infinite order, since 
$A^n = 
\left(\hspace{-1mm} \bary[c]{c@{ \ \ }c}
1 & n \\[1mm]
0 & 1 
\eary \hspace{-1mm}\right)$ for $n \in \IZ$. 
\eexp 

For fiber products of fiber bundles over $S^1$ with the product structure groups, we have the following results. 
Consider a $(N, \Gamma)$ fiber bundle $\pi_\phi : M_\phi \to S^1$ and a $(L, \Lambda)$ fiber bundle $\pi_\psi : M_\psi \to S^1$. 
There exists a group monomorphism \ \ $\iota : \Gamma \times \Lambda \ni (\phi, \psi) \lmt \phi \times \psi \in {\rm Diff}^r(N \times L)$, \\ 
where $(\phi \times \psi)(x,y) = (\phi(x), \psi(y))$ $((x,y) \in N \times L)$. 
This defines the subgroup $\iota(\Gamma \times \Lambda) < {\rm Diff}^r(N \times L)$. 
Then we have the $(N \times L, \iota(\Gamma \times \Lambda))$ fiber bundle $\pi_{\phi \times \psi} : M_{\phi \times \psi} \to S^1$. 
The symbol ${\rm lcm}(m,n)$ represents the least common multiple of $m, n \in \IZ_{\geq 1}$. 

\bprop\label{prop_fiber_products} $k(\pi_{\phi \times \psi},r) \geq 1$ iff $k(\pi_{\phi},r) \geq 1$ and $k(\pi_{\psi},r) \geq 1$. 
In this case, we have \\
\hspace*{20mm} $k(\pi_{\phi \times \psi},r) = {\rm lcm}(k(\pi_{\phi},r), k(\pi_{\psi},r))$. 
\eprop 

Note that $k(\pi_{\phi \times \psi},r) = 0$ iff $k(\pi_{\phi},r) = 0$ or $k(\pi_{\psi},r) = 0$. 

\bpfb 
First we recall basic properties of compositions and Cartesian products of $C^r$ families of isotopies. 
\benum
\item ${\rm Isot}^r_\Gamma(N)$ is a group under the composition $(\phi_t)_{t \in I}(\phi_t')_{t \in I} := (\phi_t \phi_t')_{t \in I}$. 
The power is given by $\big((\phi_t)_{t \in I}\big)^n = (\phi_t^{\, n})_{t \in I}$ $(n \in \IZ)$. 
Similarly, $C^r(I, {\rm Isot}^r_\Gamma(N))$  is a group under the composition \\
$(\Phi_s)_{s \in I}(\Phi'_s)_{s \in I} := (\Phi_s\Phi'_s)_{s \in I}$.
The power is given by $\big((\Phi_{s})_{s \in I}\big)^n = (\Phi_{s}^{\ n})_{s \in I}$ $(n \in \IZ)$. 

\item Any pair of $(\phi_t)_{t \in I} \in {\rm Isot}^r_\Gamma(N)$ and $(\psi_t)_{t \in I} \in {\rm Isot}^r_\Lambda(L)$ 
determines their product \\
\hsppp $(\phi_t)_{t \in I} \times (\psi_t)_{t \in I} := (\phi_t \times \psi_t)_{t \in I} \in {\rm Isot}^r_{\iota(\Gamma \times \Lambda)}(N \times L)$. \\
Any $\chi \in {\rm Isot}^r_{\iota(\Gamma \times \Lambda)}(N \times L)$ is represented as $\chi = \Phi \times \Psi$ for 
a unique pair of $\Phi \in {\rm Isot}^r_\Gamma(N)$ and $\Psi \in {\rm Isot}^r_\Lambda(L)$. 
This yields the group isomorphism \\
\hsppp ${\rm Isot}^r_\Gamma(N) \times {\rm Isot}^r_\Lambda(L) \ni (\Phi, \Psi) \lmt \Phi \times \Psi \in {\rm Isot}^r_{\iota(\Gamma \times \Lambda)}(N \times L)$. \\
This induces the group isomorphism \\
\hsp $C^r(I, {\rm Isot}^r_\Gamma(N)) \times C^r(I, {\rm Isot}^r_\Lambda(L)) \ni (\Phi, \Psi) \lmt \Phi \times \Psi \in 
C^r(I, {\rm Isot}^r_{\iota(\Gamma \times \Lambda)}(N \times L))$. \\
Here, for $\Phi = (\Phi_{s})_{s \in I}$ and $\Psi = (\Psi_{s})_{s \in I}$ their product $\Phi \times \Psi$ is defined 
by $\Phi \times \Psi = (\Phi_{s} \times \Psi_{s})_{s \in I}$.   

\item Take any $\xi \in \Gamma$ and $\eta \in \Lambda$. 
If $\Phi \in C^r(I, {\rm Isot}^r_\Gamma(N)_{\id, \xi}; \phi)$, then $\Phi^n \in C^r(I, {\rm Isot}^r_\Gamma(N)_{\id, \xi^n}; \phi)$ $(n \in \IZ)$. 
The group isomorphisms in (2) are restricted to the bijections of the following subsets. \\ 
\hsp ${\rm Isot}^r_\Gamma(N)_{\id, \xi} \times {\rm Isot}^r_\Lambda(L)_{\id, \eta} \lra {\rm Isot}^r_{\iota(\Gamma \times \Lambda)}(N \times L)_{\id, \xi \times \eta}$, \\
\hsp $C^r(I, {\rm Isot}^r_\Gamma(N)_{\id, \xi}; \phi) \times C^r(I, {\rm Isot}^r_\Lambda(L)_{\id, \eta} ; \psi) \lra 
C^r(I, {\rm Isot}^r_{\iota(\Gamma \times \Lambda)}(N \times L)_{\id, \xi \times \eta}; \phi \times \psi)$. 
\eenum

Let $k := k(\pi_{\phi},r)$, $\ell := k(\pi_{\psi},r)$ and $m := k(\pi_{\phi \times \psi},r)$. 
Put $n = {\rm lcm}(k, \ell)$, when $k, \ell \geq 1$. 
Proposition~\ref{prop_attaching_map}\,(2) yields the following conclusions. 

For $\ell \in \IZ$ we have \ \ $\ell \in k\IZ$ \LLRA $C^r(I, {\rm Isot}^r_\Gamma(N)_{\id, \phi^\ell}; \phi) \neq \emptyset$

(4) First we suppose $k \geq 1$ and $\ell \geq 1$ and take $k', \ell' \in \IZ_{\geq 1}$ with $n = kk' = \ell \ell'$. 
Then there exist $\Phi \in C^r(I, {\rm Isot}^r_\Gamma(N)_{\id, \phi^k}; \phi)$ and 
$\Psi \in C^r(I, {\rm Isot}^r_\Lambda(L)_{\id, \psi^\ell} ; \psi)$, 
which induce \\
\hspp $\Phi^{k'} \in C^r(I, {\rm Isot}^r_\Gamma(N)_{\id, \phi^n}; \phi)$,  
$\Psi^{\ell'} \in C^r(I, {\rm Isot}^r_\Lambda(L)_{\id, \psi^n} ; \psi)$ and \\
\hspp $\Phi^{\ell'} \times \Psi^{m'} \in C^r(I, {\rm Isot}^r_{\iota(\Gamma \times \Lambda)}(N \times L)_{\id, \phi^n \times \psi^n}; \phi \times \psi)$. \\
Hence, we have $n \in m\IZ$ and $1 \leq m \leq {\rm lcm}(k, \ell)$.

(5) Conversely, suppose $m \geq 1$.  
Since $C^r(I, {\rm Isot}^r_{\iota(\Gamma \times \Lambda)}(N \times L)_{\id, \phi^m \times \psi^m}; \phi \times \psi) \neq \emptyset$, 
by (3) we see \\ 
\hspp $C^r(I, {\rm Isot}^r_\Gamma(N)_{\id, \phi^m}; \phi) \neq \emptyset$ \ \ and \ \ 
$C^r(I, {\rm Isot}^r_\Lambda(L)_{\id, \psi^m} ; \psi) \neq \emptyset$. \\
This means that $m \in k \IZ \cap \ell \IZ$. 
Hence, we have $k, \ell \geq 1$ and $m \geq {\rm lcm}(k, \ell)$. 

These observations complete the proof. 
\epf 

\subsection{Attaching maps --- the principal bundle case} \mbox{} 

In the case of principal $\Gamma$ bundles, Proposition~\ref{prop_attaching_map} can be 
described in the term of paths in $\Gamma$. 
Suppose $\Gamma$ is a Lie group (with the unit element $e$). 
Let $\Gamma_e$ denote the connected component of $\Gamma$ including $e$. 
For $a \in \Gamma$ let $\phi_a$ denote the left translation on $\Gamma$ by $a$. 
Consider the subgroup $\Gamma_L := \{ \phi_a \mid a \in \Gamma \} < {\rm Diff}^\infty(\Gamma)$. 
Then, for each $a \in \Gamma$ we have the mapping torus $\pi_{\phi_a} : M_{\phi_a} \to S^1$ associated to $\phi_a \in \Gamma_L$, 
which is a principal $\Gamma$ bundle (i.e, a $(\Gamma, \Gamma_L)$-bundle). 

Let $r \in \IZ_{\geq 0} \cup \{ \infty \}$. Below we are concerned with the following sets : \\
\hsp ${\cal P}^r(\Gamma)$ : the set of $C^r$ paths $\gamma : I \to \Gamma$, \hsh 
${\cal P}^r(\Gamma)_{a,b} = \{ \gamma \in {\cal P}^r(\Gamma) \mid \gamma(0) = a, \gamma(1) = b \}$ \ \ $(a, b \in \Gamma)$ \\[0mm] 
\hsp 
\btab[t]{@{}l}
$C^r(I, {\cal P}^r(\Gamma))$ : the set of $C^r$ families of paths, $\eta = (\eta_s)_{s \in I}$ \ ($\eta_s \in {\cal P}^r(\Gamma)$) \\[2mm] 
$C^r(I, {\cal P}^r(\Gamma); a,b) = \{ \eta \in C^r(I, {\cal P}^r(\Gamma)) \mid b\eta_1 = \eta_0 a \}$ \hsh $(a,b \in \Gamma)$ \\[2mm] 
$C^r(I, {\rm Isot}^r_{\Gamma_L}(\Gamma))$ : the set of $C^r$ families of isotopies, $H = (H_s)_{s \in I}$ \ 
($H_s \in {\rm Isot}^r_{\Gamma_L}(\Gamma)$) \\[2mm] 
$C^r(I, {\rm Isot}^r_{\Gamma_L}(\Gamma); \psi, \chi) = \{ H \in C^r(I, {\rm Isot}^r_{\Gamma_L}(\Gamma)) \mid \chi H_1 = H_0 \psi \}$ 
\hsh $(\psi, \chi \in \Gamma_L)$ 
\etab \\[2mm] 
When $a = b$ or $\psi = \chi$, we omit the symbol $b$ or $\chi$ from the notations $C^r(I, {\cal P}^r(\Gamma); a,b)$ or 
$C^r(I, {\rm Isot}^r_{\Gamma_L}(\Gamma); \psi, \chi)$. 
The meaning of similar notations are understood without any ambiguity. 

Fact~\ref{fact_mapping-torus-2}\,(2) implies the following conclusion for conjugate elements in $\Gamma$. 

\bfact\label{fact_attaching_map_equiv-0} \mbox{}
If $a, b, c \in \Gamma$ and $b = c\hspace*{0.3mm}ac^{-1}$, then we have the map \\
\hspppp $\widehat{\phi_c} \in {\rm Diff}^r(\pi_{\phi_a}, \pi_{\phi_b}; \id_{S^1})$ : $\widehat{\phi_c}([x,s]) = [cx,s]$. 
\efact

Equivariant isotopies on $\Gamma$ are reduced to paths in $\Gamma${.}

\bfact\label{fact_attaching_map_equiv-1} \mbox{}
\benum
\item There exists a bijection \ \ 
$\theta : {\cal P}^r(\Gamma) \ni \gamma \lmt \phi_\gamma := (\phi_{\gamma(t)})_{t \in I} \in {\rm Isot}^r_{\Gamma_L}(\Gamma) (= C^r(I, \Gamma_L))$. 
\bit 
\itemI Its inverse is defined by \ \ $\psi \lmt \gamma$ : $\gamma(t) = \psi_t(e)$ $(t \in I)$. 
\itemII It is restricted to the bijection \ \ 
$\theta : {\cal P}^r(\Gamma)_{a, b} \lra {\rm Isot}^r_{\Gamma_L}(\Gamma)_{\phi_a, \phi_b}$ \ \ for $a,b \in \Gamma$. 

\itemiii 
If ${\cal P}^r(\Gamma)_{a, b} \neq \emptyset$, then ${\rm Isot}^r_{\Gamma_L}(\Gamma)_{\phi_a, \phi_b} \neq \emptyset$. 
Hence, by Fact~\ref{fact_mapping-torus-2}\,(1) ${\rm Diff}^r(\pi_{\phi_a}, \pi_{\phi_b}; \id_{S^1}) \neq \emptyset$. 
In particular, if $a \in \Gamma_e$, then $\pi_{\phi_a}$ is $C^r$ isomorphic to 
the product bundle $\pi_{\id_\Gamma} : M_{\id_\Gamma} \to S^1$ as $(\Gamma, \Gamma_L)$-bundles. 
\eit  

\item There exists a bijection \ \ 
$\Theta : C^r(I, {\cal P}^r(\Gamma)) \ni \eta = (\eta_s)_{s \in I} \lmt \Theta_\eta := (\phi_{\eta_s})_{s \in I} \in C^r(I, {\rm Isot}^r_{\Gamma_L}(\Gamma))$. \bit 
\itemI Its inverse is given by \ \ $H = (H_s)_{s \in I} \lmt \eta = (\eta_s)_{s \in I}$ : $\eta_s(t) = (H_s)_t(e)${.} 
\itemII It is restricted to the bijection \ \ $\Theta : C^r(I, {\cal P}^r(\Gamma); a,b) \lra C^r(I, {\rm Isot}^r_{\Gamma_L}(\Gamma); \phi_a, \phi_b)$ \ \ for $a,b \in \Gamma$.   
\eit 
\eenum
\efact 

Below we fix $a \in \Gamma$ and consider the mapping torus $\pi_{\phi_a} : M_{\phi_a} \to S^1$ associated to $\phi_a \in \Gamma_L$, 
which is a principal $\Gamma$ bundle (i.e, a $(\Gamma, \Gamma_L)$-bundle). 
Let $r \in \IZ_{\geq 0} \cup \{ \infty \}$ and $k := k(\pi_{\phi_a},r)$.

\bfact\label{fact_attaching_map_equiv-2} For each $\ell \in \IZ$, the bijections $\theta$ and 
$\Theta$ in Fact~\ref{fact_attaching_map_equiv-1} are restricted to the bijections : 
\benum
\item $\theta : {\cal P}^r(\Gamma)_{e, a^\ell} \lra {\rm Isot}^r_{\Gamma_L}(\Gamma)_{\id, \phi_{a^\ell}}$. 
\item $\Theta : C^r(I, {\cal P}^r(\Gamma)_{e, a^\ell}; a) \lra C^r(I, {\rm Isot}^r_{\Gamma_L}(\Gamma)_{\id, \phi_{a^\ell}}; \phi_a)$. 
\eenum
\efact

\bprop\label{prop_attaching_map_equiv} {\rm (the case of principal $\Gamma$ bundles)} \\[1mm] 
\hsh For $\ell \in \IZ$ \hsh 
$\ell \in k\IZ$ \LLRA 
\btab[t]{@{}l}
There is a path $\gamma \in {\cal P}^r(\Gamma)_{e,a^\ell}$ which admits \\[2mm] 
\hsp a $C^r$ path-homotopy $\eta : \gamma \simeq_\ast a^{-1} \gamma a$ in $\Gamma$ relative to ends. 
\etab 
\vskip 2mm 
\bit 
\itemI ${\cal P}^r(\Gamma)_{e,a^k} \neq \emptyset$ \ $($i.e., $a^k \in \Gamma_e$$)$.
\itemII $\ell|k$ \ if $\ell \in \IZ_{\geq 1}$ satisfies $a^\ell \in \Gamma_e$ and $a^i \notin \Gamma_e$ $(i = 1, 2,\cdots, \ell-1)$. 

\itemiii $k \geq 1$ \LLRA 
\btab[t]{@{}l}
There is $\ell \in \IZ_{\geq 1}$ for which there is a path $\gamma \in {\cal P}^r(\Gamma)_{e,a^\ell}$ \\[2mm] 
\hsp which admits a $C^r$ path-homotopy $\gamma \simeq_\ast a^{-1} \gamma a$ in $\Gamma$ relative to ends. 
\etab 
\eit 
\eprop

\bpfb 
It follows that \ \ 
$\ell \in k\IZ$ 
\btab[t]{@{\ }ll} 
\LLRA $C^r(I, {\rm Isot}^r_{\Gamma_L}(\Gamma)_{\id, \phi_{a^\ell}}; \phi_a) \neq \emptyset$ \ 
& by Proposition~\ref{prop_attaching_map}\,(2) \\[2mm] 
\LLRA $C^r(I, {\cal P}^r(\Gamma)_{e, a^\ell}; a) \neq \emptyset$ 
& by the bijection $\Theta$ in Fact~\ref{fact_attaching_map_equiv-2}. 
\etab 
\vskip 2mm 
For each $\eta \in C^r(I, {\cal P}^r(\Gamma)_{e, a^\ell}; a)$ we have 
$\gamma := \eta_0 \in {\cal P}^r(\Gamma)_{e, a^\ell}$ and $\eta : \gamma \simeq_\ast a^{-1}\gamma a$. 
\epf 

%

\bexp\label{exp_equiv-1} \mbox{} The integer $k \equiv k(\pi_{\phi_a}, r)$ has the following properties. 
\benum 
\item 
\bit 
\itemI If $b \in \Gamma$ is conjugate to $a$ (i.e., $b = c^{-1}ac$ for some $c \in \Gamma$), then $k = k(\pi_{\phi_b},r)$. 
\itemII If $b \in \Gamma$ and ${\cal P}^r(\Gamma)_{a, b} \neq \emptyset$ (or $b \in a\Gamma_e$), then 
$k = k(\pi_{\phi_b}, r)$. 
\itemiii $k = 1$ iff $a \in \Gamma_e$. In particular, if $\Gamma$ is connected, then $k=1$.
\eit 

\item Suppose $\ell \in \IZ_{\geq 1}$ and $a^\ell \in \Gamma_e$. 
Then $\ell \in k\IZ$ and $k \geq 1$, if one of the following conditions holds. 
\bit 
\itemI $\Gamma_e$ is simply  connected{.} 
\itemII {There exists $\gamma \in {\cal P}^r(\Gamma)_{e,a^\ell}$ with $\gamma(t) \in Z(a)$ $(t \in I)$, 
where $Z(a)$ denotes the centralizer of $a$ in $\Gamma$. 
For example, this holds if (a) $a^\ell = e$ or (b) $\Gamma_e \cap Z(a)$ is $C^r$ path connected.} 
\eit 

\item If $a^\ell \not\in \Gamma_e$ for any $\ell \in \IZ_{\geq 1}$, then $k = 0$. 
\eenum 
\eexp 

\bpfb
\benum
\item The assertion (i) follows from Fact~\ref{fact_attaching_map_equiv-0}, since $\phi_b \phi_c = \phi_c \phi_a$. \\
From Fact~\ref{fact_attaching_map_equiv-1}\,(1)(iii) it follows that 
${\rm Diff}^r(\pi_{\phi_a}, \pi_{\phi_b}; \id_{S^1}) \neq \emptyset$ in the case (ii) and 
$\pi_{\phi_a}$ is a trivial $C^\infty$ $(\Gamma, \Gamma_L)$-bundle in the case (iii). 
\item The assertion follows from Proposition~\ref{prop_attaching_map_equiv} and the following observations in each case (i) and (ii). 
\bit 
\itemI Take any path $\gamma \in {\cal P}^r(\Gamma_e)_{e, a^\ell}$. 
Since $a^{-1}\gamma a \in {\cal P}^r(\Gamma_e)_{e, a^\ell}$ and $\Gamma_e$ is simply  connected, we have a $C^r$ homotopy $\gamma \simeq_\ast a^{-1}\gamma a$ in $\Gamma_e$. 
 
\itemII The assumption implies $a^{-1}\gamma a = \gamma$.  
If $a^\ell = e$, then we can take $\gamma = \e_e$ $($the constant path at $e$$)$.
\eit 

\item Note that $a^k \in \Gamma_e$ by Proposition~\ref{prop_attaching_map_equiv}\,(i). 
\eenum 
\vskip -7mm 
\epf

\bexp\label{exp_equiv-2} 
Suppose $f : \Gamma \to \Lambda$ is a Lie group homomorphism between Lie groups. 
To any $a \in \Gamma$ and its image $b \equiv f(a) \in \Lambda$ we have $k := k(\pi_{\phi_a}, r)$ and $\ell := k(\pi_{\phi_b}, r)$. 
Proposition~\ref{prop_attaching_map_equiv} yields the following conclusions. 
\benum 
\item $k \in \ell\IZ$
\item $k = \ell$ if $f$ is surjective and ${\rm Ker}\,f$ is connected and simply connected. 
\eenum 
\eexp 

\bpfb
\benum 
\item There is a path $\gamma \in {\cal P}^r(\Gamma)_{e,a^k}$ and 
a $C^r$ homotopy $\eta : \gamma \simeq_\ast a^{-1} \gamma a$ in $\Gamma$, 
which induces the path $f\gamma \in {\cal P}^r(\Lambda)_{e,b^k}$ and 
the $C^r$ homotopy $f\eta : f\gamma \simeq_\ast b^{-1} (f\gamma) b$ in $\Lambda$. 

\item There is a path $\delta \in {\cal P}^r(\Lambda)_{e,b^\ell}$ and 
a $C^r$ homotopy $\xi : \delta \simeq_\ast b^{-1} \delta b$ in $\Lambda$.
Since $f$ is {a principal} $({\rm Ker}\,f)$\,-\,bundle, by the relative homotopy lifting property of $f$ we can find 
a path $\beta \in {\cal P}^r(\Gamma)$ and 
a $C^r$ path homotopy $\rho$ in $\Gamma$ such that 
\bit 
\item[] \ \ (i) \ $f \beta = \delta$, $\beta(1) = a^\ell$ \ \ and \ \ 
(ii) \,$f\rho = \xi$, $\rho_0 = \beta$, $\rho_1 = a^{-1}\beta a$ \ and \ $\rho(1, \ast) \equiv a^\ell$. 
\eit 
Since $\xi(0, \ast) \equiv e$ and $f \rho = \xi$, we have the path $c := \rho(0,\ast) \in {\cal P}^r({\rm Ker}\,f)$. 
Note that $c(1) = a^{-1}c(0)a$. 
Since ${\rm Ker}\,f$ is connected, there is a path $\alpha \in {\cal P}^r({\rm Ker}\,f)_{e, c(0)}$.
It yields two paths $\alpha \ast c, a^{-1} \alpha a \in {\cal P}^r({\rm Ker}\,f)_{e, c(1)}$. 
Since ${\rm Ker}\,f$ is simply connected, there is a $C^r$ homotopy $\zeta : \alpha \simeq a^{-1} \alpha a$ in ${\rm Ker}\,f$ such that 
$\zeta(0, \ast) \equiv e$ and $\zeta(1, \ast) = c$. 
Then, we obtain the path $\gamma = \alpha \ast \beta \in {\cal P}^r(\Gamma)_{e, a^\ell}$ and 
the $C^r$ homotopy $\eta = (\zeta_t \ast \rho_t)_{t \in I} : \gamma \simeq_\ast a^{-1}\gamma a$ in $\Gamma$. 
This implies $\ell \in k\IZ$ and so $k = \ell$ by (1). 
\eenum 
\vskip -3mm 
\epf 

\bexp\label{exp_equiv-3} Consider the product $\Gamma \times \Lambda$ of Lie groups $\Gamma$ and $\Lambda$. 
Note that $\iota(\Gamma_L \times \Lambda_L) = (\Gamma \times \Lambda)_L$, 
since $\phi_{(a,b)} = \phi_a \times \phi_b$ for $(a,b) \in \Gamma \times \Lambda$.  
From Proposition~\ref{prop_fiber_products} it follows that 
\benum  
\item $k(\pi_{\phi_a \times \phi_b},r) \geq 1$ iff $k(\pi_{\phi_a},r) \geq 1$ and $k(\pi_{\phi_b},r) \geq 1$. \\
In this case, \ $k(\pi_{\phi_a \times \phi_b},r) = {\rm lcm}(k(\pi_{\phi_a},r), k(\pi_{\phi_b},r))$. 
\item $k(\pi_{\phi_a \times \phi_b},r) = 0$ iff $k(\pi_{\phi_a},r) = 0$ or $k(\pi_{\phi_b},r) = 0$. 
\eenum 
\eexp

\bexp\label{exp_equiv-4} \mbox{} 
\benum 
\item $k = 1$ for any connected Lie groups \\
(for examples, $GL(n, \IC)$, $SL(n, \IC)$, $SL(n, \IR)$, $U(n)$, $SU(n)$, $SO(n)$, $\IR^n$, $T^n = \IR^n/\IZ^n$, etc.). 
\item $k = 2$ for $\Gamma = GL(n, \IR)$, $O(n)$ and $a \in \Gamma^- := \{ c \in \Gamma \mid \det c < 0 \}$. \\
Note that there exists $c \in \Gamma^-$ with $c^2 = e$ and that 
$\Gamma^- = c \Gamma_e$ and 
$\Gamma = \Gamma_e \cup c \Gamma_e$. 

\item If $\Gamma$ is commutative and ${\rm order}\,(a\Gamma_e, \Gamma/\Gamma_e) < \infty$, 
then $k =  {\rm order}\,(a\Gamma_e, \Gamma/\Gamma_e)$. 
\item If $\Gamma$ is a finite group, then $k = {\rm order}\,(a, \Gamma)$.  
\item $k = 0$ for $\Gamma = \IZ$ and $a \in \IZ - \{ 0 \}$. 
\eenum
\eexp

\subsection{Principal bundle $P : {\rm Diff}_\pi^r(M)_0 \to {\rm Diff}^r(S^1)_0$} \mbox{} 

This subsection includes some comments 
when 
the diffeomorphism groups are endowed with the Whitney $C^r$ topology.   
Some topological arguments on a principal bundle $P : {\rm Diff}_\pi^r(M)_0 \to {\rm Diff}^r(S^1)_0$ provide us with 
some important  insights to statements 
described for the discrete groups of diffeomorphisms or isotopies in the previous subsections. 

Suppose $\pi : M \to B$ is a fiber bundle with fiber $N$ and structure group $\Gamma < {\rm Diff}(N)$. Let $r \in \IZ_{\geq 0} \cup \{ \infty \}$.
Below we assume that $M$ is a closed manifold and we endow 
any subgroups of ${\rm Diff}^r(M)$ and ${\rm Diff}^r(B)$ with the subspace topology of the Whitney $C^r$ topology. 

\bfact\label{fact_p-bdle} \mbox{}
\benum
\item ${\rm Diff}_\pi^r(M)_0$ and ${\rm Diff}^r(B)_0$ are topological groups 
and the map $P : {\rm Diff}_\pi^r(M)_0 \to {\rm Diff}^r(B)_0$ is a continuous surjective group homomorphism. 

\item The map $P$ is a topological principal bundle with structure group ${\rm Ker}\,P$. 
\eenum
\efact

\bpf (2) 
We fix the following data for $\pi$ : \ (a) a finite $\Gamma$-atlas $\{ (U_i, \phi_i) \}_{i=1, \cdots, \ell}$ of $\pi$, \\  
\hspace*{20mm} (b) $K_i, L_i \in {\cal K}(B)$ $(i=1, \cdots, \ell)$ with 
$\cup_{i=1}^\ell K_i = B$, $K_i \Subset L_i \subset U_i$. 

By (1) it suffices to show that the map $P$ admits a continuous local section on a neighborhood of $\id_B$ in ${\rm Diff}^r(B)_0$. 
We can repeat local deformation of $C^r$ diffeomorphisms 
to find a small $C^s$ neighborhood ${\cal U}$ of $\id_B$ ($s = 0$ for $r = 0$ and $s = 1$ for $r \geq 1$) such that 
for any $f \in {\cal U}$ admits a canonical factorization $f = f_1 \cdots f_\ell$ 
such that 
$f_i \in {\rm Diff}^r(B; B_{L_i})_0$ $(i = 1, \cdots, \ell)$ and $f_i$ varies continuously with $f$ in the $C^r$ topology. 
Then, each trivialization $\phi_i$ of $\pi$ over $U_i$ yields a canonical lift $\widetilde{f}_i \in {\rm Diff}^r_{\pi}(M; B_{L_i})_0$ of $f_i$ and 
we obtain a canonical lift $s(f) := \widetilde{f}_1 \cdots \widetilde{f}_\ell \in {\rm Diff}_\pi^r(M)_0$ of $f$. 
From these construction the map $s : {\cal U} \to {\rm Diff}_\pi^r(M)_0$ forms a continuous local section of $P$. 
\epf 

We are concerned with the following condition. 

\begin{assumption_sharp}\label{assump_P_pbdle} 
$({\rm Ker}\,P)_0$ is an open subgroup of ${\rm Ker}\,P$ (under the $C^r$ topology). 
\end{assumption_sharp}

{For example, this assumption holds 
when $\pi$ is a principal bundle (cf. \cite{Cu})
or a locally trivial bundle ($\Gamma = {\rm Diff}(N)$). 
In fact, using a finite $\Gamma$-atlas of $\pi$, 
it is seen that if $f \in {\rm Ker}\,P$ is sufficiently $C^s$ close to $\id_M$ ($s = 0$ for $r = 0$ and $s = 1$ for $r \geq 1$), then 
there exists $F \in {\rm Isot}^r_\pi(M)_0$ with $\underline{F} = \id$ and $F_1 = f$. 
Under the Assumption $(\sharp)$ $({\rm Ker}\,P)_0$ coincides with the identity connected component of ${\rm Ker}\,P$. 

From now on, we restrict ourselves to the case that $B = S^1$ and let $k \equiv k(\pi, r)$. Suppose Assumption $(\sharp)$ holds. 
Then, taking the quotients by $({\rm Ker}\,P)_0$, we obtain 
a discrete group ${\cal G} := ({\rm Ker}\,P)\big/({\rm Ker}\,P)_0$ and 
a principal ${\cal G}$ bundle $\widetilde{P} : {\rm Diff}_\pi^r(M)_0/({\rm Ker}\,P)_0 \to {\rm Diff}^r(S^1)_0$. 
It follows that $\widetilde{P}$ is a connected regular covering projection and 
${\cal G}$ coincides with the covering transformation group of $\widetilde{P}$.  
Therefore, the connecting homomorphism $\delta$ in the homotopy exact sequence induces 
an isomorphism 
$\widetilde{\delta} : \big(\pi_1{\rm Diff}^r(S^1)_0\big) \big/{\rm Im}\,\widetilde{P}_\ast \cong {\cal G}$, where 
$\widetilde{P}_\ast : \pi_1 \big({\rm Diff}_\pi^r(M)_0/({\rm Ker}\,P)_0\big) \to \pi_1 {\rm Diff}^r(S^1)_0$. 

The groups of isotopies and the fundamental groups 
of diffeomorphism groups 
are identified as follows: \\
\hspp 
$\pi_1 {\rm Diff}^r(S^1)_0 = {\rm Isot}^r(S^1)_{\id, \id}\big/\simeq_\ast$, \ \ 
$\pi_1 {\rm Diff}^r_\pi(M)_0 = {\rm Isot}^r_\pi(M)_{\id, \id}\big/\simeq_\ast$ \ \ and \\
\hsppp $\pi_1 \big({\rm Diff}_\pi^r(M)_0/({\rm Ker}\,P)_0\big)
= C^0(I, \{ 0,1 \}; {\rm Diff}_\pi^r(M)_0/({\rm Ker}\,P)_0, [\id])\big/\simeq_\ast$ \ \ etc.

Since $({\rm Ker}\,P)_0$ is path-connected, 
the principal $({\rm Ker}\,P)_0$ bundle ${\rm Diff}_\pi^r(M)_0 \to  {\rm Diff}_\pi^r(M)_0/({\rm Ker}\,P)_0$ induces 
a surjective group homomorphism \ \ $\pi_1 {\rm Diff}_\pi^r(M)_0 \to \pi_1 \big({\rm Diff}_\pi^r(M)_0\big/({\rm Ker}\,P)_0\big)$. 
Hence, we have ${\rm Im}\,\widetilde{P}_\ast = {\rm Im}\,P_\ast$. 

Since $F \simeq_\ast G$ in ${\rm Isot}^r_\pi(M)$ implies $\underline{F} \simeq_\ast \underline{G}$ in ${\rm Isot}^r(S^1)$, 
the map $P_I : {\rm Isot}^r_\pi(M)_{\id, \id} \to {\rm Isot}^r(S^1)_{\id, \id}$
induces 
the associated map \hsh $\widetilde{P_I} : \big({\rm Isot}^r_\pi(M)_{\id, \id}\big/\!\simeq_\ast\!\big) \to \big({\rm Isot}^r(S^1)_{\id, \id}\big/\!\simeq_\ast\!\big)$, \\
which corresponds to the induced map \ \ $P_\ast : \pi_1 {\rm Diff}_\pi^r(M)_0 \to \pi_1 {\rm Diff}^r(S^1)_0$. \ \ 
Hence, it follows that \\ 
\hspp ${\rm Im}\,P_\ast = {\rm Im}\,\widetilde{P_I}$ \ \ and \ \  
$\big(\pi_1 {\rm Diff}^r(S^1)_0\big)\big/ {\rm Im}\,P_\ast \ = \ \big({\rm Isot}^r(S^1)_{\id, \id}\big/\simeq_\ast\!\big)\big/{\rm Im}\,\widetilde{P_I}$. 

Note that any $z \in \pi_1{\rm Diff}^r(S^1)_0 = 
{\rm Isot}^r(S^1)_{\id, \id}\big/\!\!\simeq_\ast$
is represented as $z = [\underline{F}]$ for some 
$F \in  {\cal I} \subset {\rm Isot}^r_\pi(M)_0$. 
Then the isomorphism $\widetilde{\delta} : \big(\pi_1{\rm Diff}^r(S^1)_0\big) \big/{\rm Im}\,\widetilde{P}_\ast \cong {\cal G}$ is 
described by \\
\hspp $\widetilde{\delta}\big(\big[[\underline{F}]\big]\big) = [F_1]$ for $F \in {\cal I} \subset {\rm Isot}^r_\pi(M)_0$.

By Fact~\ref{fact_mu=0}\,(2) 
the map $\mu$ induces a group isomorphism $\widetilde{\mu} : \big({\rm Isot}^r(S^1)_{\id,\id}\big/\!\simeq_\ast\!\big) \cong \IZ$. 
From the definition of $\mu$ and $\nu$ it follows that 
$\widetilde{\mu}({\rm Im}\,\widetilde{P_I}) = k\IZ$. 
Hence, we have a group isomorphism \\[0.5mm]
\hspp $\bary[b]{@{}c@{}}
\mbox{\footnotesize $\approx$} \\[-2.5mm]
\mu
\eary : \fbox{\makebox(7,7){$\ast$}} \ \cong \ \IZ_k$, \hsh 
where \ \ {\small $\fbox{\makebox(7,7){$\ast$}} = \big(\pi_1 {\rm Diff}^r(S^1)_0\big)\big/{\rm Im}\,P_\ast \ 
= \ \big({\rm Isot}^r(S^1)_{\id, \id}\big/\simeq_\ast\!\big)\big/{\rm Im}\,\widetilde{P_I}$.}
\vskip 4mm 


\hsh 
{\small $\bary[t]{clcl}
\pi_1 {\rm Diff}_\pi^r(M)_0 & = & {\rm Isot}^r_\pi(M)_{\id, \id}\big/\simeq_\ast & \\
\raisebox{2mm}{\rotatebox{-90}{\makebox(22,0){\rightarrowfill}}} 
& \hspace*{-25mm} \smash{\raisebox{-3mm}{$P_\ast$}}
& \hspace*{-10mm} \raisebox{2mm}{\rotatebox{-90}{\makebox(22,0){\rightarrowfill}}} 
& \hspace*{-27mm} \smash{\raisebox{-3.5mm}{$\widetilde{P_I}$}} \\[6mm] 
\pi_1 {\rm Diff}^r(S^1)_0 & = & {\rm Isot}^r(S^1)_{\id, \id}\big/\simeq_\ast & \\[0mm]
\raisebox{2mm}{\rotatebox{-90}{\makebox(22,0){\rightarrowfill}}} & 
& \hspace*{-10mm} \raisebox{2mm}{\rotatebox{-90}{\makebox(22,0){\rightarrowfill}}} & \\[5mm] 
\big(\pi_1 {\rm Diff}^r(S^1)_0\big)\big/ {\rm Im}\,P_\ast & = & \big({\rm Isot}^r(S^1)_{\id, \id}\big/\simeq_\ast\!\big)\big/{\rm Im}\,\widetilde{P_I} & 
\eary$}
\hsh 
{
\smash{\raisebox{-15mm}{\rotatebox{0}{$\nu$}}} \hspace*{-3mm} 
\smash{\raisebox{0mm}{\rotatebox{-80}{
$\xymatrix{
\makebox(0,0){} \ar@/_3ex/[rr]^{} & \makebox(10,0){} & \makebox(0,0){} 
}$ 
}}} \hspace*{-10mm} 
{\small $\bary[t]{c@{ \ \ }c@{ \ \ }cl}
{\rm Isot}^r_\pi(M)_{\id, \id} & \raisebox{1mm}{\makebox(22,0){\rightarrowfill}} & {\rm Isot}^r_\pi(M)_{\id, \id}\big/\simeq_\ast & \\
& \hspace*{-17.5mm} \smash{\raisebox{-3mm}{$\widetilde{P_I}$}}
& \hspace*{-10mm} \raisebox{2mm}{\rotatebox{-90}{\makebox(22,0){\rightarrowfill}}} 
& \hspace*{-27mm} \smash{\raisebox{-3mm}{$\widetilde{P_I}$}} \\[6mm] 
{\rm Im}\,\widetilde{P_I} & \smash{\raisebox{0mm}{\rotatebox{90}{$\bigcap$}}} & {\rm Isot}^r(S^1)_{\id, \id}\big/\simeq_\ast 
& \hspace*{-48mm} \smash{\raisebox{13.5mm}{\rotatebox{-150}{\makebox(52,0){\rightarrowfill}}}} \\[0mm]
\raisebox{2mm}{\rotatebox{-90}{\makebox(22,0){\rightarrowfill}}} 
& \hspace*{-31mm} \smash{\raisebox{-3mm}{$\widetilde{\mu}$}} \hspace*{4mm} \smash{\raisebox{-3mm}{\footnotesize $\cong$}}
& \hspace*{-10mm} \raisebox{2mm}{\rotatebox{-90}{\makebox(22,0){\rightarrowfill}}} 
& \hspace*{-26mm} \smash{\raisebox{-3mm}{$\widetilde{\mu}$}} \hspace*{4mm} \smash{\raisebox{-3mm}{\footnotesize $\cong$}}\\[5mm] 
k\IZ & \smash{\raisebox{0mm}{\rotatebox{90}{$\bigcap$}}} & \hspace*{-10mm} \IZ & 
\eary$}
}
\vskip 4mm 
Combining the isomorphisms $\widetilde{\delta}$ and 
$\bary[b]{@{}c@{}}
\mbox{\footnotesize $\approx$} \\[-2.5mm]
\mu
\eary$, we obtain the following diagram, from which it is seen that 
the isomorphism 
$(\widehat{\nu}|_{{\rm Ker}\,P})^\sim : {\cal G} 
\cong \IZ_k$ in Fact~\ref{lem_Ker_P}\,(1)(ii)  
coincides with 
the composition 
$\bary[b]{@{ \ }c@{\,}}
\mbox{\footnotesize $\approx$} \\[-2.5mm]
\mu
\eary \big(\hspace*{0.2mm}\widetilde{\delta}\hspace*{0.3mm}\big)^{-1}$. 
\vskip 2mm 
\hspp 
$\bary[t]{ccll}
{\rm Ker}\,P & \makebox(30,0){\rightarrowfill} & {\cal G} \hspace*{-16mm} \smash{\raisebox{-18mm}{$\widehat{\nu}$}} & \\[0.5mm]
\rotatebox{90}{$\lra\hspace{-5mm} \lra$} & \hspace*{-35mm} \smash{\raisebox{2mm}{$R$}} 
& \hspace{2mm} \smash{\raisebox{-0.5mm}{\rotatebox{90}{\makebox(22,0){\rightarrowfill}}}} 
\hspace*{-4mm} \smash{\raisebox{2mm}{\footnotesize $\cong$}} \hspace*{3mm} \smash{\raisebox{2mm}{\footnotesize $\widetilde{\delta}$}} & \\[-0.5mm] 
{\cal I} & & \hspace*{-0.5mm} \fbox{\makebox(5,5){$\ast$}} & \\[0mm]
\rotatebox{90}{$\lla\hspace{-5mm} \lla$} & \hspace*{-35mm} \smash{\raisebox{2mm}{$\nu$}} 
& \hspace{1.5mm} \smash{\raisebox{7mm}{\rotatebox{-90}{\makebox(22,0){\rightarrowfill}}}} 
\hspace*{-4mm} \smash{\raisebox{2mm}{\footnotesize $\cong$}}
\hspace*{1mm} 
\smash{\raisebox{4mm}{
\bary[t]{c}
\mbox{\footnotesize $\approx$} \\[-2.5mm]
\mu
\eary}} & \\
\IZ & \hspace*{-5mm} \smash{\raisebox{1mm}{\makebox(40,0){\rightarrowfill}}} 
& \IZ_k \hspace{-25mm} \smash{\raisebox{25mm}{\rotatebox{-50}{\makebox(80,0){\rightarrowfill}}}} & 
\eary$

\end{document}